\documentclass[12pt]{paper}
\usepackage{amssymb,amsfonts,amsthm,amsmath}
\usepackage{bbm}
\usepackage{hyperref}
\usepackage{cleveref}

\hypersetup{
    colorlinks,
    linkcolor={blue},
    citecolor={blue},
    urlcolor={black},
}

\usepackage{eucal}
\usepackage{mathrsfs}
\usepackage[protrusion=true,expansion=true,auto=true,tracking=true]{microtype}

\usepackage[a4paper, hmargin=3.2cm, vmargin={3.8cm, 3.5cm}]{geometry}
\usepackage{titlesec}
\numberwithin{equation}{section}
\usepackage[initials, msc-links]{amsrefs}
\usepackage{fancyhdr}
\usepackage{enumitem}
\newlist{notation}{description}{1}

\usepackage{mathtools}
\mathtoolsset{mathic=true}

\usepackage{tikz}
\usetikzlibrary{3d,perspective}
\usetikzlibrary{cd}
\tikzset{invclip/.style={clip,insert path={{[reset cm]
      (-16383.99999pt,-16383.99999pt) rectangle (16383.99999pt,16383.99999pt)
    }}}}

\theoremstyle{plain}
\newtheorem{theorem}{Theorem}[section]
\newtheorem{corollary}[theorem]{Corollary}
\newtheorem{lemma}[theorem]{Lemma}
\newtheorem{proposition}[theorem]{Proposition}
\newtheorem{innercustomgeneric}{\customgenericname}
\providecommand{\customgenericname}{}
\newcommand{\newcustomtheorem}[2]{%
  \newenvironment{#1}[1]
  {%
   \ifdefined\crefalias\crefalias{innercustomgeneric}{#2}\fi
   \renewcommand\customgenericname{#2}%
   \renewcommand\theinnercustomgeneric{##1}%
   \innercustomgeneric
  }
  {\endinnercustomgeneric}%
  \ifdefined\crefname\crefname{#2}{#2}{#2s}\fi
}
\newcustomtheorem{customthm}{Theorem}
\newcustomtheorem{customcor}{Corollary}

\theoremstyle{definition}
\newtheorem{definition}[theorem]{Definition}
\newtheorem{exmpl}[theorem]{Example}
\newtheorem{remark}[theorem]{Remark}

\newlist{notation-list}{description}{1}
\setlist[notation-list]{%
  labelindent=0.6cm,
  labelwidth=2.5cm,
  labelsep=0.5cm,
  leftmargin=!,
  font=\normalfont,
  align=left,
  itemsep=0.5ex
}

\newlist{definitions}{description}{1}
\setlist[definitions]{%
  labelindent=0.6cm,
  labelwidth=5.7cm,
  labelsep=0.5cm,
  leftmargin=!,
  font=\normalfont,
  align=left,
  itemsep=0.5ex
}

\newcommand{\bbone}{\mathbbm{1}}
\newcommand{\C}{\mathbb{C}}
\newcommand{\Cbar}{\overline{\mathbb{C}}}
\newcommand{\D}{\mathbb{D}}
\renewcommand{\Dbar}{\overline{\mathbb{D}}}
\newcommand{\N}{\mathbb{N}}
\newcommand{\Q}{\mathbb{Q}}
\newcommand{\R}{\mathbb{R}}
\newcommand{\Rd}{\mathbb{R}^{d}}
\renewcommand{\S}{\mathbb{S}}
\newcommand{\T}{\mathbb{T}}
\newcommand{\Z}{\mathbb{Z}}
\newcommand{\Zd}{\mathbb{Z}^{d}}

\newcommand{\im}{\mathrm{i}} 
\renewcommand{\Im}{\mathrm{Im}}
\renewcommand{\Re}{\mathrm{Re}}

\newcommand{\mfC}{\mathfrak{C}}
\newcommand{\mfD}{\mathfrak{D}}

\newcommand{\mfR}{\mathfrak{R}}

\newcommand{\borel}{\mathcal{B}}
\newcommand{\effros}[1]{\mathcal{F}(#1)}
\newcommand{\effrosdis}[1]{\mathcal{F}_{\textrm{dis}}(#1)}
\newcommand{\vietoris}[2][]{\mathcal{K}_{#1}(#2)}
\newcommand{\vietorisfin}[1]{\mathcal{K}_{\textrm{fin}}(#1)}

\newcommand{\contcs}[1]{C_{c}(#1)}
\newcommand{\smooth}[2][]{C^{\infty}_{#1}(#2)}
\newcommand{\smoothcs}[1]{C^{\infty}_{c}(#1)}
\newcommand{\smoothCC}{\smooth{\C,\C}}

\newcommand{\meas}[2][]{\mathcal{M}_{#1}(#2)}
\newcommand{\measp}[2][]{\mathcal{M}_{#1}^+(#2)}
\newcommand{\distrib}[1]{\mathscr{D}'(#1)}

\newcommand{\merom}[1][]{\mathcal{MR}_{#1}}
\newcommand{\meromnz}{\merom^{\times}}
\newcommand{\holomnz}[1][]{\mathcal{E}^{\times}_{#1}}
\newcommand{\holom}[1][]{\mathcal{E}_{#1}}

\newcommand{\divis}[1][]{\mathcal{D}_{#1}}
\newcommand{\divisp}[1][]{\divis[#1]^{+}}
\newcommand{\ppart}{\mathcal{P}}

\newcommand{\harm}[2][]{\mathcal{H}_{#1}(#2)}
\newcommand{\sharm}[2][]{\mathcal{SH}_{#1}(#2)}

\newcommand{\lloc}{\mathrm{L^{1}_{\mathrm{loc}}}}

\newcommand{\dm}{\mathrm{div}}      
\newcommand{\ppm}{\mathrm{pp}}      

\newcommand{\unifm}{\mathsf{D}}
\newcommand{\unifn}{\mathsf{N}}

\renewcommand{\tilde}{\widetilde}

\DeclareMathOperator{\lprod}{\boxdot}
\DeclareMathOperator{\dom}{\mathrm{dom}}

\DeclareMathOperator{\ran}{\mathrm{ran}}
\DeclareMathOperator{\proj}{\mathrm{proj}}
\DeclareMathOperator{\stab}{\mathrm{Stab}}
\DeclareMathOperator{\inter}{\mathrm{int}}
\DeclareMathOperator*{\supp}{\mathrm{supp}}
\DeclareMathOperator*{\diam}{\mathrm{diam}}
\DeclareMathOperator{\hull}{\mathrm{hull}}
\DeclareMathOperator{\shull}{\mathrm{s-hull}}

\newcommand{\acts}{\curvearrowright}
\newcommand{\norm}[1]{||#1||}

\newcommand{\concat}{{}^{\frown}}
\newcommand{\free}{\mathrm{Free}}
\newcommand{\symd}{\vartriangle}
\newcommand{\dist}{\textrm{dist}}

\newcommand{\freeab}{\mathcal{A}}
\newcommand{\say}[1]{``#1''}
\renewcommand{\epsilon}{\varepsilon}
\renewcommand{\phi}{\varphi}

\begin{document}

\title{ \vspace{-2.2cm} Equivariant Borel liftings in complex analysis and PDE}

\author{Konstantin Slutsky \and Mikhail Sodin \and Aron Wennman}

\maketitle

\begin{abstract}
  We establish Borel equivariant analogues of several classical theorems from complex analysis and
  PDE.  The starting point is an equivariant Weierstrass theorem for entire functions: there exists
  a Borel mapping which assigns to each non-periodic positive divisor \(d\) an entire function
  \(f_d\) with divisor of zeros \(\dm(f_d)=d\) and which commutes with translation,
  \(f_{d-w}(z)=f_d(z+w)\).  We also examine the existence of equivariant Borel right inverses for
  the distributional Laplacian, the heat operator, and the \(\bar{\partial}\)-operator on the space
  of smooth functions.  We demonstrate that Borel equivariant inverses for these maps exist on the
  free part of the range. In general, the freeness assumptions cannot be
  omitted and Borelness cannot be strengthened to continuity.  Our positive results follow from a
  theorem establishing sufficient conditions for the existence of equivariant Borel liftings.  Two
  key ingredients are Runge-type approximation theorems and the existence of Borel toasts, which are
  Borel analogues of Rokhlin towers from ergodic theory.
\end{abstract}

\tableofcontents

\section{Introduction}
\label{sec:introduction}
\subsection{Equivariant analysis and Borel entire functions}
\label{sec:what-is-equivariant-analysis}
This paper contributes to the field of equivariant analysis, which dates back to the notable 1997
paper~\cite{MR1422707} of Benjamin Weiss. His work was motivated by a question posed by George
Mackey, who asked whether the space \(\holom\) of entire functions admits any non-trivial
translation-invariant probability measures.  Weiss answered this question in the affirmative by
introducing the concept of a measurable entire function.  Specifically, given a free probability
measure-preserving (p.m.p.\ for short) action \(\C \acts X\), \((z,x)\mapsto z\cdot x\), of the
additive group of complex numbers on a standard probability space \((X, \mu)\), a measurable entire
function on \(X\) is a measurable map \(F : X \to \C\) such that, for \(\mu\)-almost every
\(x \in X\), the function \(F_{x}\), defined by \(\C\ni z \mapsto F(z\cdot x)\in \C\), is entire.
Weiss showed that non-constant measurable entire functions exist for all p.m.p.\ actions
\(\C \acts X\).  The push-forward of \(\mu\) along the map \(x \mapsto F_{x}\) produces the desired
non-trivial invariant measure on \(\holom\). Very recently, Glücksam and Weiss \cite{GlucksamWeiss}
extended this result to free actions of \(\C^d\). A related phenomenon was studied by
Tsirelson~\cite{tsirelson2007}, who constructed similarly \say{paradoxical} stationary random vector
fields on \(\Rd\) with constant non-zero divergence.

One can view measurable entire functions as those entire functions that can be constructed without
the choice of an origin. The perspective of ``mathematics without an origin'' has recently received
significant attention within the descriptive set-theoretic community.  For instance, the field of
descriptive combinatorics has emerged as a central area, see~\cite{KechrisMarks} for an overview.
In a sense, equivariant analysis stands in the same relation to classical analysis as descriptive
combinatorics does to its classical counterpart.  In this paper, we explore the validity of
equivariant analogues of several classical results from complex analysis and partial differential
equations.

Weiss's work is set in the framework of ergodic theory.  An alternative approach, which we choose
here, is to consider standard Borel spaces and Borel actions.  For ``positive'' results---those that
establish the existence of desired objects---Borel formulations are stronger, and in this paper, we
adopt the perspective of Borel dynamics.  This choice is motivated by two factors: first, our main
results are of this positive nature, and second, the primary examples of actions to which we apply
our results naturally carry a Borel structure but lack any distinguished measures to highlight.
As shown by Ramsay~\cite{ramsayMeasurableGroupActions1985}, every measurable action of a locally
compact group on a standard Lebesgue space can be realized as a Borel action on an invariant subset
of full measure. Consequently, ``from the measure theoretic viewpoint, no generality is lost by
studying only Borel actions.'' (\cite{ramsayMeasurableGroupActions1985}*{p.~340}).

The proof of the existence of measurable entire functions from~\cite{MR1422707} adapts to the Borel
category with a single modification (cf.~Theorem~\ref{thm:Weiss-Borel}): Instead of relying
on a Rokhlin-type lemma to cover orbits with coherent rectangular regions---a tool generally
unavailable in Borel dynamics---one employs the concept of Borel toasts (following the terminology
introduced in~\cite{gaoForcingConstructionsCountable2022}) that cover orbits by compact sets with
connected complements, see Definition~\ref{def:toast} and Figure~\ref{fig:toast}.

A slight shift in perspective allows one to think of Borel entire functions as \emph{factor maps},
i.e., Borel measurable equivariant maps, into the space \(\holom\) of entire functions.  Indeed, the
space of entire functions is naturally equipped with the Borel \(\sigma\)-algebra generated by the
topology of uniform convergence on compact sets. Starting from a Borel entire \(F:X\to\C\), we
obtain a Borel \(\C\)-equivariant map \(\psi : X \to \holom\), where \(\psi\) assigns to each
\(x \in X\) the function \(F_x\).  Conversely, given such a map \(\psi\), \(F\) can be reconstructed
by evaluating \(\psi(x)\) at the origin. This perspective motivates the term \emph{equivariant
  analysis}.

As a side note, let us illustrate the difference between the Borel and ergodic viewpoints with the
following example concerning the growth rates of entire functions.  As shown
in~\cites{MR4041104,MR3905607}, in the context of ergodic theory, every free p.m.p.\ action of \(\C\)
admits non-constant measurable entire functions with a specific bound on their growth.  However,
this property does not generally hold for Borel actions, as demonstrated in
Appendix~\ref{sec:growth-introduction}: there are free Borel actions \(\C \acts X\) for which any
Borel entire \(F : X \to \C\) that is non-constant on all orbits must have unbounded growth
(Corollary~\ref{cor:borel-entire-unbounded-growth}) in the sense that for \emph{any} rate function
\(f\), there are \(x\in X\) for which \(\displaystyle{\max_{|z|\le R}|F(z\cdot x)|= O(f(R))}\) fails.

\subsection{The equivariant Weierstrass theorem}
\label{sec:motiv-example}
Before describing our main results, we begin with an example that captures the essence of the
questions we are interested in. It was also the starting point of this paper.

Consider the space \(\holom[\ne 0]\) of entire functions which do not vanish identically.  To each
element \(f \in \holom[\ne 0]\), one can associate its divisor of zeros \(\dm(f)\)---a discrete
multiset in the complex plane.  All such divisors form the space \(\divisp\), which carries natural
Borel and topological structures, as discussed in detail in Section~\ref{sec:divisors}.  The
classical Weierstrass theorem guarantees the existence of entire functions with a prescribed set of
zeros, which is equivalent to the surjectivity of the map \(\dm : \holom[\ne 0] \to \divisp\).  In
particular, this implies the existence of right-inverses \(\psi : \divisp \to \holom[\ne 0]\)
satisfying \(\dm(\psi(d)) = d\) for all \(d \in \divisp\).

The spaces \(\holom[\ne 0]\) and \(\divisp\) have natural actions of the additive group of complex
numbers \(\C\) through argument shifts, and the divisor map \(\dm\) is equivariant.  A natural
question is then whether the right-inverse map \(\psi\) can also be chosen to be equivariant. An
immediate obstruction arises: such a map \(\psi\) must preserve stabilizers (i.e., the group of
periods). In particular, the image of a non-trivial doubly-periodic divisor must be a
doubly-periodic entire function.  However, such functions are necessarily constant, and thus their
divisors are trivial.  Apart from this issue, the axiom of choice allows us to select a
representative from each \(\C\)-orbit in the remaining part \(\divisp\setminus \divisp[2]\), where
\(\divisp[2]\) denotes the space of doubly-periodic divisors.  An equivariant map \(\psi\) can then
be constructed without difficulty.

In practice, it is natural to ask for a map \(\psi\) with certain regularity properties.  Let us for
the moment restrict the discussion to the free parts of the actions, \(\free(\holom[\ne 0])\) and
\(\free(\divisp)\), which correspond to the elements with trivial stabilizers. (The significance of
this restriction will be clarified shortly.)  Both of these spaces are equipped with natural Polish
topologies: the topology of uniform convergence on compact sets for \(\holom[\ne 0]\), and the
topology of weak convergence (sometimes also called vague convergence) of measures, when elements of
\(\divisp\) are viewed as atomic Radon measures.  Interestingly, while \emph{continuous} equivariant
inverses to the divisor map do not exist even on the free part
(Theorem~\ref{thm:no-continuous-equivariant}),
we have the following positive result. We denote by \(\divis\) the space of signed divisors and by
\(\meromnz\) the space of meromorphic functions which do not vanish identically.

\begin{customthm}{\ref*{thm:equivariant-weierstrass}}
  There exists a Borel \(\C\)-equivariant map \(\psi : \free(\divisp) \to \holom[\ne 0]\) such that
  \(\dm\circ \psi={\rm id}_{\free(\divisp)}\).  Furthermore, there exists a Borel \(\C\)-equivariant
  map \(\psi : \free(\divis) \to \meromnz\) such that \(\dm\circ\psi={\rm id}_{\free(\divis)}\).
\end{customthm}

A rather curious corollary of Theorem~\ref{thm:equivariant-weierstrass} is that any
translation-invariant non-periodic point process on \(\C\) lifts, via the map \(\psi\), to a
translation-invariant random entire function with zeros at the point process\footnote{This corollary
  was proven earlier in an unpublished work with Oren Yakir, cf.~\cite{MR4658613}*{p.~3}.}. This
implies, for instance, that one can realize the 2D Poisson process as the zero set of a
translation-invariant random entire function.

Let us outline a general proof strategy one might follow to obtain equivariant liftings.  Each orbit
of a free \(\C\)-action can be identified with a copy of the acting group itself, allowing us to
speak of the geometries and shapes of various regions within the phase space.  The space
\(X = \free(\divisp)\) can be exhausted by Borel sets, which, on each orbit of the action, decompose
into disjoint unions of regions diffeomorphic to disks. These regions are coherently aligned with
one another.  This formal concept, known as a Borel toast, is discussed in detail in
Definition~\ref{def:toast}.  The construction of the required equivariant map
\(\psi : \free(\divisp) \to \free(\holom[\ne 0])\) proceeds inductively across these regions.
Intermediate steps are only ``partially'' equivariant, that is, they are equivariant only on parts
of the orbits, but these parts eventually exhaust the whole space and full equivariance is achieved
in the limit.

The key technical step of this construction relies on the following version of the classical Runge
theorem: \emph{Given \(\epsilon > 0\), a compact set \(K \subseteq \C\) with connected complement,
  and a holomorphic \(g\) on a neighborhood of \(K\), there exists an entire \(f\) such that
  \(\sup_{z \in K}|f(z) - g(z)| < \epsilon\) and, for \(z \in K\), \(f(z) = 0\) if and only if
  \(g(z) = 0\).  Moreover, the function \(f\) can be constructed in a Borel manner with respect to
  the compact sets \(K\) and functions \(g\).}

The primary part of this statement can be easily derived from the standard Runge theorem. However,
verifying the Borelness of the construction is a tedious task (see, for instance,
\cites{MR4025019,MR994977,MR236341}), which may get even more cumbersome for other variants of
Runge-type theorems.  For example, consider the distributional Poisson equation \(\Delta u = \mu\),
where \(\mu \in \measp{\Rd}\) is a positive Radon measure on \(\Rd\).  Similar to the case of
divisors, there exists a Borel equivariant inverse to \(\Delta\) defined on the free part of the
range, see Theorem~\ref{thm:equivariant-poisson}.  The corresponding Runge-type result involves
approximating a subharmonic function on a compact set \(K\) with a given Riesz measure \(\mu|_{K}\)
by a subharmonic function \(u\) defined on all of \(\Rd\), such that
\(\Delta u|_{K} = \mu|_{K}\).  Additionally, the construction must be Borel in all parameters
involved.  Ensuring Borelness in such cases often requires meticulous bookkeeping.  We have
therefore developed a framework that allows the direct application of standard analytic versions of
Runge-type theorems as black boxes, eliminating the need to establish the Borelness of their
constructions.  This is particularly advantageous because proofs of Runge-type results oftentimes
rely on the Hahn–Banach theorem (see, for instance, \cite{MR1070713}*{Section~III.8} and
\cite{Hormander2004-vz}*{II.3.4}), which is not inherently Borel.

In addition to the argument shift action of \(\C\), there is another action of a different group
that plays a significant role.  Let us revisit the map \(\dm : \holom[\ne 0] \to \divisp\).  Its
kernel \(\holomnz\), consisting precisely of entire functions without zeros, forms a \(G_{\delta}\)
subset of \(\holom\) and is a Polish group in the topology induced from \(\holom\). This group
\(H = \holomnz\) acts on \(\holom[\ne 0]\) via multiplication, and the equality
\(\dm(f_{1}) = \dm(f_{2})\) holds if and only if \(f_{1}/f_{2} \in H\).

For the inductive construction of a Borel equivariant \(\psi : \free(\divisp) \to \free(\holomnz)\),
it suffices to apply the Runge theorem exclusively to elements of the group~\(H\).  Specifically,
for a given holomorphic function \(f\) with no zeros on a compact set \(K\) with connected
complement, there exists an entire function \(g \in H\) that approximates \(f\) on \(K\).
Crucially, we do not initially require \(g\) to depend on \(f\) and \(K\) in a Borel manner; the
mere existence of such a \(g\) is sufficient.  The construction in the main theorem
(see Theorem~\ref{thm:main-theorem}) then automatically produces a Borel version that accommodates
functions \(f\) with zeros in \(K\) and ensures Borel measurability in the relevant parameters.  We
emphasize that in this variant, we do not need to preserve a prescribed set of zeros; instead, we
work with holomorphic functions that have no zeros.  This simplification is key to achieving Borel
measurability with ease.  Intuitively, transitioning from \(\holom[\ne 0]\) to \(\holomnz\)
transforms the Runge-type condition into an \emph{open} condition in the topology of \(\holomnz\).
Obtaining a Borel version then reduces to finding a Borel uniformization\footnote{ By a Borel
  uniformization of a set \(P\subseteq X\times Z\), we mean a Borel function
  \(f : \proj_X(P) \to Z\) such that \((x, f(x)) \in P\) for all \(x \in \proj_X(P)\).} of an open
set with non-empty slices, which is accomplished using standard descriptive set-theoretic
techniques.

\section{Our results}
\subsection{The main theorem}
\label{sec:main-theorem-description}
The scenario described above for the divisor map \(\dm\) is quite representative, and many problems
in complex analysis and PDE exhibit a similar structure.  For example, in the case of the Poisson
equation \(\Delta u = \mu\), the relevant group \(H\), in which we apply the Runge-type theorem, is the
group of harmonic functions \(\harm{\Rd}\).  Our main theorem captures this structure in an abstract
setting, and gives rather general sufficient conditions for the existence of Borel liftings.

\subsubsection{Standing assumptions and setting of the main theorem}
\label{sec:standing_assumptions}

Here is an overview of the setting of our main result. The precise definitions are given in
Sections~\ref{sec:preliminaries} and~\ref{sec:equivariant-borel-liftings}.  We suppose we are given
the following:

\paragraph{1.} Standard Borel spaces \(X\), \(Y\), and \(Z\), and a locally compact Polish group
\(G\) acting on each of them in a Borel way.  The action \(G\acts X\) is always assumed to be free.
In our main applications, \(G\) is either of the groups \(\Rd\) or \(\Rd\times \T^p\) of the
appropriate dimension, \(X\), \(Y\), and \(Z\) are spaces of functions, measures, or distributions on
\(G\), and \(G\) acts on all these spaces by argument shifts.  Also, \(X\) is often taken to be
\(\free(Y)\).

\paragraph{2.} An \emph{equation} \(\pi\circ\psi=\phi\), where \(\pi:Z\to Y\) is an equivariant Borel
surjection and where \(\phi:X\to Y\) is a Borel equivariant map, to which we seek a Borel
equivariant solution \(\psi:X\to Z\). In practice, we often take \(X=\free(Y)\) and
\(\phi={\rm id}_X\). For the equation \(\dm f=d\), \(Y\) would be the space \(\divisp\), \(Z\) is
the space of entire functions \(\holom[\ne 0]\), and \(\pi\) is the divisor map.  Our applications
to the Poisson equation correspond to taking \(\pi\) to be the Laplacian.

\paragraph{3.} A Polish group \(H\) and a Borel action \(H\acts Z\), whose orbit equivalence
relation \(E_H\) is classified by \(\pi\). That is, \(\pi(x)=\pi(y)\) if and only if \(x\) and \(y\)
lie in the same \(H\)-orbit.  In the aforementioned examples, this condition is automatically
satisfied since \(H\) is taken to be the kernel of \(\pi\), that is, \(H\) is either the
multiplicative group of zero-free entire functions or the additive group of harmonic functions.

\paragraph{4.} A continuous action \(\tau : G \acts H\) by automorphisms, yielding a Polish
semidirect product \(H \rtimes_{\tau} G\) equipped with the product topology and group operations
\((h_{1},g_{1})(h_{2},g_{2}) = (h_{1}\tau^{g_{1}}(h_{2}), g_{1}g_{2})\). The semidirect product is
assumed to act on \(Z\) in a Borel way, compatible with the given actions \(G\acts Z\) and
\(H\acts Z\).  That is, we have \(g\cdot z=(e_H, g)\cdot z\), \(h\cdot z=(h, e_G)\cdot z\) where
\(e_H\) and \(e_G\) are the units in \(H\) and \(G\), respectively.  In applications, \(G\) acts on
\(H\) by shifting the argument.

\paragraph{5.} A cofinal\footnote{i.e., exhaustive: for any compact \(K \subseteq G\), there exists
  some \(K'\in \mfR\) such that \(K \subseteq K'\).}  \(G\)-invariant class of compact sets
\(\mfR \subseteq \vietoris{G}\).  In applications, \(\mfR\) is a (sub)set of Runge domains for the
class of functions \(H\).  Oftentimes, \(\mfR\) is the class of compact sets in \(\C\) or \(\Rd\)
diffeomorphic to the closed unit ball.

\bigskip

This completes the general setup.  To show the existence of an equivariant Borel lifting of
\(\phi\), we need two additional key assumptions: the \emph{Runge approximation property} and the
existence of \emph{Borel toasts}.

\paragraph{6.} Let \(\unifn = (\norm{\cdot}_{K})_{K \in \vietoris{G}}\) be a \(\tau\)-family of
seminorms on \(H\).  This notion is formalized in Definition~\ref{def:U-family-seminorms} below, but
the reader may keep in mind that for our applications we usually use
\(\lVert f\rVert_K=\max_K|f|\) or
\(\lVert f\rVert_K=\max_K|\log|f|\hspace{1pt}|\).\\
\phantom{ }\quad We say that \(\unifn\) satisfies the \emph{\(\mfR\)-Runge property}
(Definition~\ref{def:weak-runge-property}) if, for any pairwise disjoint compact sets
\(K_1,\ldots,K_m\in\mfR\), any \(h_1,\ldots,h_m\in H\), and any given \(\epsilon>0\), there exists
\(h\in H\) such that \(\max_{1\le i\le m}\lVert h h_i^{-1}\rVert_{K_i}<\epsilon\).

\paragraph{7.} The final assumption is that the free action \(G\acts X\) admits a \emph{Borel
  \(\mfR\)-toast}, Definition~\ref{def:toast}.  A Borel toast is a Borel version of a Rokhlin tower
from ergodic theory, and consists of a sequence \((\mathcal{C}_{n})_n\) of Borel sets
\(\mathcal{C}_n\) together with Borel functions \(\lambda_n:\mathcal{C}_n\to \mfR\), such that the
regions \(R_n(c) = \lambda_n(c)\cdot c\) cover the orbits in a tree-like coherent way.  The central
toast axioms entail that for any \(x\in X\) and any compact \(K \subseteq G\) there is some region
such that \(K\cdot x\subseteq R_n(c)\).  Regions in the same ``generation'' are moreover disjoint,
and two regions \(R_m(c)\) and \(R_n(c')\), \(m<n\), are either disjoint
or \(R_m(c)\) is contained in \(R_n(c')\).\\
\phantom{ }\quad Any free \(\Rd\)-flow admits a Borel \(\mfR\)-toast with
\(\mfR = \mfD_{d}\) being the class of compact sets diffeomorphic to the closed unit ball, whereas
\(\R^{p}\times\T^{q}\)-flows admit toasts for the class of sets of the form \(K \times \T^{q}\), \(K \in
\mfD_{d}\) (see Section~\ref{sec:toast}).

\subsubsection{The main result}
In this framework, we establish the existence of Borel liftings.

\begin{customthm}{{\ref*{thm:main-theorem}}}
  Assume that the free action \(G \acts X\) admits a Borel \(\mfR\)-toast, and that the
  \(\tau\)-family \(\unifn\) on \(H\) satisfies the \(\mfR\)-Runge property.  Then, for any
  \(G\)-equivariant Borel map \(\phi : X \to Y\), there exists a \(G\)-equivariant Borel map
  \(\psi : X \to Z\) such that \(\pi \circ \psi = \phi\), making the following diagram commute:
  \begin{displaymath}
    \begin{tikzcd}[ampersand replacement=\&]
      G \acts X \arrow[r, dashed, "\psi"] \arrow[dr, "\phi"]
      \& G \acts Z \arrow[d, "\pi"]\\
      \& G \acts Y
    \end{tikzcd}
  \end{displaymath}
\end{customthm}

Specializing to \(X=\free(G\acts Y)\) and taking \(\phi\) to be the identity map \({\rm id}_X\), we
get the following corollary.

\begin{customcor}{\ref*{cor:main-corollary}}
  Suppose that the free action \(G \acts \free(Y)\) admits a Borel \(\mfR\)-toast, and that the
  \(\tau\)-family \(\unifn\) on \(H\) satisfies the \(\mfR\)-Runge property.  Then there exists a
  \(G\)-equivariant Borel map \(\psi : \free(Y) \to Z\) satisfying \((\pi \circ \psi)(y) = y\) for
  all \(y \in \free(Y)\).
\end{customcor}

\subsection{Applications of the main theorem}
\noindent We apply Theorem~\ref{thm:main-theorem} and Corollary~\ref{cor:main-corollary} to the
following maps \(\pi\):
\begin{enumerate}[leftmargin=.6cm, label={\sf \arabic*}., ref={\sf \arabic*}]
\item The divisor maps \(\dm : \holom[\ne 0] \to \divisp\) and \(\dm : \meromnz \to \divis\) that
  associate with a non-trivial entire or meromorphic function its divisor.
\item The quotient map \(\holom[\ne 0] \times \holom[\ne 0] \ni (f_{1}, f_{2}) \mapsto f_{1}/f_{2}
  \in \meromnz\).
\item The principal part map \(\ppm : \merom \to \ppart\) that associates with a meromorphic
  function its principal parts at each of its poles.
\item The d-bar operator \(\bar{\partial} : \smoothCC \to \smoothCC\), where
  \(\bar{\partial} = \frac{1}{2}(\frac{\partial}{\partial x} + \im \frac{\partial}{\partial y})\)
  and \(\smoothCC\) is the space of smooth complex-valued functions on \(\C\).
\item The distributional Laplacian \( \Delta \), which we will consider in three settings:
  \[ \Delta: \distrib{\Rd} \to \distrib{\Rd} ,\ \Delta\colon C^\infty(\Rd) \to C^\infty(\Rd) , \
    \text{and} \ \Delta : \sharm{\Rd} \to \measp{\Rd},\]
  where \(\sharm{\Rd}\) is the space of subharmonic functions and \(\measp{\Rd}\) is the space of
  Radon measures on \(\Rd\).
\item The boundary trace map \( R: \mathcal H (\R^{d+1}_+) \to C(\Rd) \), \( (Rh)(x)=h(x, 0) \),
  where \( \mathcal H(\R^{d+1}_+) \) is the space of functions harmonic in the upper half-space
  \[ \R^{d+1}_+ = \{(x, y): x\in\Rd, y>0 \} \] and continuous up to the boundary.
\item The heat operator
  \((\frac{\partial}{\partial t} - \Delta) : \smooth{\R^{d+1}} \to \smooth{\R^{d+1}}\).
\end{enumerate}

In all these cases, the domain and co-domain can be endowed with the structure of a standard Borel
space.  The corresponding maps between these spaces are Borel and surjective.  Additionally, in each
case, there is a natural argument-shift action of \(\Rd\) (for the appropriate value of \(d\)),
and the maps \(\pi\) are all equivariant with respect to these actions.

Section~\ref{sec:appl-main-theor} applies the main theorem to the cases described above.
Specifically, taking \(X = \free(Y)\) and \(\phi\) as the identity map,
Corollary~\ref{cor:main-corollary} is applied to \(\pi : \meromnz \to \divis\) and
\(\pi : \merom \to \ppart\). This yields the existence of Borel equivariant right-inverses
\(\psi : \free(\divis) \to \free(\meromnz)\) and \(\psi : \free(\ppart) \to \free(\merom)\), as stated
in Theorems~\ref{thm:equivariant-weierstrass} and~\ref{thm:mittag-leffler}, respectively.  These
results represent Borel equivariant versions of the Weierstrass and Mittag-Leffler theorems,
respectively. The restriction to the free part is essential, as no such map \(\psi\) exists on the
space \(\divis[1]\) of \(1\)-periodic divisors
(Theorem~\ref{thm:no-borel-equivariant-inverse-1-periodic}).

Theorem~\ref{thm:merormophic-quotient-entire} shows that the quotient map, taking a pair \((f_1,f_2)\)
of entire functions with disjoint divisors to their quotient \(f_1/f_2\),
admits a Borel equivariant right inverse on the space \(\free(\meromnz)\).

Similarly, Corollary~\ref{cor:main-corollary} applies to the Laplacian. For instance, it yields a
Borel equivariant map \(\psi : \free(\measp{\Rd}) \to \free(\sharm{\Rd})\) satisfying
\(\Delta \psi(\mu) = \mu\). In other words, we have an equivariant Brelot--Weierstrass theorem for
subharmonic functions.  In contrast to the Weierstrass theorem, the situation for the periodic part
is more nuanced.  Let \(\Gamma\) be a closed subgroup of \(\Rd\), and let \(\measp[\Gamma]{\Rd}\)
denote the space of Radon measures \(\mu\) with \(\stab(\mu) = \Gamma\), and \(\sharm[\Gamma]{\Rd}\)
denote the set of subharmonic functions with stabilizer \(\Gamma\).  A Borel equivariant map
\(\psi : \measp[\Gamma]{\Rd} \to \sharm[\Gamma]{\Rd}\) satisfying \(\Delta \psi(\mu) = \mu\) exists
if and only if \(\dim \Gamma \le d-2\) (Theorems~\ref{thm:equivariant-poisson-not-free}
and~\ref{thm:subharmonic-large-stabilizers}).

The Dirichlet problem in half-spaces with continuous boundary trace can likewise be solved in
a Borel equivariant way, see Theorem~\ref{thm:Dirichlet}.

Theorem~\ref{thm:d-bar-problem} provides an application of Theorem~\ref{thm:main-theorem} to the
d-bar equation \(\bar{\partial} f = g\).  It establishes the existence of Borel equivariant
right-inverses \[\psi : \free(\smoothCC) \to \smoothCC.\]

Theorem~\ref{thm:equivariant-heat} deals with the heat operator
\((\frac{\partial}{\partial t} - \Delta) : \smooth{\R^{d+1}} \to \smooth{\R^{d+1}}\),
which presents a distinct case compared to other applications.  Typically, the class \(\mfR\) is
chosen to consist of arbitrary sets diffeomorphic to the unit ball, since this is sufficiently
restrictive to ensure that the Runge theorem applies to finite disjoint unions of elements of
\(\mfR\).  However, this choice is insufficient for the heat operator.  The characterization of
the Runge domains for the heat operator is known~\cite{jonesApproximationTheoremRunge1975}. Here, the
class \(\mfR\) is defined as the collection of compact sets \(K \in \vietoris{\R^{d+1}}\) for which
\(P \setminus K\) is connected for every hyperplane \(P\) orthogonal to the time axis.  The
existence of Borel \(\mfR\)-toasts for arbitrary free Borel \(\R^{d+1}\)-actions is established in
Theorem~\ref{thm:existence-of-slice-filled-toasts}.

\subsection{Lack of equivariant inverses}
The remaining sections show that the statement of Theorem~\ref{thm:main-theorem} is, in many
respects, optimal. For instance, Section~\ref{sec:continuous-inverses} shows that the existence of
Borel equivariant inverses on the free part generally cannot be strengthened to the existence of
continuous such maps.  Here, we focus specifically on the maps \(\dm : \holom[\ne 0] \to \divisp\)
(Theorem~\ref{thm:no-continuous-equivariant}) and \(\bar{\partial} : \smoothCC \to \smoothCC\)
(Theorem~\ref{thm:no-dbar-inverse}), primarily because these cases involve spaces endowed with
natural Polish topologies.

The restriction to the free part of the range is essential to ensure the existence of Borel
equivariant inverses.  Such inverses generally fail to exist on the periodic parts of the range.  A
descriptive set-theoretic perspective is particularly useful here, as it provides a framework to
articulate the underlying reasons.  For example, we demonstrate that the argument-shift action of
\(\C\) on the space of entire functions with period 1 admits a Borel transversal
(Proposition~\ref{prop:action-on-H1-smooth}), whereas the action of \(\C\) on the space of divisors
with period 1 does not (Lemma~\ref{lem:D11-no-transversal}). This evidently precludes the existence
of a Borel equivariant inverse to the divisor map \(\dm\).  A similar argument rules out the
existence of a Borel equivariant map \(\psi : \meas[\Gamma]{\Rd} \to \sharm[\Gamma]{\Rd}\)
when \(\dim \Gamma = d-1\).

The non-existence of Borel equivariant inverses on the \(1\)-periodic part also applies to the
\(\bar{\partial}\)-problem (Theorem~\ref{thm:non-exist-dbar}), though proving it requires a
different approach.  Our proof hinges on the existence of non-stationary \(1\)-periodic random
\(\smoothCC\)-functions \(f\) for which \(\bar{\partial}f\) is stationary.

Section~\ref{sec:lack-equiv-invers} treats another natural case that can be framed in terms of the
existence of a lifting \(\psi\) in the diagram of Theorem~\ref{thm:main-theorem}.  This is the
question of whether a given Borel entire function admits a Borel entire antiderivative.  While this
is always true for classical entire functions in complex analysis, there exist Borel entire
functions that lack Borel antiderivatives (Theorem~\ref{thm:sufficient-condition} and
Section~\ref{subsect:applications}).  Juxtaposing with the setting of
Theorem~\ref{thm:main-theorem}, we have \(G=\C\) acting on the spaces \(Z = Y = \holom \) by
argument shifts, the group \(H = \C\) is identified with the constant functions, and the action
\(G\acts H\) is trivial. The issue which prevents application of the corollary arises from the
failure of the Runge property among the constants.  Consequently, the main theorem does not apply,
and indeed the corresponding equivariant Borel liftings may fail to exist. This non-existence proof
relies on Birkhoff's classical observation that \(\C\acts \holom\) has dense orbits.

The exponential map, the absolute value function, the logarithmic derivative, and a few other maps
discussed in Section~\ref{subsect:applications}, also fail to admit equivariant inverses in
general. The latter has some interesting consequences for the relation between the Mittag-Leffler
and Weierstrass theorems in the equivariant world. The classical Weierstrass theorem is often
deduced from Mittag-Leffler's theorem by finding a meromorphic function \(g\) with principal part
\(m(w)/(z-w)\) for each desired zero \(w\) of multiplicity \(m(w)\).  A holomorphic function \(f\)
with logarithmic derivative \(\frac{f'}{f} = g\) has zeros at the poles of \(g\) and their
multiplicities are precisely \(m(w)\). However, this approach fails in equivariant analysis, as the
logarithmic derivative does not admit Borel equivariant right-inverses.

\subsection{The appendices}
The paper includes several appendices.

Appendix~\ref{sec:growth-introduction} demonstrates that, unlike in ergodic theory, there exist free
Borel actions \(\C \acts X\) for which every nowhere constant entire function exhibits unbounded
growth (Theorem~\ref{thm:growth-rate-limsup}).  The proof relies on a result by Gao, Jackson, Krohne,
and Seward on the vanishing rate of sequences of markers for
\(\Zd\)-actions~\cite{gaoForcingConstructionsCountable2022}*{Thm.~1.1}.

Appendix~\ref{topol-divis-merom} demonstrates that the multiplicative group of meromorphic functions
\(\meromnz\) admits no Polish group topologies.  In particular, the map
\(\dm : \meromnz \to \divis\) cannot be interpreted as a continuous homomorphism between Polish
groups.  Our results also resolve a question posed
in~\cite{grosse-erdmannLocallyConvexTopology1995}, proving that there are no metrizable complete
algebra topologies on \(\merom\) (Theorem~\ref{thm:no-completely-metrizable-topological-algebra}).
The question of whether non-metrizable such topologies exist, also raised in
\cite{grosse-erdmannLocallyConvexTopology1995}, is not addressed by our methods.

Appendix~\ref{sec:rung-theor-peri} establishes a version of the Runge theorem for periodic
harmonic functions, which is needed for proving the existence of Borel equivariant inverses to the
Laplacian \(\Delta\) on \(\meas[\Gamma]{\Rd}\) when \(\dim \Gamma \le d-2\).  While this result is a
special case of the Lax--Malgrange approximation theorem and is certainly well-known to the experts,
we provide a self-contained proof for the reader's convenience.

Appendix~\ref{sec:borel-toasts-heat} establishes the existence of Borel \(\mfR\)-toasts for free
\(\R^{d+1}\)-flows for the class \(\mfR\) of compact sets that are Runge domains for the heat
operator.

The results in Appendices~\ref{sec:growth-introduction}, \ref{topol-divis-merom}, and
\ref{sec:borel-toasts-heat} are, to the best of our knowledge, new. However, as they lie outside the
main scope of this work, we have included them in the appendices.

\subsection*{Frequently used notation}

\begin{notation-list}
\item[\(X,Y,Z\)] Standard Borel spaces (see Section~\ref{sec:preliminaries} for the definition).

\item[\(G,H\)] Polish groups (Section~\ref{sec:preliminaries}).  Typically, \(G\) is assumed locally
  compact.

\item[\(G\acts X\) etc.] A Borel action (Section~\ref{sec:preliminaries}). For \(g\in G\) and \(x\in X\),
the result of the action is denoted by \( g\cdot x\).

\item[\(a_Z,a_X,\tau\)] Notation for Borel actions. The action \(\tau\) is always by automorphisms
  \(G\acts H\).

\item[\(\free (X)\)] The free part of \(X\) w.r.t.\ the action of \(G\), that is,
  \[ \free (X) =\{x\in X\colon g\cdot x \ne x,\ g\in G\setminus \{e\}\} ,\]
  where \(e\) is the unit of \(G\).

\item[\(\pi,\phi,\psi\)] Equivariant Borel maps.  Typically, \(\pi:Z\to Y\) and \(\phi:X\to Y\) are
  given, and we seek an equivariant lifting \(\psi\) satisfying \(\pi\circ\psi=\phi\) (see
  Section~\ref{sec:standing_assumptions}).

\item[\(\effros{G},\vietoris{G}\)] Effros and Vietoris Borel spaces of the group \(G\)
  (Section~\ref{sec:effros-vietoris}).

\item[\(\mfD_d,\mfC_d\)] Classes of all compact subsets of \(\Rd\) diffeomorphic to the unit ball,
  and with connected complements, respectively.

\item[\(\mfR\)] A \(G\)-invariant cofinal family of compact subsets of \(G\)
  (Section~\ref{sec:equiv-unif}).

\item[\((\mathcal{C}_n,\lambda_n)_n\)] A Borel toast (Section~\ref{sec:toast}) on a standard Borel
  space \(X\).

\item[\(R_n(c)\)] Regions \(R_n(c)=\lambda_n(c)\cdot c\), \(c\in\mathcal{C}_n\), on the orbits,
  forming part of a Borel toast.

\item[\(\unifm\)] A family of pseudometrics on \(Z\), satisfying the axioms of an \(a_Z\)-family
  for an action \(a_Z:G\acts Z\) (Definition~\ref{def:U-family}).

\item[\(\unifn\)] A family of seminorms on the group \(H\) satisfying the axioms of a
  \(\tau\)-family (Definition~\ref{def:U-family-seminorms}) for an action \(\tau:G\acts H\). Often
  required to satisfy the \(\mfR\)-Runge property (Definition~\ref{def:weak-runge-property}).

\item[\(\holom\), {\(\holom[\ne 0]\)}, \(\holomnz\)] Spaces of entire functions, of entire functions
  which do not vanish identically, and of entire functions without zeros, respectively
  (Section~\ref{sec:merom-entire-functions}).

\item[\(\merom, \meromnz\)] Space of all meromorphic functions, and space of those which do not
  vanish identically (i.e., invertible elements) (Section~\ref{sec:merom-entire-functions}).

\item[\(\divis,\divisp\)] Spaces of all (signed) divisors and positive divisors, respectively
  (Section~\ref{sec:divisors}).

\item[\(\ppart\)] The space of principal parts (Section~\ref{sec:mitt-leffl}).

\item[\(\dm\)] The divisor map \(\dm:\merom\to\divis\) (Section~\ref{sec:divisors}), assigning to
  each meromorphic function its divisor of zeros and poles.

\item[\(\ppm\)] The principal part map \(\ppm:\merom\to\ppart\)
  (Section~\ref{sec:mitt-leffl}), assigning to a meromorphic function its principal part.

\item[\(L^1_{\rm loc}(\Rd)\)] Space of locally Lebesgue integrable real-valued functions on
  \(\Rd\) (Section~\ref{sec:subh-funct}).

\item[\(\mathcal{SH}\), \(\mathcal{H}\)] The spaces of subharmonic and harmonic functions,
  respectively (Section~\ref{sec:subh-funct}).

\item[\(\mathcal{M},\mathcal{M}^+\)] The spaces of signed and positive Radon measures, respectively
  (Section~\ref{sec:subh-funct}).
\end{notation-list}

\subsection*{Central definitions}

\begin{definitions}

\item[Borel toast \((\mathcal{C}_n,\lambda_n)_n\)] -- \qquad Definition~\ref{def:toast}

\item[\(\mfR\)-Runge property] -- \qquad Definition~\ref{def:weak-runge-property}

\item[\((\mfR,P)\)-Runge property] -- \qquad Definition~\ref{def:runge-property}

\item[\(a_Z\)-family of pseudometrics] -- \qquad Definition~\ref{def:U-family}

\item[\(\tau\)-family of seminorms] -- \qquad Definition~\ref{def:U-family-seminorms}

\item[Concrete classifiability] -- \qquad Section~\ref{sec:polish-borel-spaces}

\item[Borel transversal] -- \qquad Section~\ref{sec:polish-borel-spaces}
\end{definitions}

\bigskip

\noindent \textbf{A note concerning the exposition.} This paper uses tools from a number of areas,
including descriptive set theory, Borel dynamics, topological vector spaces, complex analysis, and
PDE.  We have therefore tried to provide explanations of several standard and well-known arguments,
hopefully, to the benefit of the reader.

\bigskip

\noindent \textbf{Acknowledgment.} We express our gratitude to Benjamin Weiss and Oren Yakir for
several very interesting discussions during the work on this paper. We also thank Mark Agranovsky,
Alberto Enciso, and Daniel Peralta-Salas for helpful discussions on the intricacies of PDE and
Stephen Gardiner for helping with references.

\bigskip

\noindent \textbf{Funding.}  Konstantin Slutsky was partially supported by NSF grant
DMS-2153981. The research of Mikhail Sodin was supported by ISF Grants 1288/21, 2319/25 and by BSF
Grants 202019, 2024142.  Aron Wennman was supported by Odysseus Grant G0DDD23N from Research
Foundation Flanders (FWO).

\section{Descriptive set-theoretic preliminaries}
\label{sec:preliminaries}

\subsection{Polish and Borel spaces}
\label{sec:polish-borel-spaces}

A \emph{Polish space} is a separable completely metrizable topological space. A \emph{Polish group}
is a topological group whose underlying topology is Polish.  For a topological space, its
\emph{Borel \(\sigma\)-algebra} is the \(\sigma\)-algebra generated by the open sets. A
\emph{standard Borel space} is a pair \((X, \borel)\), where \(\borel\) is a \(\sigma\)-algebra on
\(X\) that coincides with the Borel \(\sigma\)-algebra for some Polish topology on \(X\).  All
uncountable standard Borel spaces are isomorphic. Two standard references for the theory of Polish
and standard Borel spaces are the books by Kechris \cites{Kechris} and Srivastava \cite{MR1619545}
(when readily available, we try to give references to both).

An action \(G \acts X\) of a Polish group on a standard Borel space \(X\) is \emph{Borel} if the map
\(G \times X \ni (g,x) \mapsto g\cdot x \in X\)
is Borel measurable.  A continuous action of \(G\)
on a Polish space is defined similarly.  (When no confusion should occur, we sometimes simplify
notation and write \(g x\) for the action.)  An important result due to Douglas Miller
(see~\cite{Kechris}*{9.17} or~\cite{MR1619545}*{Thm.~4.8.4}) states that for any Borel action
\(G \acts X\) of a Polish group and any \(x \in X\), the stabilizer
\(\stab(x) = \{g \in G : g\cdot x = x\}\) is necessarily a closed subgroup of \(G\).

For an action \(G \acts X\), the \emph{free part} of the action, denoted \(\free(G \acts X)\) or
simply \(\free(X)\), is defined as:
\[\free(X) = \{x \in X : g\cdot x \ne x \textrm{ for all } g \ne e\}
  = \{x \in X : \stab(x) = \{e\}\}.\]
Recall that a subset of a Polish space is \(G_{\delta}\) if it is a countable intersection of open
sets.  The following proposition is well-known, and its proof is included for the reader's
convenience.

\begin{proposition}
  \label{prop:free-part-g-delta}
  Let \(G\) be a locally compact Polish group, and let \(G \acts X\) be a continuous action on a
  Polish space \(X\).  Then the set \(\free(G \acts X)\) is \(G_{\delta}\).
\end{proposition}

\begin{proof}
Let \(K \subseteq G\) be compact, and define
\[
  Y_{K} = \{y \in X : g\cdot y = y \textrm{ for some } g \in K\}.
\]
Note that \(Y_{K}\) is closed. Indeed, suppose \(y_{n} \in Y_{K}\), \(n \in \N\), converges to
some \(y \in X\).  Let \(g_{n}\in K\) satisfy \(g_{n}\cdot y_{n} = y_{n}\).  By passing to a
subsequence if necessary, we may assume \(g_{n} \to g\) for some \(g \in K\).  By continuity of
the action, \(g_{n}\cdot y_{n} \to g\cdot y\), but \(g_{n}\cdot y_{n} = y_{n} \to y\), so
\(g\cdot y = y\) and thus \(y \in Y_{K}\).

Let \((K_{n})_{n}\) be a countable sequence of compact sets such that
\(\bigcup_{n} K_{n} = G \setminus \{e\}\).  Then
\(\free(G \acts X) = X \setminus \bigcup_{n} Y_{K_{n}} = \bigcap_{n} (X\setminus Y_{K_{n}})\),
which is \(G_{\delta}\).
\end{proof}

A \emph{Borel equivalence relation} on a standard Borel space \(X\) is a Borel subset
\(E \subseteq X \times X\) that is reflexive, symmetric, and transitive: \((x,x) \in E\),
\((x,y) \in E\) implies \((y,x) \in E\), and \((x,y), (y,z) \in E \) implies \((x,z) \in E\) for all
\(x, y, z \in X\). The notation \(x E y\) is equivalent to \((x,y) \in E\).

An action \(a : G \acts X\) generates the \emph{orbit equivalence relation} \(E_{a}\) on \(X\),
defined by \(x E_{a} y\) whenever \(G \cdot x = G \cdot y\).  When the action is clear from the
context, we may write \(E_{G}\) instead of \(E_{a}\).  Orbit equivalence relations of Borel actions
of locally compact Polish groups are always Borel \cite{beckerDescriptiveSetTheory1996}*{p.~109}.

A map \(\rho : E_{G} \to G\) is called a \emph{cocycle} if it satisfies the \emph{cocycle identity}:
\[\rho(x,y) = \rho(z,y) \rho(x,z) \quad \textrm{for all } x,y,z \in X \textrm{ such that } x
  E_{G} y E_{G} z.\]
For example, if the action is free, then for each pair of \(E_{G}\)-equivalent elements \(x,y\),
there corresponds a unique \(g \in G\) such that \(g \cdot x = y\); we denote this element by
\(\rho_{G}(x,y)\), or simply \(\rho(x,y)\) when there is no danger of confusion.

A Borel equivalence relation \(E\) is \emph{concretely classifiable}\footnote{The terms
  \emph{smooth} and \emph{tame} are also commonly used in the literature.} if there exists a
standard Borel space \(Y\) and a Borel map \(f : X \to Y\) such that \(x E y\) if and only if
\(f(x) = f(y)\), for all \(x, y \in X\).  A Borel \emph{transversal} for \(E\) is a Borel set
\(T \subseteq X\) that intersects each \(E\)-class in exactly one point.

It is worth mentioning that an orbit equivalence relation, given by a Borel action of a Polish
group, is concretely classifiable if and only if it admits a Borel transversal. One direction
(\emph{existence of a transversal implies concrete classifiability}) is straightforward. The other
direction is due to Burgess, see \cite{Kechris}*{Exercise~18.20, cf.\ p.~360}, though we will not
use this hereafter.

We will need several times the standard fact that any Borel action \(G\acts X\) of a compact Polish
group \(G\) on a standard Borel space \(X\) has a Borel transversal.  For continuous actions on
compact spaces this can be seen by noting that the map
\[
X \ni x \mapsto G \cdot x \in \vietoris{X}
\]
is continuous with respect to the Vietoris topology
on the space of compact subsets of \(X\) (Section~\ref{sec:effros-vietoris}).  If
\(s : \vietoris{X} \to X\) is a Borel selector given by the Kuratowski--Ryll-Nardzewski theorem (see
Section~\ref{sec:effros-vietoris}), then \(\{s(G \cdot x) : x \in X\}\) is a Borel transversal for
the action.  The general case follows by embedding an arbitrary Borel action \(G \acts X\) into a
continuous action on a compact set.  Pick a Haar measure on \(G\) and consider the action
\(G \acts L^{\infty}(G,\mu)\) given by \((g \cdot f)(h) = f(g^{-1}h)\).  This is a continuous action
when \(L^{\infty}(G,\mu)\) is endowed with the weak*-topology.  Furthermore, the unit ball \(B\) of
\(L^{\infty}(G,\mu)\) is a \(G\)-invariant compact set (Alaoglu's theorem).  Any Borel action
\(G \acts X\) can be embedded into \(G \acts B\).  Indeed, if the standard Borel space \(X\) is
identified with \([0,1]\), then an embedding is given by \(G \ni x \mapsto f_{x} \in B\),
\(f_{x}(g) = g^{-1}\cdot x\).  Further details, as well as an alternative argument, can be found
in~\cite{beckerDescriptiveSetTheory1996}*{pp.~27--28}.

Let \(E_{1}\) and \(E_{2}\) be Borel equivalence relations on spaces \(X_{1}\) and \(X_{2}\),
respectively. A Borel \emph{reduction} from \(E_{1}\) to \(E_{2}\) is a Borel map
\(f : X_{1} \to X_{2}\) such that \(x E_{1} y\) if and only if \(f(x) E_{2} f(y)\), for all
\(x,y \in X_{1}\).  An equivalence relation is concretely classifiable if and only if it is Borel
reducible to the equality relation on some standard Borel space.  If \(E_{1}\) is Borel reducible to
\(E_{2}\) and \(E_{2}\) is concretely classifiable, then so is \(E_{1}\).

Throughout this paper, \(G\) denotes a locally compact Polish group, and unless stated otherwise,
\(G \acts X\) is a free Borel action on a standard Borel space.  (Actions of \(G\) on spaces other
than \(X\) are generally not assumed to be free.)  An \emph{exhaustion} of a locally compact Polish
group \(G\) by compact sets is a sequence of compact sets \((K_{n})_{n}\) such that
\(K_{n} \subseteq \inter K_{n+1}\), \(n \in \N\), and \(G = \bigcup_{n} K_{n}\).  Here \(\inter K\)
denotes the interior of \(K\).

\subsection{Effros and Vietoris spaces}
\label{sec:effros-vietoris}

Let \(X\) be a locally compact Polish space. The space \(\vietoris{X}\) of compact subsets of \(X\)
is equipped with the \emph{Vietoris topology}~\cite{Kechris}*{4.F}, \cite{MR1619545}*{p.~66}, whose
basis is parametrized by open sets \(U_{0}, U_{1}, \ldots, U_{n} \subseteq X\) and is given by
\[\{K \in \vietoris{X} : K \subseteq U_{0}, K \cap U_{1} \ne \varnothing, \ldots, K \cap U_{n} \ne
  \varnothing\}.\]
The space \(\vietoris{X}\) is Polish~\cite{Kechris}*{4.25}, \cite{MR1619545}*{Cor.~2.4.16}.

The \emph{Effros Borel space} \(\effros{X}\) consists of all closed subsets of \(X\) and is endowed
with the \(\sigma\)-algebra generated by the sets
\(\{F \in \effros{X} : F \cap U \ne \varnothing\}\), where \(U \subseteq X\) is
open~\cite{Kechris}*{12.C}, \cite{MR1619545}*{p.~97}.  This is a standard Borel
space~\cite{Kechris}*{12.6}, \cite{MR1619545}*{Thm.~3.3.10}, and the Borel structure of
\(\vietoris{X}\) coincides with the one induced from \(\effros{X}\)~\cite{Kechris}*{12.11i}.

An important fact about Effros Borel spaces is the Kuratowski--Ryll-Nardzewski
theorem~\cite{Kechris}*{12.13}, \cite{MR1619545}*{Thm.~5.2.1}, which guarantees the existence of
Borel selectors \(s_{n} : \effros{X} \to X\), \(n \in \N\), i.e., Borel functions such that
\(\{s_{n}(F)\}_{n}\) is dense in \(F\) whenever \(F \in \effros{X}\) is non-empty.

\subsection{Borel toasts}
\label{sec:toast}
\begin{figure}
  \centering
  \includegraphics[width=.6\linewidth]{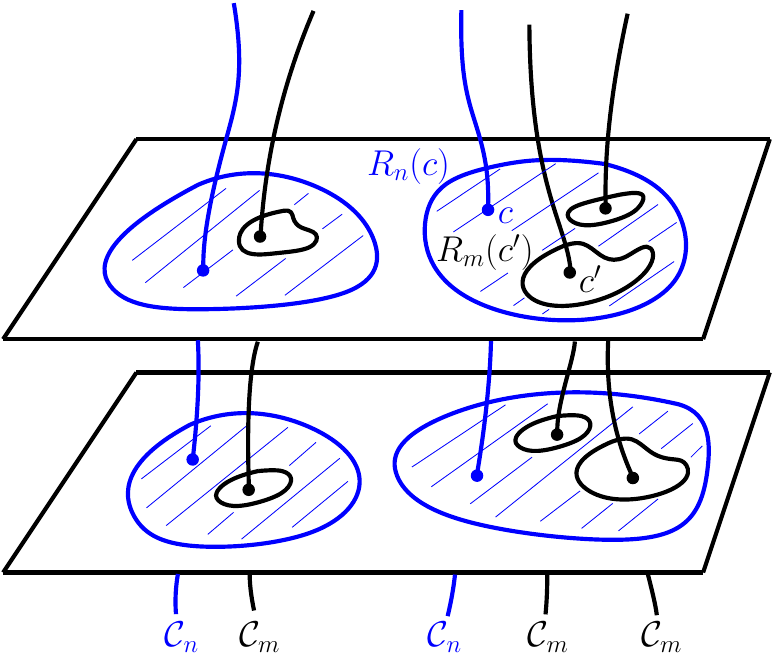}
\caption{Illustration of a Borel toast, showing two cross-sections \(\mathcal{C}_m\)
and \(\mathcal{C}_n\), \(m<n\). Two points \(c\in \mathcal{C}_n\) and
\(c'\in \mathcal{C}_m\) with regions \(R_m(c')\) and \(R_n(c)\) around them
are highlighted. The two planes correspond to distinct \(G\)-orbits.
\label{fig:toast}}
\end{figure}

A key tool in many proofs within Borel dynamics is the concept termed ``Borel toast'' by
Gao--Jackson--Krohne--Seward in~\cite{gaoForcingConstructionsCountable2022}.  It is the descriptive
set-theoretic counterpart of Rokhlin towers in ergodic theory.  Specific properties that one chooses
to require in the definition of a toast vary depending on the application.  In its simplest form,
this concept is closely related to the notion of hyperfiniteness. Borel actions of \(\Z^{d}\) were
shown to be hyperfinite by Weiss (unpublished) (see
Jackson--Kechris--Louveau~\cite{jacksonCountableBorelEquivalence2002}*{Thm.~1.16} for a more general
result for groups of polynomial growth).  In a more refined setting, the definition of a toast puts
restrictions on the shape of the regions or on the mutual location of regions across different
levels.  A powerful way of constructing toasts is based on the cross-section construction from
Marks--Unger \cite{marksBorelCircleSquaring2017}*{Appendix~A} based on the ideas of
Boykin--Jackson~\cite{boykinBorelBoundednessLattice2007}; we provide some brief comments in
Appendix~\ref{sec:borel-toasts-heat}.  Other applications of the toast idea can be found in
Gao--Jackson--Krohne--Seward~\cites{gaoForcingConstructionsCountable2022,gaoBorelCombinatoricsAbelian2024,
  gaoContinuousCombinatoricsAbelian2023} and Slutsky~\cite{MR4215747}.

The following definition is tailored to the needs of
Lemma~\ref{lem:runge-equivariant-uniformization}, and is essentially from \cite{MR4215747} (see,
however, Remark~\ref{rem:layered-regions}).  We let \(\vietoris[*]{G}\) denote the collection of
compact subsets of \(G\) with non-empty interior.

\begin{definition}
  \label{def:toast}
  Let \(a_{X} : G \acts X\) be a free Borel action of a locally compact Polish group on a standard
  Borel space. A Borel \emph{toast} for \(a_{X}\) is a sequence
  \((\mathcal{C}_{n}, \lambda_{n})_{n}\) of Borel sets \(\mathcal{C}_{n} \subseteq X\) and Borel
  functions \(\lambda : \mathcal{C}_{n} \to \vietoris[*]{G}\) satisfying the following
  conditions. For each \(n\), define \(R_{n}(c) = \lambda_{n}(c)\cdot c\) for
  \(c \in \mathcal{C}_{n}\), and let \(X_{n} = \bigcup_{c \in \mathcal{C}_{n}} R_{n}(c)\). Then:
  \begin{enumerate}[leftmargin=1cm, label={\sf \arabic*}., ref={\sf \arabic*}]
  \item\label{item:toast-disjoint} \(R_{n}(c_{n}) \cap R_{n}(c'_{n}) = \varnothing\) for all
    distinct \(c_{n}, c'_{n} \in \mathcal{C}_{n}\).
  \item\label{item:toast-coherent} For all \(m < n\), \(c_{m} \in \mathcal{C}_{m}\), and
    \(c_{n} \in \mathcal{C}_{n}\), either \(R_{m}(c_{m}) \cap R_{n}(c_{n}) = \varnothing\) or
    \(R_{m}(c_{m}) \subseteq R_{n}(c_{n})\).
  \item\label{item:toast-layered} For all \(c_{n} \in \mathcal{C}_{n}\) there exists
    \(c_{n+1} \in \mathcal{C}_{n+1}\) such that \(R_{n}(c_{n}) \subseteq R_{n+1}(c_{n+1})\).
  \item\label{item:toast-directed} For all \(c_{m_{1}} \in \mathcal{C}_{m_{1}}\) and
    \(c_{m_{2}} \in \mathcal{C}_{m_{2}}\), there exist \(n > m_{1}, m_{2}\) and an element
    \(c_{n} \in \mathcal{C}_{n}\) such that \(R_{m_{i}}(c_{m_{i}}) \subseteq \inter R_{n}(c_{n})\)
    for \(i = 1,2\), where \(\inter R_{n}(c_{n}) = (\inter \lambda_{n}(c_{n})) \cdot c_{n}\).
  \item\label{item:lacunary} There exists a neighborhood of the identity \(U \subseteq G\) such that
    \(U \subseteq \lambda_{n}(c_{n})\) for all \(c_{n} \in \mathcal{C}_{n}\) and all \(n\).
  \item\label{item:toast-exhaustive} \(\bigcup_{n} X_{n} = X\).
  \item\label{item:toast-rational} The range \(\ran \lambda_{n}\) is countable for each \(n\) and
    \[\{\rho(c_{m},c_{n}) : c_{m} \in \mathcal{C}_{m}, c_{n} \in \mathcal{C}_{n}, c_{m} E_{a_{X}}
      c_{n}\}\]
    is countable, where \(\rho : E_{a_{X}} \to G\) is the cocycle defined by the condition
    \(\rho(x_{1},x_{2})x_{1} = x_{2}\).
  \end{enumerate}
  We say that \((\mathcal{C}_{n}, \lambda_{n})_{n}\) is an \(\mfR\)-toast, where
  \(\mfR \subseteq \vietoris{G}\), if \(\ran \lambda_{n} \subseteq \mfR\) for all \(n\).
\end{definition}

Here are a couple of remarks and consequences of this definition.

\begin{remark}
  \label{rem:layered-regions}
  Item~\eqref{item:toast-layered} should not be confused with the so-called layeredness condition.
  In some descriptive set-theoretic applications, it is important to know that containment
  \(R_{m}(c_{m}) \subseteq R_{n}(c_{n})\), \(m < n\), implies that the region \(R_{m}(c_{m})\) is
  far from the boundary of \(R_{n}(c_{n})\).  In this case, one generally cannot guarantee that for
  each \(c_{n} \in \mathcal{C}_{n}\) there is some \(c_{n+1} \in \mathcal{C}_{n+1}\) satisfying
  \(R_{n}(c_{n}) \subseteq R_{n+1}(c_{n+1})\).  However, for the purposes of this section, we allow
  regions at different levels to coincide, which is why item~\eqref{item:toast-layered} can be
  easily achieved.

  Indeed, if \((\mathcal{C}'_{n}, \lambda'_{n})_{n}\) satisfies all the items in the definition of
  the toast except possibly for item~\eqref{item:toast-layered}, then we can set
  \(X'_{n} = \bigcup_{c_{n} \in \mathcal{C}'_{n}} R'_{n}(c_{n})\) and
  \begin{displaymath}
    \begin{aligned}
      &\mathcal{C}_{n}
        = \mathcal{C}'_{n} \cup \bigcup_{k < n}\{c_{k} \in \mathcal{C}'_{k} : c_{k}
        \not \in X'_{m} \textrm{ for all } k < m \le n\}, \\
      &\lambda_{n}(c_{n})
        = \lambda'_{k}(c_{k}) \quad \textrm{for \(k\) such that } c_{n} = c_{k} \in \mathcal{C}'_{k}.
    \end{aligned}
  \end{displaymath}
  In plain words, we can repeat regions \(R_n(c_n)\) to ensure that \(X_{n+1}\) contains \(X_{n}\).
\end{remark}

\begin{remark}
  \label{rem:cross-section-re-indexing}
  Definition~\ref{def:toast} does not require sets \(\mathcal{C}_n\) to be complete, i.e.,
  \(G \cdot \mathcal{C}_{n}\) may be a proper subset of \(X\).  If needed, completeness can be
  achieved by ``re-indexing'' the points to ensure that \(\mathcal{C}_{0}\) intersects every orbit.
  Items~\eqref{item:toast-disjoint}, \eqref{item:lacunary} and~\eqref{item:toast-exhaustive} will
  then ensure that each \(\mathcal{C}_n\) is a \emph{Borel cross-section}---a Borel set that
  intersects each orbit in a lacunary set.

  As Kechris showed, cross-sections exist for arbitrary Borel actions of locally compact
  groups~\cite{kechrisCountableSectionsLocally1992}. On the other hand, the existence of a Borel
  toast is a considerably stronger condition.

  In our applications, we work with the groups \(G=\Rd\) and \(G=\Rd\times \T^p\), and, except
  for the section studying the heat operator, it suffices to take the class
  \(\mfR \subseteq \vietoris{G}\) to be either the class \(\mfD_{d}\) of compact sets diffeomorphic
  to the \(d\)-dimensional ball or the products of such sets with \(\T^p\).  Free Borel actions of
  \(\Rd\) admit Borel \(\mfD_{d}\)-toasts in the sense of Definition~\ref{def:toast}.  This is
  essentially \cite{MR4215747}*{Thm.~5}.  The formulation therein does not include
  item~\eqref{item:toast-layered}, but as we explained in Remark~\ref{rem:layered-regions}, this
  property can be easily achieved.  Borel actions of \(\Rd \times \T^{p}\) admit Borel
  \(\mfD_{d} \times \T^{p}\)-toasts, as will follow from Lemma~\ref{lem:toasts-for-direct-products}.
\end{remark}

\begin{lemma}
  \label{lem:exhaustive-taust}
  Let \((\mathcal{C}_{n},\lambda_{n})_{n}\) be a Borel toast for \(G \acts X\).  For all \(x \in X\)
  and \(K \in \vietoris{G}\), there exist \(n\) and \(c_{n} \in \mathcal{C}_{n}\) such that
  \(K \cdot x \subseteq \inter R_{n}(c_{n})\).
\end{lemma}

\begin{proof}
  Fix \(x\in X\), \(K\in \vietoris{G}\), and \(h \in K\). By item~\eqref{item:toast-exhaustive} of
  the toast definition, there exists \(c_{m} \in \mathcal{C}_{m}\) and \(g_{h} \in \lambda(c_{m})\)
  such that \(h\cdot x = g_{h}\cdot c_{m}\).  Without loss of generality, by
  item~\eqref{item:toast-directed}, increasing \(m\) if necessary, we may assume
  \(g_{h} \in \inter \lambda(c_{m})\).  Thus, there exists an open neighborhood of the identity
  \(U_{h}\) such that \(U_{h}g_{h} \subseteq \inter \lambda(c_{m})\), and consequently,
  \(U_{h}h \cdot x \subseteq \inter R_{m}(c_{m})\).

  The sets \((U_{h}h)_{h \in K}\) form an open cover of \(K\).  By compactness, we may pass to a
  finite sub-cover, yielding finitely many elements \(h_{i} \in K\), open sets
  \(U_{i} \subseteq G\), indices \(m_{i}\), and elements \(c_{m_{i}} \in \mathcal{C}_{m_{i}}\) such
  that \(K \subseteq \bigcup_{i}U_{i}h_{i}\) and
  \(U_{i}h_{i} \cdot x \subseteq \inter R_{m_{i}}(c_{m_{i}})\).  Iterating the property from
  item~\eqref{item:toast-directed}, we can find a single \(c_{n} \in \mathcal{C}_{n}\) such that
  \(R_{m_{i}}(c_{m_{i}}) \subseteq \inter R_{n}(c_{n})\) for all \(i\).  This \(c_{n}\) satisfies
  \(\bigcup_{i}U_{i}h_{i} \cdot x \subseteq \inter R_{n}(c_{n})\), and thus
  \(K \cdot x \subseteq \inter R_{n}(c_{n})\), as required.
\end{proof}

\begin{remark}
  \label{rem:lacunary-finitely-many-predecessors}
  Item~\eqref{item:lacunary} of Definition~\ref{def:toast} ensures that for any
  \(c_{n} \in \mathcal{C}_{n}\) there are only finitely many \(c_{m} \in \mathcal{C}_{m}\),
  \(m < n\), satisfying \(R_{m}(c_{m}) \subseteq R_{n}(c_{n})\).  This follows easily by identifying
  the orbit of \(c_{n}\) with \(G\) and using the Haar measure to limit the number of pairwise
  disjoint regions \(R_{m}(c_{m})\) inside \(R_{n}(c_{n})\).
\end{remark}

Finally, we argue that if every free Borel action of \(\Rd\) admits a Borel \(\mfR\)-toast, then
free actions of \(\Rd \times \T^p\) admit Borel \(\mfR \times \T^p\)-toasts.  Let \(T\) be a compact
Polish group and let, as before, \(G\) be a locally compact Polish group.  Given a cofinal
\(G\)-invariant class of compact sets \(\mfR \subseteq \vietoris{G}\), let \(\mfR \times T\) denote
the class \(\{L \times T : L \in \mfR\}\).  Note that the cofinality of \(\mfR\) in \(\vietoris{G}\)
implies the cofinality of \(\mfR \times T\) in \(\vietoris{G \times T}\).

\begin{lemma}
  \label{lem:toasts-for-direct-products}
  If every (free) Borel \(G\)-action admits a Borel \(\mfR\)-toast, then every (free) Borel
  \( G \times T\)-action admits a Borel \((\mfR \times T)\)-toast.
\end{lemma}
\begin{proof}
  Let \(G \times T \acts X\) be a Borel action and consider its restriction to the action
  \(T \acts X\).  Since \(T\) is compact, there exists a Borel transversal \(Y \subseteq X\) for
  \(E_{T}\).  Define a Borel action \(a : G \acts Y\) by setting \(a(g,y)\) to be the unique
  \(y' \in Y\) such that \(gy E_{T} y'\).  Equivalently, if \(s : X \to Y\) is the Borel selector
  for \(E_{T}\) corresponding to the chosen transversal, i.e., \(s(x) = y\) for the unique
  \(y \in Y\) such that \(x E_{T} y\), then \(a(g,y) = s(g \cdot y)\).  Note that \(a\) is free if
  the original action \(G \times T \acts X\) is free, since \(a(g,y) = y\) implies
  \(g \cdot y = t \cdot y\) for some \(t \in T\), which necessitates \(g = e\).

  By assumption, \(a\) admits a Borel \(\mfR\)-toast \((\mathcal{C}_{n}, \lambda_{n})_{n}\).  We
  transform it into a Borel \((\mfR \times T)\)-toast \((\mathcal{C}_{n}, \mu_{n})_{n}\) by setting
  \(\mu_{n}(c_{n}) = \lambda_{n}(c_{n}) \times T \).  The toast axioms are straightforward to
  verify.
\end{proof}

\section{Equivariant Borel liftings}
\label{sec:equivariant-borel-liftings}

The primary goal of this section is to establish Theorem~\ref{thm:main-theorem}.  Let \(G\) be a
locally compact Polish group acting in a Borel way on standard Borel spaces \(Z\) and \(Y\), and let
\(\pi:Z\to Y\) be a Borel \(G\)-equivariant surjection. The theorem provides sufficient conditions
under which, for a free Borel \(G\)-action \(G\acts X\) and a \(G\)-equivariant map \(\phi:X\to Y\),
there exists a Borel \(G\)-equivariant lifting \(\psi:X\to Z\) satisfying \(\pi\circ \psi=\phi\):
\begin{equation}
  \label{diagram:lift}
  \begin{tikzcd}[ampersand replacement=\&]
    G \acts X \arrow[r, dashed, "\psi"] \arrow[dr, "\phi"]
    \& G \acts Z \arrow[d, "\pi"]\\
    \& G \acts Y
  \end{tikzcd}
\end{equation}

We begin in Section~\ref{sec:equiv-unif} with a preliminary result
(Lemma~\ref{lem:runge-equivariant-uniformization}), which gives sufficient conditions under which a
Borel set \(P \subseteq X \times Z\) admits an \emph{equivariant Borel uniformization}, which we will
define a few lines below.  This result requires a measurable version of the Runge property
(Definition~\ref{def:runge-property}) in the space \(Z\).

As outlined in the introduction, we wish to refrain from establishing such measurable Runge theorems
in our applications.  For that reason, we develop a framework to transfer the application of the
Runge property from \(Z\) to the Polish group \(H\), where an a priori weaker Runge-type property
suffices (Definition~\ref{def:weak-runge-property}). The reader may wish to keep in mind the
relation between \(H\) and \(Z\) in the motivating examples from the introduction: in the case of
entire functions we take \(Z\) to be the space \(\holom[\ne 0]\) of entire functions which do not
vanish identically, while \(H\) is the multiplicative group \(\holomnz\) of zero-free entire
functions.  For our application to the Poisson equation, \(Z\) is the space of subharmonic functions
on \(\Rd\), while \(H\) is the additive group of harmonic functions.  The framework to accomplish
the transfer between \(H\) and \(Z\) is laid out in Section~\ref{sec:semidirect-products}.  The key
result is Lemma~\ref{lem:invariant-uniformizes}, which shows that a certain family of seminorms on
\(H\) induces a family of pseudometrics on \(Z\) with desirable properties.

In Section~\ref{sec:equivariant-borel-liftings-semiproduct}, we state and prove our main
result. This will be derived rather directly from the equivariant uniformization of
Lemma~\ref{lem:runge-equivariant-uniformization}.  The key technical step is to show that the
\(\mfR\)-Runge property in \(H\) implies the a priori stronger measurable Runge property in the
space \(Z\).

\subsection{Equivariant uniformizations}
\label{sec:equiv-unif}

Consider standard Borel spaces \(X\) and \(Z\) equipped with Borel \(G\)-actions \(a_X : G \acts X\)
and \(a_Z : G \acts Z\) of a locally compact Polish group \(G\). Let \(P \subseteq X \times Z\) be a
Borel subset. We say that \(P\) is \emph{\((a_X, a_Z)\)-equivariant} if for all \((x, z) \in P\) and
\(g \in G\), the pair \((gx, gz)\) also lies in \(P\). When the actions \(a_X\) and \(a_Z\) are
clear from the context, we may simply say that \(P\) is \emph{equivariant}. A \emph{uniformization}
of \(P\) is a function \(f : \proj_X(P) \to Z\) such that \((x, f(x)) \in P\) for all
\(x \in \proj_X(P)\). A uniformization \(f\) is said to be \emph{\((a_X, a_Z)\)-equivariant} if its
graph is equivariant, or equivalently, if \(f(g\cdot x) = g\cdot f(x)\) for all \(x \in \proj_X(P)\)
and \(g \in G\).

Conditions ensuring the existence of Borel uniformizations for a Borel set
\(P \subseteq X \times Z\) are well established (see \cite{Kechris}*{Sec.~18}). Our goal is to
provide an instance of conditions under which a set \(P \subseteq X \times Z\) admits an
\((a_{X}, a_{Z})\)-equivariant Borel uniformization
(Lemma~\ref{lem:runge-equivariant-uniformization} and Theorem~\ref{thm:main-theorem}). Prior work in
this area includes the recent work of Kechris and Wolman~\cite{2405.15111}, which, when restricted
to the framework of orbit equivalence relations, examines the existence of equivariant
uniformizations with the trivial action \(a_{Z}\).

Let us look at an example before we proceed:

\begin{exmpl}
  Let \(P=\{(x,z)\in X\times Z:\phi(x)=\pi(z)\}\), where \(X,Z,\phi\) and \(\pi\) are as in
  Diagram~\ref{diagram:lift}.  By the equivariance of \(\pi\) and \(\phi\), \(P\) is readily seen to
  be an equivariant Borel subset of \(X\times Z\).  The set \(P\) consists of all candidate pairs
  \((x,z)\) for which we could have \(\psi(x)=z\) for the lifting \(\psi:X\to Z\) in
  Diagram~\ref{diagram:lift}. An equivariant Borel uniformization of \(P\) is then nothing but the
  desired equivariant lifting.
\end{exmpl}

Throughout the rest of this section we assume that the action \(a_{X} : G \acts X\) is
free. (Actions of \(G\) on other spaces are generally not assumed to be free.)

We recall that a \emph{pseudometric} on a set \(Z\) is a map
\(d : Z \times Z \to \R^{\ge0} \cup \{\infty\}\)
that is symmetric, satisfies the triangle
inequality, and is zero on the diagonal, \(d(z,z) = 0\) for all \(z \in Z\).  Unlike a genuine
metric, a pseudometric may assign distance \(0\) to distinct points.  We also remind that
\(\vietoris{G}\) denotes the space of compact subsets of \(G\).

\begin{definition}
  \label{def:U-family}
  Let \(a_{Z} : G \acts Z\) be a Borel \(G\)-action.  A family
  \(\unifm = (d_{K})_{K \in \vietoris{G}}\) of pseudometrics on \(Z\) is said to be an
  \emph{\(a_{Z}\)-family} if it satisfies the following conditions:
  \begin{enumerate}[leftmargin=1cm, label={\sf \arabic*}., ref={\sf \arabic*}]
  \item\label{item:unif-monotone} \(d_{K}(z_{1},z_{2}) \le d_{K'}(z_{1}, z_{2})\) for all
    \(z_{1}, z_{2} \in Z\) and \(K, K' \in \vietoris{G}\) satisfying \(K \subseteq K'\).
  \item\label{item:unif-shift} \(d_{K}(gz_{1}, gz_{2}) = d_{Kg}(z_{1},z_{2})\) for all
    \(z_{1},z_{2} \in Z\), \(g \in G\), and \(K \in \vietoris{G}\).
  \item\label{item:unif-hausdorff} The family \(\unifm\) separates points: for any distinct
    \(z_{1}, z_{2} \in Z\) there exists \(K \in \vietoris{G}\) such that
    \(d_{K}(z_{1}, z_{2}) > 0\).
  \item\label{item:unif-complete} The uniformity generated by \( \unifm \) is complete: if
    \((z_{n})_{n}\) is \(\unifm\)-Cauchy (i.e., \(d_{K}\)-Cauchy for every \(K \in \vietoris{G}\)),
    then there exists some \(z \in Z\) such that \(d_{K}(z_{n},z) \to 0\) for all
    \(K \in \vietoris{G}\).
  \end{enumerate}
  An \(a_{Z}\)-family is \emph{Borel} if each \(d_{K}\) is Borel as a map
  \(d_{K} : Z \times Z \to \R^{\ge 0} \cup \{\infty\}\).
\end{definition}

In applications, \(Z\) is often a space of functions on the acting group \(G\), and the relevant
families of pseudometrics will oftentimes be of the form \(d_K(f,g)=\sup_{x\in K}|f(x)-g(x)|\).

\begin{definition}
  \label{def:confinale-invariant-class}
  Let \(\mfR \subseteq \vietoris{G}\) be a class of compact subsets of \(G\).  We say that \(\mfR\)
  is
  \begin{itemize}
  \item \emph{cofinal} if for all \(K \in \vietoris{G}\) there exists \(K' \in \mfR\) such that
    \(K \subseteq K'\);
  \item \emph{invariant}\footnote{More precisely, such a class is \emph{right-invariant}, but
      left-invariance will not be used in this section.} if \(Kg \in \mfR\) whenever \(K \in \mfR\)
    and \(g \in G\).
  \end{itemize}
\end{definition}
Given a set \(P \subseteq X \times Z\) and \(m \in \N\), we use the notation \(P^{(m)}\) for
the set
\[P^{(m)} = \{(x,z_{1}, \ldots, z_{m}) : (x,z_{i}) \in P \textrm{ for } 1 \le i \le m\} \subseteq X
  \times Z^{m}.\] In particular, \(P^{(1)} = P\).

Given a set \(P \subseteq X \times Z\) and \(x \in X\), we let \(P_{x}\) stand for
\(\{z \in Z : (x,z) \in P\}\).

\begin{definition}
\label{def:runge-property}
Let \(\mfR \subseteq \vietoris{G}\) be a cofinal invariant class of compact sets and let
\(P \subseteq X \times Z\) be an \((a_{X}, a_{Z})\)-equivariant Borel set.  An \(a_{Z}\)-family
\(\unifm = (d_{K})_{K \in \vietoris{G}}\) of pseudometrics on \(Z\) is said to satisfy the
\emph{\((\mfR, P)\)-Runge property} if for all pairwise disjoint \(K_{1}, \ldots, K_{m} \in \mfR\)
and all \(\epsilon > 0\) there exists a Borel map \(f : P^{(m)} \to Z\) such that
\(d_{K_{i}}(z_{i}, f(x, z_{1}, \ldots, z_{m})) < \epsilon\)
for all \(1 \le i \le m\) and all
\((x, z_{1}, \ldots, z_{m}) \in P^{(m)}\).
\end{definition}

On its face, the \((\mfR, P)\)-Runge property of Definition~\ref{def:runge-property} is a stronger
condition than the \(\mfR\)-Runge property which appears in the main theorem.  Indeed, the former
requires measurability of the \say{Runge map} \(f : P^{(m)} \to Z\) while the latter merely requires
the existence of approximating elements.  However, in the proof of Theorem~\ref{thm:main-theorem} we
show that the \(\mfR\)-Runge property (in the Polish group \(H\)) which appears in
Theorem~\ref{thm:main-theorem} implies the \((\mfR, P)\)-Runge property for a particular choice of
\(P\).

The following lemma provides the key technical statement that ensures the existence of an
equivariant uniformization.  For its formulation, we fix Borel actions \(a_{X} : G \acts X\) and
\(a_{Z} : G \acts Z\) of a locally compact Polish group \(G\).  Suppose that \(a_{X}\) is free.  Let
\(P \subseteq X \times Z\) be a Borel \((a_{X}, a_{Z})\)-equivariant set that admits a Borel
uniformization.  Let \(\mfR \subseteq \vietoris{G}\) be a cofinal invariant class and
\(\unifm = (d_{K})_{K \in \vietoris{G}}\) be an \(a_{Z}\)-family of pseudometrics on \(Z\).  Recall
the notion of a Borel \(\mfR\)-toast from Definition~\ref{def:toast} from the previous section.

\begin{lemma}
  \label{lem:runge-equivariant-uniformization}
  Suppose that \(\unifm\) satisfies the \((\mfR, P)\)-Runge property and assume that \(a_{X}\)
  admits a Borel \(\mfR\)-toast.  If each fiber \(P_{x}\), \(x \in \proj_{X}(P)\), is
  \(\unifm\)-closed, then \(P\) admits an \((a_{X}, a_{Z})\)-equivariant Borel uniformization.
\end{lemma}

\begin{proof}
  By the assumption, the set \(P\) admits a Borel uniformization, which implies that
  \(\proj_{X}(P)\) is a Borel subset of \(X\). Additionally, \(\proj_{X}(P)\) is
  \(a_{X}\)-invariant. Without loss of generality, we may substitute \(\proj_{X}(P)\) for \(X\),
  thereby assuming that \(\proj_{X}(P) = X\).

  Recall that \(a_{X}\) is assumed to be free, which allows us to define the cocycle map
  \(\rho : E_{G}^{X} \to G\) by the condition \(\rho(x, y)x = y\) for \(x E_{G}^{X} y\).  Let
  \((\mathcal{C}_{n}, \lambda_{n})_{n}\) be a Borel \(\mfR\)-toast for \(a_{X}\).

  We inductively construct a sequence of Borel maps \(\xi_{n} : \mathcal{C}_{n} \to Z\) such that
  \((c_{n}, \xi_{n}(c_{n})) \in P\) and the following condition holds:
  \begin{equation}
    \label{eq:xi-condition}
    d_{\lambda_{n-1}(c_{n-1})\rho(c_{n},c_{n-1})}(\xi_{n}(c_{n}),
    \rho(c_{n-1},c_{n})\cdot \xi_{n-1}(c_{n-1})) < 2^{-n}
  \end{equation}
  for all \(c_{n-1} \in \mathcal{C}_{n-1}\) and \(c_{n} \in \mathcal{C}_{n}\) such that \(c_{n-1} \)
  is a predecessor of \(c_{n}\), i.e., \(R_{n-1}(c_{n-1}) \subseteq R_{n}(c_{n})\).

  For the base case, define \(\xi_{0} : \mathcal{C}_{0} \to Z\) to be the restriction of any Borel
  uniformization of \(P\) to \(\mathcal{C}_{0}\). Now, assume that \(\xi_{n-1}\) has already been
  constructed. Choose an element \(c_{n} \in \mathcal{C}_{n}\), and let
  \(c_{n-1}^{1}, \ldots, c_{n-1}^{m} \in \mathcal{C}_{n-1}\) denote all the elements of
  \(\mathcal{C}_{n-1}\) satisfying \(R_{n-1}(c_{n-1}^{i}) \subseteq R_{n}(c_{n})\).  Define compact
  sets \(K_{i} = \lambda_{n-1}(c_{n-1}^{i})\rho(c_{n},c_{n-1}^{i})\) for \(1 \le i \le m\). By the
  invariance of \(\mfR\), each \(K_{i}\) lies in \(\mfR\), and the definition of the toast ensures
  that the sets \(K_{i}\), \(1 \le i \le m\), are pairwise disjoint.

  Let \(f : P^{(m)} \to Z\) be the map provided by the \((\mfR, P)\)-Runge property for the sets
  \(K_{1}, \ldots, K_{m}\) and \(\epsilon = 2^{-n}\). Observe that, for each \(1 \le i \le m\),
  \((c_{n-1}^{i}, \xi_{n-1}(c_{n-1}^{i})) \in P\) implies
  \((c_{n}, \rho(c_{n-1}^{i},c_{n})\cdot \xi_{n-1}(c_{n-1}^{i})) \in P\). We can thus define
  \[
    \xi_{n}(c_{n}) = f(c_{n}, \rho(c_{n-1}^{1},c_{n})\cdot \xi_{n-1}(c_{n-1}^{1}), \ldots,
    \rho(c_{n-1}^{m},c_{n})\cdot \xi_{n-1}(c_{n-1}^{m})).
  \]
  The defining property of \(f\) ensures that
  \[
    d_{K_{i}}(\xi_{n}(c_{n}), \rho(c_{n-1}^{i},c_{n})\cdot \xi_{n-1}(c_{n-1}^{i})) < 2^{-n}
  \]
  for all \(1 \le i \le m\).  Thus, \(\xi_{n}\) satisfies Eq.~\eqref{eq:xi-condition}.  Note that
  item~\eqref{item:toast-rational} of the definition of a Borel toast ensures that there are only
  countably many distinct possibilities for the sets \(K_{i}\), which allows us to perform the
  construction above in a Borel way over all \(\mathcal{C}_{n}\).

  Let \(X_{n} = \bigcup_{c_{n} \in \mathcal{C}_{n}} R_{n}(c_{n})\), and define functions
  \(\beta_{n} : X_{n} \to \mathcal{C}_{n}\) by the condition \(x \in R_{n}(\beta_{n}(x))\) for all
  \(x \in X_{n}\). Next, define functions \(\psi_{n} : X_{n} \to Z\) as
  \begin{equation}
    \label{eq:psi-n-definition}
      \psi_{n}(x) = \rho(\beta_{n}(x), x)\cdot \xi_{n}(\beta_{n}(x)).
  \end{equation}
  Since \((c_{n}, \xi_{n}(c_{n})) \in P\) for all \(c_{n} \in \mathcal{C}_{n}\) and using the
  \((a_{X},a_{Z})\)-equivariance of \(P\), it follows that
  \[
    P \ni \big(\rho(\beta_{n}(x), x)\cdot\beta_{n}(x), \rho(\beta_{n}(x),
    x)\cdot\xi_{n}(\beta_{n}(x))\big) = (x, \psi_{n}(x)).
  \]
  Thus, each \(\psi_{n}\) serves as a Borel uniformization of \(P \cap (X_{n} \times Z)\).

  We claim that the sequence \((\psi_{n}(x))_{n}\) is \(\unifm\)-Cauchy for every \(x \in X\). To
  verify this, it is enough to show that for each \(K \in \vietoris{G}\), every \(x \in X\), and all
  sufficiently large \(n\), the inequality \(d_{K}(\psi_{n}(x), \psi_{n-1}(x)) < 2^{-n}\) holds. Fix
  \(K \in \vietoris{G}\) and define \(c_{n-1} = \beta_{n-1}(x)\) and \(c_{n} = \beta_{n}(x)\). For
  sufficiently large \(n\), Lemma~\ref{lem:exhaustive-taust} gives
  \[
    K\cdot x \subseteq \lambda_{n-1}(c_{n-1})\cdot c_{n-1} \subseteq \lambda_{n}(c_{n})\cdot c_{n}.
  \]
  Since the action \(a_{X}\) is free, it follows that
  \begin{equation}
    \label{eq:containment}
    K\rho(c_{n},x) = K\rho(c_{n-1},x)\rho(c_{n},c_{n-1}) \subseteq
    \lambda_{n-1}(c_{n-1})\rho(c_{n},c_{n-1}).
  \end{equation}
  We now have the following equalities
  \begin{displaymath}
    \begin{aligned}
      d_{K}(\psi_{n}(x), \psi_{n-1}(x))
      &= d_{K}(\rho(c_{n},x)\cdot\xi_{n}(c_{n}), \rho(c_{n-1},x)\cdot\xi_{n-1}(c_{n-1})) \\
      &= d_{K\rho(c_{n}, x)}(\xi_{n}(c_{n}), \rho(x,c_{n})\rho(c_{n-1},x)\cdot\xi_{n-1}(c_{n-1})) \\
      &= d_{K\rho(c_{n}, x)}(\xi_{n}(c_{n}), \rho(c_{n-1},c_{n})\cdot\xi_{n-1}(c_{n-1})) \\
    \end{aligned}
  \end{displaymath}
  which yields the estimate
  \begin{equation}
    \label{eq:psi-n-cauchy}
    \begin{aligned}
      d_{K}(\psi_{n}(x), \psi_{n-1}(x))
      & = d_{K\rho(c_{n}, x)}(\xi_{n}(c_{n}),
        \rho(c_{n-1},c_{n})\cdot\xi_{n-1}(c_{n-1})) \\
      &\le
        d_{\lambda_{n-1}(c_{n-1})\rho(c_{n},c_{n-1})}(\xi_{n}(c_{n}),
        \rho(c_{n-1},c_{n})\cdot\xi_{n-1}(c_{n-1})) \\
      &< 2^{-n},
    \end{aligned}
  \end{equation}
  where the first inequality is due to Eq.~\eqref{eq:containment} and the second follows from
  Eq.~\eqref{eq:xi-condition}.

  We have established that \((\psi_{n}(x))_{n}\) is \(\unifm\)-Cauchy for every \(x \in X\). By
  item~\eqref{item:unif-complete} of Definition~\ref{def:U-family}, the family of pseudometrics
  \(\unifm\) is complete. Consequently, for each \(x \in X\), the limit
  \(\psi(x) = \textrm{\(\unifm\)-}\lim_{n}\psi_{n}(x)\) exists and is unique, as \(\unifm\)
  separates points. Furthermore, \(\psi\) is a Borel map because \(\unifm\) is Borel. To verify
  this, take an exhaustion \((K_{n})_{n}\) of \(G\) by compact sets. The graph of \(\psi\) can be
  written as
  \[
    \Bigl\{(x,y) : \forall k\ \forall m\ \exists N\ (\forall n\ge N)\ [d_{K_{m}}(\psi_{n}(x),y) <
    1/k]\Bigr\},
  \]
  which is Borel since the pseudometrics \(d_{K_{m}}\) and the functions \(\psi_{n}\) are Borel.

  We claim that \(\psi\) is \((a_{X}, a_{Z})\)-equivariant.  Observe that
  \(\rho(y,g\cdot x) = g \rho(y,x)\) for all \(x,y \in X\) and \(g \in G\).  Moreover, for all
  sufficiently large \(n\), we have \(\beta_{n}(g\cdot x) = \beta_{n}(x)\) and consequently
  \begin{displaymath}
    \psi_{n}(g\cdot x) = \rho(\beta_{n}(g\cdot x), g\cdot x) \cdot \xi_{n}(\beta_{n}(g\cdot x)) =
    g\rho(\beta_{n}(x), x)\cdot \xi_{n}(\beta_{n}(x)) = g\cdot \psi_{n}(x).
  \end{displaymath}
  Taking the limit as \(n \to \infty\), we obtain
  \[\psi(g\cdot x) = \unifm\textrm{-}\lim_{n}\psi_{n}(g\cdot x)
    = \unifm\textrm{-}\lim_{n}g\cdot \psi_{n}(x) = g\cdot (\unifm\textrm{-}\lim_{n}\psi_{n}(x)) =
    g\cdot \psi(x),\]
  where the penultimate equality follows from the continuity of the map \(z \mapsto g\cdot z\) in
  the topology of \(\unifm\), given by item~\eqref{item:unif-shift} of
  Definition~\ref{def:U-family}.  Thus, \(\psi \) is \(G\)-equivariant.

  Recall that each \(\psi_{n}\) is a Borel uniformization of \(P \cap (X_{n} \times Z)\).  Since
  slices \(P_{x}\) are assumed to be \(\unifm\)-closed, we conclude that \((x,\psi(x)) \in P\) holds
  for all \( x \in X\) and \(\psi\) is therefore a Borel \((a_{X}, a_{Z})\)-equivariant
  uniformization of \(P\).
\end{proof}

\begin{remark}
  \label{rem:close-to-uniformization}
  The proof of Lemma~\ref{lem:runge-equivariant-uniformization} yields a more refined result that
  can be useful in applications. Given an \(\mfR\)-toast \((\mathcal{C}_{n}, \lambda_{n})_{n}\) for
  \(a_{X}\), the initial map \(\xi_{0} \colon \mathcal{C}_{0} \to Z\) in the construction can be any
  Borel map satisfying \((c_{0}, \xi_{0}(c_{0})) \in P\) for all \(c_{0} \in \mathcal{C}_{0}\). The
  map \(\psi_{0} \colon X_{0} \to Z\), defined by Eq.~\eqref{eq:psi-n-definition}, agrees with
  \(\xi_{0}\) on \(\mathcal{C}_{0}\). If \(K \subseteq \lambda_{0}(c_{0})\) for all
  \(c_{0} \in \mathcal{C}_{0}\) and \(\epsilon > 0\), we can replace the \(2^{-n}\) bound with
  \(\epsilon 2^{-n}\) in the Cauchy estimates ensuring
  \(d_{K}(\psi_{0}(c_{0}), \psi(c_{0})) < \epsilon\) for all \(c_{0} \in
  \mathcal{C}_{0}\). Consequently, the resulting \((a_{X}, a_{Z})\)-equivariant uniformization of
  \(P\) remains \(\epsilon\)-close (with respect to the pseudometric \(d_{K}\)) to the original
  (non-equivariant) map \(\xi_{0}\) on \(\mathcal{C}_{0}\). This observation will be useful below
  in Section~\ref{s:weierstrass-liftings}, where we show how to deduce a Borel counterpart
  of Weiss's result on the existence of measurable entire functions.
\end{remark}

\subsection{Orbital action families}
\label{sec:semidirect-products}

We now examine a specific class of examples of \(a_{Z}\)-families. Consider an action \(H \acts Z\)
of some group \(H\).  When the action is free, each orbit in \(Z\) can be identified with an affine
copy of \(H\). More precisely, for any \(x_{0} \in Z\), the map \(H \ni h \mapsto hx_{0} \in Z\)
establishes a bijection between \(H\) and the \(H\)-orbit of \(x_{0}\).

Given a pseudometric \(d\) on \(H\), we can induce a pseudometric on the orbit of \(x_{0}\) by
defining \(d_{x_{0}}(z_{1}, z_{2}) = d(h_{1}, h_{2})\), where \(h_{i}\) is the unique element in
\(H\) such that \(h_{i}x_{0} = z_{i}\). If we select a different point \(x_{1}\) in the same orbit,
the resulting pseudometric \(d_{x_{1}}\) on the orbit may differ from \(d_{x_{0}}\). Specifically,
if \(fx_{1} = x_{0}\) for some \(f \in H\), then \(d_{x_{1}} = d_{x_{0}}\) holds if and only if
\(d(h_{1}f, h_{2}f) = d(h_{1}, h_{2})\) for all \(h_{1}, h_{2} \in H\).  Thus, the pseudometric
\(d_{x_{0}}\) is independent of the choice of the orbit representative \(x_{0}\) precisely when
\(d\) is right-invariant.

Even when the \(H\)-action is not free, a right-invariant pseudometric \(d\) on \(H\) induces a
pseudometric on each orbit of the action. This is defined by the formula
\[
  d(z_{0}, z_{1}) = \inf \{\norm{h} : hz_{0} = z_{1}\},
\]
where \(\norm{h} = d(h, e)\), and the infimum of an empty set is interpreted as
\(+\infty\). Notably, \(d(z_{0}, z_{1}) < \infty\) holds if and only if \(z_{0} E^{Z}_{H} z_{1}\).

Recall that a seminorm on a group \(H\) is a function \(\norm{\cdot} : H \to \R^{\ge 0}\) satisfying
the following properties for all \(h, h_{1}, h_{2} \in H\):
\begin{itemize}
\item \(\norm{e} = 0\),
\item \(\norm{h} = \norm{h^{-1}}\),
\item \(\norm{h_{1} h_{2}} \le \norm{h_{1}} + \norm{h_{2}}\).
\end{itemize}

Now, suppose \(H\) is equipped with the structure of a standard Borel space (for example, \(H\) is a
Polish group), and let \(\tau : G \acts H\) be a Borel action. Consider a \(\tau\)-family
\(\unifm^{H} = (d^{H}_{K})_{K \in \vietoris{G}}\) of pseudometrics on \(H\). We say that
\(\unifm^{H}\) is \emph{right-invariant} if each pseudometric \(d^{H}_{K}\) satisfies
\(d^{H}_{K}(h_{1}f, h_{2}f) = d^{H}_{K}(h_{1},h_{2})\)
for all \(h_{1}, h_{2}, f \in H\).

Note that right-invariant pseudometrics are in one-to-one correspondence with
seminorms. Specifically, if \(d\) is a right-invariant pseudometric on \(H\), then the function
\(\norm{h} = d(h,e)\) defines a seminorm. Conversely, \(d\) can be recovered from the seminorm via
the formula \(d(h_{1},h_{2}) = \norm{h_{1}h_{2}^{-1}}\). As a result, any right-invariant
\(\tau\)-family of pseudometrics on \(H\) (Definition~\ref{def:U-family}) for the action
\(G \acts H\) can be uniquely determined by specifying a \(\tau\)-family
\(\unifn = (\norm{\cdot}_{K})_{K \in \vietoris{G}}\) of seminorms on~\(H\).  The properties of
pseudometrics given in Definition~\ref{def:U-family} translate into the following properties of the
seminorms \(\unifn\).

\begin{definition}
  \label{def:U-family-seminorms}
  A family of seminorms \(\unifn = (\norm{\cdot}_{K})_{K \in \vietoris{G}}\) on \(H\) is a
  \(\tau\)-family if the following holds for all \(h \in H\) and \(g \in G\).
  \begin{enumerate}[leftmargin=1cm, label={\sf \arabic*}., ref={\sf \arabic*}]
  \item\label{item:semi-unif-monotone} \(\norm{h}_{K} \le \norm{h}_{K'}\) for all
    \(K, K' \in \vietoris{G}\) satisfying \(K \subseteq K'\).
  \item\label{item:semi-unif-shift} \(\norm{\tau^{g}(h)}_{K} = \norm{h}_{Kg}\) for all
    \(K \in \vietoris{G}\).
  \item\label{item:semi-unif-hausdorff} If \(h \ne e\) then \(\norm{h}_{K} \ne 0\) for some
    \(K \in \vietoris{G}\).
  \item\label{item:semi-unif-complete} The right uniformity generated by \( \unifn \) is complete:
    if \((h_{n})_{n}\) is \(\unifn\)-Cauchy in the sense that \(\norm{h_{n}h_{m}^{-1}}_{K} \to 0\)
    as \(m,n \to \infty\), then there exists some \(h_{\infty} \in H\) such that
    \(\norm{h_{\infty}h^{-1}_{n}}_{K} \to 0 \) for all \(K \in \vietoris{G}\).
  \end{enumerate}
\end{definition}

Let \(H\) be a Polish group, and let \(\tau: G \acts H\) be a continuous action by
automorphisms. The semidirect product \(H \rtimes_{\tau} G\) is defined by the multiplication rule
\begin{equation}
  \label{eq:semidirect-product}
  (h_{1}, g_{1}) \cdot (h_{2}, g_{2}) = (h_{1} \tau^{g_{1}}(h_{2}), g_{1}g_{2}),
\end{equation}
and it is Polish in the product topology. Consider a Borel action \(H \rtimes_{\tau} G \acts Z\) on
a standard Borel space \(Z\), and let \(a_{Z}: G \acts Z\) and \(H \acts Z\) denote the
corresponding restrictions of this action.

Suppose we are given a right-invariant \(\tau\)-family, defined by seminorms
\(\unifn = (\norm{\cdot}_{K})_{K \in \vietoris{G}}\) on \(H\). Assume further that \(\unifn\)
generates the Polish topology on \(H\). Under these assumptions, we can construct an
\(a_{Z}\)-family \(\unifm = (d_{K})_{K \in \vietoris{G}}\) of pseudometrics on \(Z\) by defining
\[
  d_{K}(z_{1}, z_{2}) = \inf\bigl\{\norm{h}_{K} : h \in H \text{ satisfies } hz_{1} = z_{2}\bigr\},
\]
where, as usual, the infimum over the empty set is taken to be \(+\infty\).

\begin{lemma}
  \label{lem:invariant-uniformizes}
  Let \(\unifn\) and \(\unifm\) be as above.  Then \(\unifm = (d_{K})_{K \in \vietoris{G}}\) is an
  \(a_{Z}\)-family.  Moreover, \(\unifm\) is Borel if and only if the orbit equivalence relation
  \(E^{Z}_{H}\) is Borel.
\end{lemma}

\begin{proof}
  Item \eqref{item:unif-monotone} of Definition~\ref{def:U-family} follows directly from the
  corresponding property of the seminorms.  For~\eqref{item:unif-shift}, observe that the definition
  of the semidirect product Eq.~\eqref{eq:semidirect-product} gives
  \[hgz_{1} = gz_{2} \iff g\tau^{g^{-1}}(h)z_{1} = gz_{2} \iff \tau^{g^{-1}}(h)z_{1} = z_{2}.\]
  Consequently,
  \begin{displaymath}
    \begin{aligned}
      d_{K}(gz_{1},gz_{2})
      &= \inf\{ \norm{h}_{K} : hgz_{1} = gz_{2}\} \\
      &= \inf\{ \norm{h}_{K} : \tau^{g^{-1}}(h)z_{1} = z_{2}\} \\
      [h' = \tau^{g^{-1}}(h)]
      &= \inf\{ \norm{\tau^{g}(h')}_{K} : h'z_{1} = z_{2}\} \\
      &= \inf\{ \norm{h'}_{Kg} : h'z_{1} =z_{2}\} = d_{Kg}(z_{1},z_{2}). \\
    \end{aligned}
  \end{displaymath}
  The family \(\unifm\) separates points.  To see this, let \((K_{n})_{n}\) be an exhaustion of
  \(G\), and suppose \(d_{K_{n}}(z_{1},z_{2}) = 0 \) for all \(n\).  Then there exist
  \(h_{n} \in H\) such that \(h_{n}z_{1} = z_{2}\) and \(\norm{h_{n}}_{K_{n}} < 2^{-n}\) for all
  \(n\).  The sequence \((h_{n})_{n}\) converges to the identity element \(e\) of \(H\) because
  \(\unifn\) generates the given Polish topology of \(H\).  Since
  \(H_{z_{1},z_{2}} = \{h : hz_{0} = z_{1}\}\) is closed in \(H\) for any Borel action (by Miller's
  theorem~\cite{Kechris}*{9.17} or \cite{MR1619545}*{Thm.~4.8.4}), it follows that
  \(z_{2} = (\lim_{n}h_{n})z_{1} = ez_{1} \), and thus \(z_{1} = z_{2}\).

  To verify completeness (item~\eqref{item:unif-complete} of Definition~\ref{def:U-family}), suppose
  \((z_{n})_{n}\) in \(Z\) is \(d_{K}\)-Cauchy for each \(K \in \vietoris{G}\).  In particular,
  \(z_{m} E^{Z}_{H} z_{n}\) for all sufficiently large \(m, n\).  By passing to a subsequence of
  \((z_{n})_{n}\) if necessary, we may assume that \(d_{K_{m}}(z_{m}, z_{n}) < 2^{-m}\) for all
  \(n \ge m\).  Let \(h_{m} \in H\) satisfy \(\norm{h_{m}}_{K_{m}} < 2^{-m}\) and
  \(h_{m}z_{m} = z_{m+1}\).  The sequence \((h_{n}h_{n-1}\cdots h_{m})_{n}\) is \(\unifn\)-Cauchy.
  Indeed, for all \(m \le k \le n\),
  \begin{displaymath}
    \begin{aligned}
      \norm{(h_{n}h_{n-1}\cdots h_{m}) (h_{k}h_{k-1}\cdots h_{m})^{-1}}_{K}
      &= \norm{h_{n}h_{n-1} \cdots h_{k+1}}_{K} \\
      &\le \sum_{i = k+1}^{n} \norm{h_{i}}_{K} \le \sum_{i=k+1}^{n} 2^{-i} < 2^{-k},
    \end{aligned}
  \end{displaymath}
  where the penultimate inequality holds for \(k\) sufficiently large so that
  \(K \subseteq K_{k+1}\).  Thus, for each \(m\), the limit \(f_{m} = \lim_{n}h_{n}\cdots h_{m}\)
  exists.  Moreover,
  \[\norm{f_{m}}_{K} = \lim_{n}\norm{h_{n} \cdots h_{m}}_{K} \le \sum_{i=m}^{n}\norm{h_{i}}_{K} \le
    2^{-m+1},\]
  provided \(K \subseteq K_{m}\).  In particular, \(\norm{f_{m}}_{K} \to 0\) as \(m \to \infty\) for
  all \(K \in \vietoris{G}\).  Since \(H\) is a topological group in the topology of \(\unifn\),
  \[f_{m}h_{m-1} \cdots h_{0} = (\lim_{n}h_{n} \cdots h_{m})h_{m-1} \cdots h_{0} = \lim_{n}h_{n}
    \cdots h_{0} = f_{0},\]
  and thus \(f_{0}z_{0} = f_{m}h_{m-1}\cdots h_{0}z_{0} = f_{m}z_{m}\).
  The point \(f_{m}z_{m}\) is
  independent of \(m\); let \(z\) denote this common value.  Finally, note that
  \(d_{K}(z_{n},z) \le \norm{f_{n}}_{K} \to 0\) as \(n \to \infty\), thus showing that \(z\) is the
  limit of \((z_{n})_{n}\).

  We have established that \(\unifm\) is an \(a_{Z}\)-family.  If \(\unifm\) is Borel, then
  \[E_{H} = \{(z_{1},z_{2}) : d_{K}(z_{1}, z_{2}) < \infty\}\]
  is also Borel.  Conversely, if \(E_{H}\) is Borel, then by
  Becker--Kechris~\cite{beckerDescriptiveSetTheory1996}*{Thm.~7.1.2}, the map
  \[E_{H} \ni (z_{1},z_{2}) \mapsto H_{z_{1},z_{2}} = \{h \in H : hz_{1} = z_{2}\} \in \effros{H}\]
  is Borel with respect to the Effros Borel structure on \(\effros{H}\).  By the
  Kuratowski--Ryll-Nardzewski selection theorem, there exist Borel functions
  \(s_{n} : \effros{H} \to H\) such that \(\{s_{n}(F)\}_{n}\) is dense in \(F\) for every non-empty
  \(F \in \effros{H}\). For \(z_{1}, z_{2} \in E_{H}\), the pseudometric \(d_{K}\) can then be
  expressed as
  \begin{multline*}
    d_{K}(z_{1},z_{2}) = \alpha \iff \forall n\ \norm{s_{n}(H_{z_{1},z_{2}})}_{K} \ge \alpha
    \textrm{ and }\\ \forall k\ \exists n\ [\, \norm{s_{n}(H_{z_{1},z_{2}})}_{K} < \alpha + 1/k\,].
  \end{multline*}
  This demonstrates that \(d_{K}\) has a Borel graph and is therefore a Borel
  function~\cite{Kechris}*{14.12}, \cite{MR1619545}*{Thm.~4.5.2}.
\end{proof}

\subsection{Equivariant Borel liftings for semidirect product actions}
\label{sec:equivariant-borel-liftings-semiproduct}
We are now ready to state and prove our main theorem. First, let us recall the relevant notation and
assumptions (cf.\ Section~\ref{sec:standing_assumptions}).

Let \(G\) be a locally compact Polish group acting in a Borel way on standard Borel spaces \(Z\) and
\(Y\), and let \(\pi:Z\to Y\) be a Borel \(G\)-equivariant surjection. Recall that our goal is to
provide sufficient conditions under which, for a free Borel \(G\)-action \(G\acts X\) and a
\(G\)-equivariant map \(\phi:X\to Y\), there exists a Borel \(G\)-equivariant lifting
\(\psi:X\to Z\) satisfying \(\pi\circ \psi=\phi\).

Next, let's assume that we have a Polish group action \(H\acts Z\) whose orbit equivalence relation
\(E_H\) is classified by \(\pi\), i.e., \(\pi(z_{1}) = \pi(z_{2})\) if and only if
\(z_{1}E_{H}z_{2}\). We also have a \emph{continuous} action \(\tau : G \acts H\) by automorphisms,
giving rise to a Polish semidirect product \(H \rtimes_{\tau} G\) equipped with the product topology
and group operations
\[
(h_{1},g_{1})(h_{2},g_{2}) = (h_{1}\tau^{g_{1}}(h_{2}), g_{1}g_{2}).
\]
The semidirect product is assumed to act on \(Z\) in a Borel way, inducing actions of the subgroups
\(G\) and \(H\) on \(Z\).  We require that the induced actions coincide with the actions
\(G\acts Z\) and \(H\acts Z\) already present.  (In our applications, this is absolutely automatic.)

We denote by \(\mfR \subseteq \vietoris{G}\) a cofinal \(G\)-invariant class of compact subsets of
\(G\) and let \(\unifn = (\norm{\cdot}_{K})_{K \in \vietoris{G}}\) be a \(\tau\)-family of seminorms
on \(H\) (Definition~\ref{def:U-family-seminorms}).  The following approximation property is a
central requirement for our main theorem.

\begin{definition}
  \label{def:weak-runge-property}
  A \(\tau\)-family \(\unifn\) of seminorms on \(H\) is said to satisfy the \(\mfR\)-Runge property
  if for any pairwise disjoint compact sets \(K_1,\ldots,K_m\in \mfR\), any
  \(h_1,\ldots, h_m\in H\), and any \(\epsilon>0\), there exists an element \(h\in H\) such that
  \(\lVert h h_i^{-1}\rVert_{K_i}<\epsilon\) for \(i=1,\ldots,m\).
\end{definition}

Recall also the definition of a Borel \(\mfR\)-toast from Definition~\ref{def:toast}.  We are now
ready to state our main result.

\begin{theorem}
  \label{thm:main-theorem}

  Assume that the free action \(G \acts X\) admits a Borel \(\mfR\)-toast, and that the
  \(\tau\)-family \(\unifn\) on \(H\) satisfies the \(\mfR\)-Runge property.  Then, for any
  \(G\)-equivariant Borel map \(\phi : X \to Y\), there exists a \(G\)-equivariant Borel map
  \(\psi : X \to Z\) such that \(\pi \circ \psi = \phi\), making the following diagram commute:
  \begin{displaymath}
    \begin{tikzcd}[ampersand replacement=\&]
      G \acts X \arrow[r, dashed, "\psi"] \arrow[dr, "\phi"]
      \& G \acts Z \arrow[d, "\pi"]\\
      \& G \acts Y
    \end{tikzcd}
  \end{displaymath}
\end{theorem}

\begin{proof}
  Let \(P \subseteq X \times Z\) be defined as \(P = \{(x,z) : \phi(x) = \pi(z)\}\). Since \(\phi\)
  and \(\pi\) are equivariant, the set \(P\) is also equivariant. An equivariant uniformization of
  \(P\) will yield the desired map \(\psi\). The existence of such uniformizations will be proved
  using Lemma~\ref{lem:runge-equivariant-uniformization}.

  The map \(\pi\) shows that the action \(H \acts Z\) is concretely classifiable, i.e., its orbit
  equivalence relation is Borel reducible to the equality relation. A theorem of Burgess (see, for
  example, \cite{Kechris}*{18.20iii}) guarantees the existence of a Borel inverse
  \(\delta_{Y} : Y \to Z\) satisfying \(\pi(\delta_{Y}(y)) = y\) for every \(y \in Y\).  Setting
  \(\delta(x) =\delta_{Y}(\phi(x))\), we see that \(P\) admits a Borel uniformization.  Let
  \(\unifm\) be the \(a_{Z}\)-family for \(G \acts Z\) as given in
  Lemma~\ref{lem:invariant-uniformizes}.

  It remains to verify that \(\unifm\) satisfies the \((\mfR, P)\)-Runge property and that the
  slices \(P_{x}\) are \(\unifm\)-closed. The latter is straightforward: if \((x, z_{n}) \in P\) and
  \(z = \unifm\text{-}\lim_{n} z_{n}\), then \(z_{n} E_{H}^{Z} z\) holds for all sufficiently large
  \(n\). In particular, \(\pi(z_{n}) = \pi(z)\) eventually, which implies \((x, z) \in P\) by the
  definition of \(P\).

  We now verify the \((\mfR,P)\)-Runge property. Let \(s : \effros{H} \to H\) be a Borel selector
  provided by the Kuratowski--Ryll-Nardzewski theorem. Define \(\beta : P \to H\) by
  \(\beta(x,z) = s(H_{\delta(x), z})\), where \(H_{z_{1}, z_{2}} = \{h \in H : hz_{1} = z_{2}\}\).
  Note that \(H_{\delta(x),z}\) is non-empty since \(\pi(\delta(x)) = \phi(x) = \pi(z)\) and \(\pi\)
  classifies \(H\)-orbits. Recall the notation
  \(P^{(m)}=\{(x,z_1,\ldots,z_m): \forall i\; (x,z_i)\in P\}\).

  Since \(\unifn\) satisfies the \(\mfR\)-Runge property, for any pairwise disjoint compact sets
  \(K_{1}, \ldots, K_{m} \in \mfR\), any \(\epsilon > 0\), and any \(h_{1}, \ldots, h_{m} \in H\),
  there exists \(h \in H\) such that \(\norm{h h^{-1}_{i}}_{K_{i}} < \epsilon\).  Let
  \(({\mathfrak h}_{n})_{n}\) be a dense sequence in \(H\), and set
  \begin{multline*}
    f(x, z_{1}, \ldots, z_{m}) = {\mathfrak h}_{n}\cdot \delta(x) \textrm{ for the minimal \(n\) such that} \\
    \norm{\beta(x,z_{i}){\mathfrak h}_{n}^{-1}}_{K_{i}} < \epsilon
    \textrm{ for all \(1 \le i \le m\)}.
  \end{multline*}
  We finally claim that the index \(n=n(x,z_1,\ldots,z_m)\) is a Borel function of \(x\) and
  \(z_{i}\), which guarantees Borelness of \(f:P^{(m)}\to Z\) and thus shows the \((\mfR, P)\)-Runge
  property for \(\unifm\).

  To see that \(n:P^{(m)}\to H\) is Borel, note first that for each \(k\) and \(i\), the function
  \((x,z_i)\mapsto \lVert \beta(x,z_i){\mathfrak h}_k^{-1}\rVert_{K_i}\)
  is Borel.  This is because
  all the involved ingredients, that is, the map \(\beta:P\to H\), the seminorms
  \(\lVert \cdot\rVert_{K_i}: H\to \R^{\ge 0}\) and the group operations in \(H\), are Borel. We
  define for each \(k\in\N\) the set
  \[
    \begin{aligned}
      A_k&=\big\{(x,z_1,\ldots,z_m)\in P^{(m)}:
           \lVert \beta(x,z_i){\mathfrak h}_k^{-1}\rVert_{K_i}<\epsilon \text{ for }i=1,\ldots,m\big\}\\
         &=\bigcap_{i=1}^m \big\{(x,z_1,\ldots,z_m)\in P^{(m)}:
           \lVert \beta(x,z_i){\mathfrak h}_k^{-1}\rVert_{K_i}<\epsilon\big\},
    \end{aligned}
  \]
  which, in view of the Borelness of
  \((x,z_i)\mapsto \lVert \beta(x,z_i){\mathfrak h}_k^{-1}\rVert_{K_i}\), is a Borel subset of
  \(P^{(m)}\). The minimal index function \(n\) is nothing but
  \[
    n(x,z_1,\ldots,z_m)=\inf\big\{k\in \N: (x,z_1,\ldots,z_m)\in A_k\big\},
  \]
  and hence for each \(\ell\in\N\), the pre-image \(n^{-1}(\{\ell\})\) can be written as
  \[
    n^{-1}(\{\ell\})=A_\ell\setminus \bigcup_{j=1}^{\ell-1} A_j.
  \]
  Since we just established that all the sets \(A_k\) are Borel, so is the minimal index map
  \(n\). This completes the proof.
\end{proof}

The following corollary is obtained by applying Theorem~\ref{thm:main-theorem} to \(X = \free(Y)\)
and taking \(\phi\) to be the identity map \({\rm id} : \free(Y) \hookrightarrow Y\).

\begin{corollary}
  \label{cor:main-corollary}
  Suppose that the free action \(G \acts \free(Y)\) admits a Borel \(\mfR\)-toast, and that the
  \(\tau\)-family \(\unifn\) on \(H\) satisfies the \(\mfR\)-Runge property.  Then there exists a
  \(G\)-equivariant Borel map \(\psi : \free(Y) \to Z\) satisfying \((\pi \circ \psi)(y) = y\) for
  all \(y \in \free(Y)\).
\end{corollary}

\section{Applications of the main theorem}
\label{sec:appl-main-theor}

We now specialize Theorem~\ref{thm:main-theorem} and present several applications to complex
analysis and PDEs. Sections~\ref{sec:weierstr-mitt-leffl},~\ref{sec:fract-entire-funct},
and~\ref{sec:mitt-leffl} give several applications to entire and meromorphic functions. But first,
we introduce the relevant spaces and maps and establish their measurability properties.

\subsection{Entire and meromorphic functions}
\label{sec:merom-entire-functions}

Let \(\merom\) denote the space of meromorphic functions on \(\C\) (note that we do not consider the
constant \(\infty\) as a meromorphic function).  This space carries a natural Polish topology
\(\tau_{\merom}\), defined as the topology of uniform convergence on compact sets with respect to,
say, the spherical metric on the Riemann sphere \(\Cbar = \C \cup \{\infty\}\) (see, for
instance,~\cite{conwayFunctionsOneComplex1978}*{VII.3}).  The spherical metric on \(\Cbar\) is given
by
\begin{displaymath}
  \begin{aligned}
    \dist(z_{1}, z_{2})
    &= \frac{2|z_{1} - z_{2}|}{\sqrt{(1+|z_{1}|^{2})(1+|z_{2}|^{2})}} \quad
      \textrm{for } z_{1}, z_{2} \in \C \quad \textrm{and} \\
    \dist(z, \infty) &= \frac{2}{\sqrt{1+|z|^{2}}} \quad \textrm{for } z \in \C.
  \end{aligned}
\end{displaymath}
Let \(\meromnz = \merom \setminus \{0\}\) denote the subspace of meromorphic functions which do not
vanish identically. With pointwise operations, \(\merom\) forms an associative division algebra over
\(\C\), and \(\meromnz\) is its multiplicative group of invertible elements.

As observed by Cima and Schober in~\cite{cimaSpacesMeromorphicFunctions1979}, addition of
meromorphic functions is discontinuous with respect to \(\tau_{\merom}\). Furthermore, no comparable
topology makes \(\merom\) a locally convex topological vector
space~\cite{cimaSpacesMeromorphicFunctions1979}*{Prop.~4}.  Multiplication on \(\merom\) is also
discontinuous. To illustrate this, consider the functions \(f_{n}(z) = z - 1/n\) and
\(h_{n}(z) = 1/(z + 1/n)\).  Sequences \((f_{n})_{n}\) and \((h_{n})_{n}\) converge to the functions
\(z \mapsto z\) and \(z \mapsto 1/z\), respectively. If multiplication were continuous, then
\(f_{n} \cdot h_{n}\) would converge to the constant function \(z \mapsto 1\).  However, any
function in a sufficiently small neighborhood of \(z \mapsto 1\) cannot have zeros or poles in, say,
the closed disk \(\Dbar\), whereas \(f_{n} \cdot h_{n}\) has both a zero and a pole in \(\Dbar\) for
each \(n\).

This phenomenon is not due to the choice of topology on \(\merom\) but rather a consequence of its
algebraic structure. In Appendix~\ref{topol-divis-merom}, we use an automatic continuity result due
to Dudley~\cite{dudleyContinuityHomomorphisms1961} to show in Corollary~\ref{cor:no-polish-topology}
that \(\meromnz\) admits no Polish group topologies.  Incidentally, in
Theorem~\ref{thm:no-completely-metrizable-topological-algebra}, we prove that there is no metrizable
topology on \(\merom\) that makes it a complete topological algebra, answering a question raised by
Grosse-Erdmann in~\cite{grosse-erdmannLocallyConvexTopology1995}.  These results, though tangential
to our main focus, justify the generality of the setup in
Section~\ref{sec:equivariant-borel-liftings}.  Had we had the luxury of working with continuous
homomorphisms between Polish groups, some arguments therein could have been significantly simpler.

Nonetheless, algebraic operations on \(\merom\) are Borel.

\begin{proposition}
  \label{prop:meromnz-standard-Borel-algebra}
  \(\merom\) is a standard Borel algebra in the sense that addition, multiplication, and scalar
  multiplication are Borel maps with respect to the Borel \(\sigma\)-algebra of \(\tau_{\merom}\).
\end{proposition}

\begin{proof}
  Since \(\tau_{\merom}\) is finer than the topology of pointwise convergence, the evaluation maps
  \( f \mapsto f(w) \in \Cbar\) are \(\tau_{\merom}\)-continuous for all \(w \in \C\).

  Let \(\C^{*} = \C^{\times} \sqcup \{*\}\) be the quotient of \(\Cbar\), where \(0\) and \(\infty\)
  are identified and denoted by \(*\).  Let \(\pi : \Cbar \to \C^{*}\) be the corresponding
  projection, with \(\pi(0) = * = \pi(\infty)\) and \(\pi(z) = z\) for \(z \in \C^{\times}\).  We
  extend the algebraic operations from \(\C^{\times}\) to all of \(\C^{*}\) by setting
  \( z \cdot * = * = * \cdot z \) and \(z + * = * = * + z\) for all \(z \in \C^{*}\) and
  \(w \cdot * = *\) for all \(w \in \C\).  Choose a countable dense set \(\{\omega_{n}\}_{n} \) of
  complex numbers. Restricted to functions in \(\meromnz\), the graphs of multiplication, addition,
  and scalar multiplication are given by the following expressions:
  \begin{displaymath}
    \begin{aligned}
      &\{(f_{1}, f_{2}, g) : f_{1}f_{2} = g\} = \\
      &\qquad \Bigl\{(f_{1}, f_{2}, g) : \forall n\ \bigl[\textrm{ either }
        (\pi \circ f_{1}(\omega_{n})) \cdot (\pi \circ f_{2}(\omega_{n})) =
        \pi \circ g(\omega_{n}) \textrm{ or } \\
      &\qquad \qquad \pi \circ f_{1}(\omega_{n}) = * \textrm{ or } \pi \circ f_{2}(\omega_{n}) = *
        \bigr]\Bigr\}, \\
      &\{(f_{1}, f_{2}, g) : f_{1} + f_{2} = g\} = \\
      &\qquad \Bigl\{(f_{1}, f_{2}, g) : \forall n\ \bigl[\textrm{ either }
        (\pi \circ f_{1}(\omega_{n}))
        + (\pi \circ f_{2}(\omega_{n})) = \pi \circ g(\omega_{n}) \textrm{ or } \\
      &\qquad \qquad \pi \circ f_{1}(\omega_{n}) = * \textrm{ or } \pi \circ f_{2}(\omega_{n}) = *
        \bigr]\Bigr\}, \\
      &\{(w, f, g) : wf_{2} = g\} =
        \{(w, f, g) : \forall n\ [ w \cdot (\pi \circ f(\omega_{n}))
        = \pi \circ g(\omega_{n})]\}.
    \end{aligned}
  \end{displaymath}
  Since functions with Borel graphs are themselves Borel~\cite{Kechris}*{14.12}, the proposition
  follows.
\end{proof}

The Borelness of the product implies the Borelness of the inverse map on \(\meromnz\), which when
combined with Corollary~\ref{cor:no-polish-topology} leads to the following.

\begin{corollary}
  \label{cor:meromnz-standard-borel-group}
  \(\meromnz\) is a non-Polishable standard Borel group.
\end{corollary}

The subspace of entire functions on \(\C\) is denoted by \(\holom\), and \(\holomnz\) denotes the
subspace of entire functions with no zeros.  Note that \(\holomnz\) consists precisely of the entire
functions that admit entire inverses.  The restriction of \(\tau_{\merom}\) to \(\holom\) coincides
with the topology \(\tau_{\holom}\) of uniform convergence on compact subsets of the complex plane.
\(\holom\) is closed in \(\merom\), and \(\tau_{\holom}\) is Polish.  A basis of neighborhoods of an
entire function \(f\) is parametrized by \(K \in \vietoris{\C}\) and \(\epsilon > 0\), and is given
by the sets
\[\{h \in \holom : \sup_{z \in K} |f(z) - h(z)| < \epsilon\}.\]
The subspace \(\holomnz\) is \(G_{\delta}\) in \(\holom\).  Indeed, Hurwitz's
theorem~\cite{conwayFunctionsOneComplex1978}*{VII.2.5} states that if a sequence of entire functions
\((f_{n})_{n}\) converges uniformly on compact sets to some (necessarily entire) \(f\), then either
\(f\) is identically zero, or each zero of \(f\) is a limit of zeros of \(f_{n}\).  Thus
\(\holomnz \cup \{0\}\) is closed in \(\holom\), hence \(\holomnz\) is \(G_{\delta}\) and therefore
Polish in the induced topology~\cite{Kechris}*{3.11}.

Unlike the whole space of meromorphic functions, addition and multiplication are continuous on
\(\holom\).  In fact, \(\holom\) is a separable Fr\'echet space, and \(\holomnz\) is a
multiplicative Polish group.

\subsection{Spaces of divisors}
\label{sec:divisors}

A \emph{(signed) divisor} on \(\C\) is a map \(d : \C \to \Z\) such that the support
\(d^{-1}(\Z \setminus \{0\}) \) is a closed discrete subset of \(\C\).  The space of divisors can be
defined in two equivalent ways.  First, any divisor can be identified with an atomic Radon measure,
thus viewed as an element of \(\meas{\C}\) (see Section~\ref{sec:radon-meas-distrib}).  The space of
divisors \(\divis\) then corresponds to the following Borel subset of \(\meas{\C}\):
\[\divis = \{\mu \in \meas{\C} : \forall n\
  \exists m \in \Z\ [\mu(U_{n}) = m]\},\]
where \((U_{n})_{n}\) is a countable basis of open bounded subsets of \(\C\).

Alternatively, divisors can be viewed as discrete subsets of \(\C\) labeled by non-zero integers
(see Section~\ref{sec:spaces_of_labeled_discrete_subsets} below).  These two perspectives induce the
same Borel structure on \(\divis\).  The space of divisors inherits the structure of an abelian
group from \(\meas{\C}\).  While it is a Borel group, it is not Polishable, as discussed in
Corollary~\ref{cor:no-polish-topology}.

A \emph{positive divisor} on \(\C\) is a map \(d : \C \to \N\) such that
\(d^{-1}(\N \setminus \{0\})\) is a closed discrete subset of \(\C\).  Equivalently, a positive
divisor is an element \(d \in \divis \cap \measp{\C}\).  The space of positive divisors \(\divisp\)
is closed in \(\measp{\C}\), hence Polish in the induced topology.  A specific metric for this
topology is discussed in~\cite{daleyIntroductionTheoryPoint2003}*{p.~403} (see
also~\cite{morariu-patrichiWeakHashMetricBoundednly2018}).

Poles and zeros of a non-zero meromorphic function \(f\) naturally define a divisor \(\dm(f)\) given
by
\[\dm(f)(z) =
  \begin{cases}
    \phantom{-}m  & \textrm{if \(z\) is a zero of order \(m\) for \(f\),} \\
    - m  & \textrm{if \(z\) is a pole of order \(m\) for \(f\),} \\
    \phantom{-}0  & \textrm{otherwise.} \\
  \end{cases}
\]
The map \(\dm : \meromnz \to \divis\) is a homomorphism,
\(\dm(f_{1}f_{2}) = \dm(f_{1}) + \dm(f_{2})\). Its surjectivity follows from the Weierstrass theorem
on the existence of entire functions with prescribed zeros.  The standard proof of this theorem
yields a Borel function \(\xi : \divis \to \meromnz\) such that \(\dm(\xi(d)) = d\) for all
\(d \in \divis\).  We call any such \(\xi\) a \emph{Borel right-inverse} for \(\dm\).  Moreover, the
argument can be adapted to produce a \emph{continuous} right-inverse \(\xi : \divisp \to \holom\)
(see~\cite{201919}).

The homomorphism \(\dm : \meromnz \to \divis\) is itself Borel.  Indeed, its graph
\[\{(f,d) \in \meromnz \times \divis : \dm(f) = d\}\]
is Borel, as it equals \(\{(f,d) : f/\xi(d) \in \holomnz\}\), where \(\xi\) is any Borel
right-inverse to \(\dm\).

\subsection{The Weierstrass theorem}
\label{sec:weierstr-mitt-leffl}

We can now formulate and prove an equivariant version of the Weierstrass theorem. Recall that
\(\divisp\) denotes the space of positive divisors, and that \(\holom[\ne0]\) denotes the space of
holomorphic functions which do not vanish identically. Both of these are Polish spaces on which the
complex plane acts continuously by argument shifts and translation, respectively.

\begin{theorem}
  \label{thm:equivariant-weierstrass}
  There exists a Borel \(\C\)-equivariant map \(\psi : \free(\divisp) \to \holom[\ne 0]\) such that
  \(\dm\circ \psi={\rm id}_{\free(\divisp)}\).  Furthermore, there exists a Borel \(\C\)-equivariant
  map \(\psi : \free(\divis) \to \meromnz\) such that \(\dm\circ\psi={\rm id}_{\free(\divis)}\).
\end{theorem}

\begin{remark}
  In classical complex analysis, the Weierstrass theorem for meromorphic functions follows from the
  corresponding result for entire functions by taking quotients. A signed divisor \(d\) admits a
  canonical decomposition \(d = d^{+} - d^{-}\), \(d^{+}, d^{-} \in \divisp\), which is
  characterized by \(d^{+}\) and \(d^{-}\) having disjoint supports. One can therefore get the
  equivariant Weierstrass theorem for meromorphic functions on the subspace of those divisors
  \(d \in \divis\), for which both \(d^{+}, d^{-}\) are non-periodic, simply by applying the
  equivariant Weierstrass theorem for entire functions.

  In general, a non-periodic \(d \in \divis\) may have periodic decomposition, as can be seen by
  taking \(d\) to be the difference of two periodic positive divisors with incommensurable periods.
  This can be circumvented by taking a decomposition of the form
  \(d = (d^{+} + d') - (d^{-} + d') \) for a suitably chosen positive divisor \(d'\), which can be
  selected in a Borel way for each \(d\). Instead of this, we will construct the right-inverse
  \(\psi : \free(\divis) \to \meromnz\) via a direct application of
  Corollary~\ref{cor:main-corollary}, and obtain the theorem for entire functions by restricting to
  the positive divisors.
\end{remark}

\begin{proof}
  To obtain Theorem~\ref{thm:equivariant-weierstrass}, we apply Corollary~\ref{cor:main-corollary}
  in the following context.  The group \(G\) is the additive group of \(\C\), and the class
  \(\mfR = \mfD_{2}\) consists of compact subsets of \(\C\) which are diffeomorphic to the closed
  unit disk.  Recall that every free Borel \(\C\)-action admits a Borel \(\mfD_{2}\)-toast
  (Section~\ref{sec:toast}).

  Let \(H = \holomnz\) be the multiplicative group of entire functions without zeros, let \(\C \)
  act on \(\holomnz\) via the argument shift, and the semidirect product
  \( H \rtimes G = \holomnz \rtimes \C\) acts on \(Z = \meromnz\) according to the rule
  \[((f, z) \cdot h)(w) = f(w+z)h(w+z).\]
  The family of seminorms \(\unifn = (\norm{\cdot})_{K \in \vietoris{\C}} \) on \(\holomnz\) is
  given by
  \[\begin{aligned}\norm{f}_{K} &= \max\{\log (\sup_{w \in K}|f(w)|) , -\log(\inf_{w \in K}|f(w)|)\}
    \\& = \sup_{w \in K}\big|\log|f(w)|\big|.\end{aligned}\]
We claim that it satisfies the \(\mfD_{2}\)-Runge property.  Indeed, consider pairwise disjoint sets
\(K_{1}, \ldots, K_{m} \in \mfD_{2}\) and nowhere zero entire functions
\(f_{1}, \ldots, f_{m} \in \holomnz\).  Note that \(\C \setminus (K_{1} \cup \cdots \cup K_{m}) \)
is connected.
Since \(\C\) is connected and simply connected, there exist entire functions
\(g_{i}\) such that \(f_{i} = e^{g_{i}}\).  By the Runge approximation theorem, for any
\(\delta > 0\), there exists an entire function \(g\) (in fact, a polynomial) satisfying
\[
\sup_{z \in K_{i}}|g(z) - g_{i}(z)| < \delta
\]
for all \(i\).  Setting \(f = e^{g}\), we observe
that \(\norm{ff^{-1}_{i}}_{K_{i}} < \epsilon\) for some \(\epsilon = \epsilon(\delta, K_{i})\),
where \(\epsilon(\delta, K_{i}) \to 0\) as \(\delta \to 0\).

The group \(\C\) also acts on the space of divisors \(Y = \divis\) by translations, and the map
\(\dm : \meromnz \to \divis\) satisfies \( f_{1} E_{\holomnz} f_{2} \) if and only if
\(\dm(f_{1}) = \dm(f_{2})\). Consequently, Theorem~\ref{thm:main-theorem} guarantees the existence
of the required map \(\psi\) and the proof of Theorem~\ref{thm:equivariant-weierstrass} is complete.
\end{proof}

\subsection{Weiss's theorem for Borel entire functions}
\label{s:weierstrass-liftings}

As mentioned in the introduction, Weiss's result \cite{MR1422707} on the existence
of measurable entire functions readily adapts to the Borel category.

\begin{theorem}
\label{thm:Weiss-Borel}
Given a free Borel action \(\C\acts X\) on a standard Borel space, there exists
a Borel \(\C\)-equivariant map \(\psi:X\to \mathcal{E}\setminus \{\text{constants}\}\).
\end{theorem}

\begin{proof}
Using Theorem~\ref{thm:main-theorem} instead of
Corollary~\ref{cor:main-corollary} in the proof of
Theorem~\ref{thm:equivariant-weierstrass}, one shows that for any free Borel action \(\C \acts X\)
and any equivariant Borel map \(\phi : X \to \divis\) there exists a Borel equivariant lifting
\(\psi : X \to \meromnz\) such that \(\dm \circ \psi = \phi\).  Let \(\mathcal{C} \subseteq X\) be
a Borel set that intersects each orbit in a countable discrete set.  With such \(\mathcal{C}\), we
can associate a Borel equivariant map \(\phi_{\mathcal{C}} : X \to \divisp\) given by
\begin{displaymath}
\phi_{\mathcal{C}}(x)(z) =
\begin{cases}
1 & \textrm{if \((-z) \cdot x \in \mathcal{C}\)}, \\
0 & \textrm{otherwise}. \\
\end{cases}
\end{displaymath}
A Borel equivariant lifting \(\psi : X \to \holom\) is then a Borel entire function.  Furthermore,
if \(\mathcal{C}\) has a non-empty intersection with each orbit, then \(\psi\) is necessarily
non-constant, because it has zeros within each orbit.  This establishes the existence of
non-constant Borel entire functions for any free Borel action \(\C \acts X\), which completes the proof.
\end{proof}

An alternative derivation of Weiss's result arises by considering the trivial one-point space
\(Y = \{*\}\) in the lifting diagram of Theorem~\ref{thm:main-theorem}. In this case, both
\(\phi\) and \(\pi\) degenerate to trivial maps, establishing a bijection between all Borel entire
functions and liftings \(\psi\) in the diagram:
\begin{displaymath}
\begin{tikzcd}[ampersand replacement=\&]
\C \acts X \arrow[r, dashed, "\psi"] \arrow[dr, "\phi"]
\& \C \acts \holom \arrow[d, "\pi"]\\
\& \C \acts \{*\}
\end{tikzcd}
\end{displaymath}

While Theorem~\ref{thm:main-theorem} guarantees the existence of such liftings, it does not
immediately ensure the existence of non-constant
\(\psi : X \to \holom \setminus \{\textrm{constants}\}\). To address this,
recall that the first step of the proof of Theorem~\ref{thm:main-theorem} is to construct
(using the Burgess theorem in the general case) a Borel uniformization \(\delta : X \to Z\)
for the set \(P = X\times Z\). This is then modified to a Borel equivariant uniformization
with the aid of Lemma~\ref{lem:runge-equivariant-uniformization}.
As discussed in Remark~\ref{rem:close-to-uniformization},
if a toast \((\mathcal{C}_n,\lambda_n)_{n}\) for \(G\acts X\) is taken such that
\(K \subseteq \lambda_{0}(c_{0})\), \(c_{0} \in \mathcal{C}_{0}\), and \(\epsilon > 0\), then the
lifting \(\psi\) produced by Theorem~\ref{thm:main-theorem} can be assumed to satisfy
\begin{equation}
\label{eq:psi-epsilon-close}
\norm{\rho\bigl(\delta(c_{0}), \psi(c_{0})\bigr)}_{K} < \epsilon
\end{equation}
for \(c_{0} \in \mathcal{C}_{0}\). We take \(\delta_{Y} : Y = \{*\} \to \holom\) to be any
non-constant function, \(f = \delta_{Y}(*)\), choose \(K = \Dbar\) to be the closed unit disk, and
select \(\epsilon > 0\) sufficiently small such that for any \(h \in \holom\)
\[\max_{z \in \Dbar}|f(z) - h(z)| < \epsilon \implies h \textrm{ is not constant}.\]
By \eqref{eq:psi-epsilon-close}, there exists a lifting \(\psi\) satisfying
\[
  \sup_{z \in \Dbar}|\psi(c_{0})(z) - f(z)| < \epsilon, \quad \textrm{for } c_{0} \in \mathcal{C}_{0}.
\]
If \(\mathcal{C}_{0}\) intersects every orbit of the action \(\C \acts X\) (see
Remark~\ref{rem:cross-section-re-indexing}), then \(\psi : X \to \holom\) has no constants in its
range.

This argument, when expressed directly in terms of Borel entire functions, closely parallels
Weiss's original proof. The argument can be adapted to prove similar results for the other applications
of Theorem~\ref{thm:main-theorem} below.

\subsection{Fractions of entire functions}
\label{sec:fract-entire-funct}

Every meromorphic function can be expressed as a quotient of two entire functions. While this
representation is non-unique even when requiring the divisors of the entire functions to have
disjoint supports, one may ask whether such a representation can be chosen in an equivariant
manner. An application of Theorem~\ref{thm:main-theorem} yields an affirmative answer when we
consider aperiodic meromorphic functions.

\begin{theorem}
  \label{thm:merormophic-quotient-entire}
  Let \(Z\) denote the space of pairs of entire functions with disjoint divisors:
  \[ Z = \{(f_{1},f_{2}) \mid f_{i} \in \holom[\ne 0],\, \supp(\dm(f_{1})) \cap \supp(\dm(f_{2})) =
    \varnothing \}. \] There exists a Borel \(\C\)-equivariant map
  \[ \psi : \free(\meromnz) \to Z, \quad \psi(g) = (\psi_{1}(g), \psi_{2}(g)), \]
  such that \( g = \psi_{1}(g)/\psi_{2}(g) \) for all \( g \in \free(\meromnz) \).
\end{theorem}

\begin{proof}
  The proof strategy mirrors that of Theorem~\ref{thm:equivariant-weierstrass}. We apply
  Corollary~\ref{cor:main-corollary} with the following parameters: \( G = \C \),
  \( \mfR = \mfD_{2} \), and \( H = \holomnz \), equipped with the same family of seminorms
  \( \unifn \). The target space is \( Y = \meromnz \), and \( \pi : Z \to Y \) is the quotient map
  \( \pi(f_{1},f_{2}) = f_{1}/f_{2} \). The group \( G = \C \) acts via argument shifts, while
  \( H = \holomnz \) acts on \( Z \) by multiplication:
  \( h \cdot (f_{1}, f_{2}) = (hf_{1}, hf_{2}) \).

  For any two pairs \((f_{1}, f_{2})\) and \((\tilde{f}_{1}, \tilde{f}_{2})\) in \(Z\) satisfying
  \(f_{1}/f_{2} = \tilde{f}_{1}/\tilde{f}_{2}\), there exists a function
  \(h = f_{1}/\tilde{f}_{1} = f_{2}/\tilde{f}_{2}\) such that
  \( (f_{1}, f_{2}) = h \cdot (\tilde{f}_{1}, \tilde{f}_{2})\). The divisors' disjointness condition
  implies \(h \in \holomnz\), so Corollary~\ref{cor:main-corollary} applies, yielding the desired
  equivariant map \(\psi\).
\end{proof}

The restriction to the free part of \(\meromnz\) in Theorem~\ref{thm:merormophic-quotient-entire}
cannot be removed, as explained in Remark~\ref{rem:yosida-type-normal-meromorphic}.

\subsection{Spaces of labeled discrete subsets}
\label{sec:spaces_of_labeled_discrete_subsets}

Our next application will be a Borel equivariant analogue of the Mittag-Leffler theorem for
meromorphic functions.  Before we can formulate and prove it, we need to establish several
properties of the space of principal parts, which we can think of as labeled discrete subsets of
\(\C\).

Recall that \(\effros{\C}\) denotes the Effros Borel space of closed subsets of \(\C\).  Let
\(\effrosdis{\C} \subseteq \effros{\C}\) denote the subspace of closed discrete subsets of \(\C\),
i.e., those \(F \in \effros{\C}\) that have finite intersection with every compact subset of
\(\C\). This is a Borel subset of \(\effros{\C}\).
Indeed, \(\vietoris{K}\)
is a Borel subset of \(\effros{\C}\) for each \(K \in \vietoris{\C}\),
and the set \(\vietorisfin{K}\) of finite subsets of \(\vietoris{K}\) is Borel\footnote{
In fact, the set of finite subsets of size at most \(n\) is closed,
because a set has at least \((n+1)\)-many points if and only if it intersects a collection of
\((n+1)\)-many pairwise disjoint open sets.}.
The intersection map
\[
\effros{\C}^{2} \ni (F_{1}, F_{2}) \mapsto F_{1} \cap F_{2} \in \effros{\C}
\]
is Borel due to the local compactness of \(\C\)\footnote{More specifically,
let \((V_{n})_{n}\) be a precompact basis for the
topology on \(\C\). For each \(n\), pick a sequence \((\mathcal{U}_{n,m})_{m}\),
\(\mathcal{U}_{n,m} = (U^{i}_{n,m})_{i=1}^{p_{n,m}}\), of finite open covers of \(\overline{V}_{n}\) by
sets \(U_{n,m}^{i}\) satisfying \(\diam (U^{i}_{n,m}) < 1/m\) for all \(1 \le i \le p_{n,m}\).  Then for an open
\(U\), we have \(F_{1} \cap F_{2} \cap U \ne \varnothing \) if and only if there exists \(n\) such
that \(\overline{V}_{n} \subseteq U\) and
\[\forall m\ \exists 1 \le i \le p_{m,n}\ \bigl[ F_{1} \cap U_{n,m}^{i} \ne \varnothing
  \textrm{ and } F_{2} \cap U_{n,m}^{i} \ne \varnothing\bigr].\]
}.
Finally, if \((K_{n})_{n}\) is an exhaustion of \(\C\) by compact sets,
then \(F \in \effrosdis{\C}\) if and
only if \(F \cap K_{n} \in \vietorisfin{K_{n}}\) for all \(n\). Thus,
\(\effrosdis{\C}\) is indeed a Borel subset of \(\effros{\C}\).

We now consider a generalization of \(\effrosdis{\C}\) where points of discrete subsets of \(\C\)
are labeled.  This arises naturally in the context of the space of principal parts of meromorphic
functions, where the set of poles can be viewed as an element of \(\effrosdis{\C}\), and each pole
is labeled by the finite tuple of coefficients determining the corresponding principal part.

One way to define the required space would be as follows.  Let \(Y\) be a standard Borel space.
Choose a Polish topology on \(Y\) compatible with its Borel structure, and consider the space
\(\effros{\C \times Y}\).  The projection \(\proj_{\C} : \C \times Y \to \C\) is continuous, so the
map
\[
\effros{\C\times Y} \ni F \mapsto \overline{\proj_{\C}(F)} \in \effros{\C}
\]
is Borel.  In fact, \(\proj_{\C}(F) \cap U \ne \varnothing\) if and only if
\(F \cap (U \times Y) \ne \varnothing\). Also, note that
\(\overline{\proj_{\C}(F)} = \proj_{\C}(F)\) whenever
\(\overline{\proj_{\C}(F)} \in \effrosdis{\C}\), since elements of \(\effrosdis{\C}\) have no proper
dense subsets.  In particular, the set
\[
\{F \in \effros{\C \times Y} : \proj_{\C}(F) \in \effrosdis{\C}\}
\]
is Borel.  We need to pass to a further subset that consists of those \(F \in \effros{\C\times Y}\)
for which the projection map is injective.  Let \(s_{n} : \effros{\C\times Y} \to \C \times Y\),
\(n \in \N\), be Kuratowski--Ryll-Nardzewski Borel selectors (as in \cite{Kechris}*{12.13}), and
define
\begin{multline*}
  \effrosdis{\C; Y} = \{\varnothing\} \cup \Bigl\{F \in \effros{\C \times Y} :
  \proj_{\C}(F) \in \effrosdis{\C} \textrm{ and } \\
  \forall n\ \forall m\ \bigl[ \proj_{\C}(s_{n}(F)) = \proj_{\C}(s_{m}(F)) \implies s_{n}(F)
  = s_{m}(F) \bigr]\Bigr\}.
\end{multline*}
The condition \(\proj_{\C}(F) \in \effrosdis{\C}\) ensures that
\(\{\proj_{\C}(s_{n}(F))\}_{n} = \proj_{\C}(F)\). Thus, \(F \in \effros{\C \times Y}\) belongs to
\(\effrosdis{\C; Y}\) if and only if \(\proj_{\C}(F) \in \effrosdis{\C}\) and the projection map
\(\proj_{\C} : F \to \C\) is injective.  This allows us to view \(\effrosdis{\C; Y}\) as the
collection of discrete subsets of \(\C\) whose points are labeled by elements of~\(Y\).  Note that
if \(Y = \{*\}\) consists of a single point, then \(\effrosdis{\C; \{*\}}\) is naturally isomorphic
to \(\effrosdis{\C}\).

The Borel structure on \(\effrosdis{\C; Y}\) is independent of the choice of a compatible Polish
topology on \(Y\).  This can be verified directly or deduced from the following alternative
presentation of \(\effrosdis{\C;Y}\).  Let \((K_{n})_{n}\) be an exhaustion of \(\C\) by compact
sets.  Define
\begin{displaymath}
  \begin{aligned}
    \C \lprod_{n} Y
    &= \{(x_{k}, y_{k})_{k < n} \in (\C \times Y)^{n} : \forall k < n\ \forall m <n\
      [ m \ne k \implies x_{k} \ne x_{m}]\},\\
    \C \lprod_{\infty} Y
    &= \bigl\{(x_{k},y_{k})_{k} \in (\C \times Y)^{\infty} : \forall k\  \forall m\
      [ m \ne k \implies x_{k} \ne x_{m}] \textrm{ and } \\
    &\qquad \qquad \qquad \qquad \qquad \qquad \qquad \qquad \qquad
      \forall n\ \exists N\ \forall k \ge N\ [x_{k} \not \in K_{n}] \bigr\}.
  \end{aligned}
\end{displaymath}
Each sequence in \( \C \lprod_{n} Y\) provides an injective enumeration of an \(n\)-element set
\(\{x_{k}\}_{k<n}\), with \(y_{k}\) as the label of \(x_{k}\).  Sequences in
\( \C \lprod_{\infty} Y\) enumerate infinite discrete subsets of \(\C\), and every discrete subset
of \(\C\) is enumerated by some sequence in
\( \C \lprod Y = \C \lprod_{\infty} Y \sqcup \bigl( \bigsqcup_{n} \C \lprod_{n} Y \bigr)\).

Two sequences of the same length, say \((x_{k},y_{k})_{k}\) and \((x'_{k}, y'_{k})_{k}\), encode the
same labeled set if and only if \((x'_{k},y'_{k})_{k}\) is a permutation of \((x_{k},y_{k})_{k}\).
For \(\kappa \in \N \cup \{\infty\}\), let \(S_{\kappa}\) denote the group of permutations of a
\(\kappa\)-element set.  The space of labeled discrete \(\kappa\)-element subsets of \(\C\) can be
identified with the factor space of the action \(S_{\kappa} \acts \C \lprod_{\kappa} Y\) via
coordinate permutations.

The action of \(S_{\kappa}\) on \(\C \lprod_{\kappa} Y\) is concretely classifiable.  To see this,
choose a Borel linear order~\(\prec\) on \(\C\) and let
\(T^{\kappa} \subseteq \C \lprod_{\kappa} Y \) consist of those
\((x_{k},y_{k})_{k<\kappa} \in \C \lprod_{\kappa} Y\) satisfying
\begin{multline*}
  \forall n\ \forall m < \kappa\ \forall k < m\ \bigl[ x_{m} \in K_{n} \implies x_{k} \in K_{n}
  \textrm{ and } \\ x_{k},x_{m} \in K_{n}\setminus K_{n-1} \implies x_{k} \prec x_{m}\bigr].
\end{multline*}
In other words, the sequence \((x_{k},y_{k})_{k < \kappa}\) is in \(T^{\kappa}\) if it lists
elements \(\{x_{k}\}\) of \(K_{0}\) first, followed by those in \(K_{1}\setminus K_{0}\),
\(K_{2}\setminus K_{1}\), etc., with elements in each block ordered by \(\prec\).  The set
\(T^{\kappa}\) is a Borel transversal for \(S_{\kappa} \acts \C \lprod_{\kappa} Y\), so the factor
space \(\C \lprod_{\kappa} Y / E_{S_{\kappa}}\) is standard Borel.

Let \(E_{S}\) be the union of the orbit equivalence relations \(E_{S_{\kappa}}\): for
\(z_{1}, z_{2} \in \C \lprod Y\),
\[z_{1} E_{S} z_{2} \iff \exists \kappa \in \N \cup \{\infty\}\ \bigl[ z_{1},z_{2} \in \C
  \lprod_{\kappa} Y \textrm{ and } z_{1} E_{S_{\kappa}} z_{2} \bigr].\]
The factor space \(\C \lprod Y/E_{S}\) is standard Borel and encodes discrete subsets of \(\C\)
whose points are labeled by elements of \(Y\).

The spaces \(\C \lprod Y/E_{S}\) and \(\effrosdis{\C;Y}\) are Borel isomorphic.  To see this, let
\[
\pi : \C \lprod Y \to \effrosdis{\C;Y}
\]
be the map sending a sequence \((x_{k},y_{k})_{k < \kappa} \in \C \lprod Y\) to the set
\(\{(x_{k},y_{k})\}_{k < \kappa} \in \effros{\C \times Y}\).  Since for any open
\(U \subseteq \C \times Y\) (in fact, for any Borel \(U \subseteq \C \times Y\)), the set
\[\{(x_{k},y_{k})_{k < \kappa} \in \C \lprod_{\kappa} Y : \exists k < \kappa\ (x_{k},y_{k}) \in
  U\}\]
is Borel, the map \(\pi\) is also Borel.  Moreover, \(\pi : \C \lprod Y \to \effrosdis{\C;Y}\) is
surjective, and \(\pi((x_{k},y_{k})_{k < \kappa}) = \pi((x'_{k}, y'_{k})_{k < \kappa'})\) if and
only if \((x_{k},y_{k})_{k < \kappa}E_{S} (x'_{k},y'_{k})_{k < \kappa'}\).  Thus, \(\pi\) factors
into a Borel injection from \(\C \lprod Y/E_{S}\) onto \(\effrosdis{\C;Y}\), which must be an
isomorphism between the standard Borel spaces~\cite{Kechris}*{15.2}.

\subsection{Principal parts and the Mittag-Leffler theorem}
\label{sec:mitt-leffl}

We next introduce the space \(\ppart \), which describes the principal parts of meromorphic
functions, and establish some of its basic properties.  Near a pole \(w\) of order \(m\), a
meromorphic function \(f\) admits an expansion of the form
\[f(z) = \frac{c_{1}}{z-w} + \frac{c_{2}}{(z-w)^{2}} + \cdots + \frac{c_{m}}{(z-w)^{m}} +
  \sum_{n=0}^{\infty}a_{n}(z-w)^{n} = p_{w}(z) + \sum_{n=0}^{\infty}a_{n}(z-w)^{n},\]
where \(p_{w}(z)\) is the \emph{principal part} at \(w\), encoded by the \(m\)-tuple
\((c_{1}, \ldots, c_{m})\).  Let \( \C^{<\infty} = \bigsqcup_{n = 1}^{\infty} \C^{n}\) denote the
space of finite non-empty sequences of complex numbers.  The \emph{space of principal parts}
\(\ppart = \effrosdis{\C; \C^{<\infty}}\) consists of discrete subsets of \(\C\) labeled by finite
sequences of complex numbers.

Let \(\ppm : \merom \to \ppart\) be the map associating each meromorphic function \(f\) with the
labeled set of principal parts at its poles.  This map is Borel.  To see this, it is best to view
\(\ppart = \effrosdis{\C; \C^{<\infty}}\) as the quotient space of \(\C \lprod \C^{<\infty}\), as
described at the end of Section~\ref{sec:spaces_of_labeled_discrete_subsets}.  Consider the set
\(A\) of pairs \((f, (w_{k}, y_{k})_{k})\), where \(f\) is a meromorphic function and
\((w_{k}, y_{k})_{k} \) is a finite or infinite sequence with pairwise distinct \(w_{k} \in \C\) and
labels \(y_{k} \in \C^{<\infty}\) such that \((w_{k})_{k}\) enumerates the poles of \(f\) and
\(y_{k} \in \C^{<\infty}\) encodes the principal part of \(f\) at \(w_{k}\).  To show that \(\ppm\)
is Borel, it suffices to verify that \(A\) is Borel.  For \((w_{k},y_{k})\) with
\(y_{k} = (c^{k}_{1}, \ldots, c^{k}_{m_{k}}) \in \C^{<\infty}\), define
\[p_{k}(z) = \frac{c^{k}_{1}}{z-w_{k}} + \cdots + \frac{c^{k}_{m_{k}}}{(z-w_{k})^{m_{k}}}.\]
The condition \(f(n\Dbar) \subseteq \C\) is equivalent to \(f\) having no poles in \(n\Dbar\), which
is an open condition in \(\tau_{\merom}\).  Finally, \((f, (w_{k}, y_{k})_{k}) \in A\) if and only
if for each \(n\)
\[ \bigl(f - \sum_{\mathclap{w_{k} \in n\Dbar}} p_{k} \bigr)(n\Dbar) \subseteq \C,\]
which is Borel since the algebraic operations in \(\merom\) are Borel by
Proposition~\ref{prop:meromnz-standard-Borel-algebra}.

Mittag-Leffler's theorem on the existence of meromorphic functions with prescribed principal parts
implies that \(\ppm : \merom \to \ppart\) is surjective.  We summarize these results as follows:
\begin{proposition}
  \label{prop:dvs-and-ppm-are-Borel-surjective}
  The map \(\ppm : \merom \to \ppart\) is a Borel surjection.
\end{proposition}

We are now ready to state and prove the equivariant Mittag-Leffler theorem.
\begin{theorem}
  \label{thm:mittag-leffler}
  There exists a Borel \(\C\)-equivariant map \(\psi : \free(\ppart) \to \merom\) that is a
  right-inverse to \(\ppm\).
\end{theorem}

\begin{proof}
  Just as in the proof of Theorem~\ref{thm:equivariant-weierstrass}, we apply
  Corollary~\ref{cor:main-corollary} in the following context.  The group \(G\) is the additive
  group of \(\C\), and the class \(\mfR = \mfD_{2}\) consists of compact subsets of \(\C\)
  diffeomorphic to the unit disk. Note that every free Borel \(\C\)-action admits a Borel
  \(\mfD_{2}\)-toast (Section~\ref{sec:toast}).

  We next let \(H = \holom\) denote the additive group of entire functions, and let \(\C \) act on
  \(\holom\) via the argument shift.  Define \(Z\) to be the space of meromorphic
  functions~\(\merom\). The group \(\holom\) acts on \(Z\) additively,
  \( (h \cdot f)(w) = f(w) + h(w)\).  The semidirect product \(H \rtimes G = \holom \rtimes \C\)
  acts on \(Z\) as
  \[((h,z) \cdot f)(w) = f(w+z) + h(w+z).\]
  The family of seminorms \(\unifn = (\norm{\cdot}_{K})_{K \in \vietoris{\C}}\) on \(\holom\) is
  given by \(\norm{f}_{K} = \sup_{z \in K} |f(z)|\).  It satisfies the \(\mfD_{2}\)-Runge property
  in view of the standard Runge theorem.  The space \(Y = \ppart\) is the space of principal parts
  and \(\pi = \ppm\).  The equivariant Mittag-Leffler theorem is then an instance of
  Corollary~\ref{cor:main-corollary}.
\end{proof}

\subsection{Luzin spaces and LF-spaces}
\label{sec:lusin-space}

Our next application uses the notion and basic properties of Luzin spaces.  A \emph{Luzin space} is
a Hausdorff topological space \((X,\tau)\) whose topology can be refined into a Polish topology.  A
systematic treatment of Luzin spaces can be found in
\cite{schwartzRadonMeasuresArbitrary1973}*{Ch.~II}.  A key property for our purposes is that the
Borel \(\sigma\)-algebras of Luzin spaces are standard.  Specifically, if \((X,\tau)\) is a Luzin
space and \((X,\tau')\) is a Polish refinement of \(\tau\), then
\(\borel_{\tau} = \borel_{\tau'}\)~\cite{schwartzRadonMeasuresArbitrary1973}*{p.~101}, where
\(\borel_{\tau}\) and \(\borel_{\tau'}\) are the \(\sigma\)-algebras generated by the open sets in
the corresponding topologies.

As observed in \cite{schwartzRadonMeasuresArbitrary1973}*{p.~90}, most of the separable topological
spaces encountered in analysis are Luzin.  Recall that a \emph{Fr\'echet space} is a completely
metrizable locally convex topological vector space.  An \emph{LF-space} is a locally convex
topological vector space \(E\) that admits an increasing and exhaustive sequence of subspaces,
\(E_{0} \subseteq E_{1} \subseteq \cdots \subseteq E = \bigcup_{n} E_{n}\) , where each \(E_{n}\) is
a Fr\'echet space in the induced topology, and the topology of \(E\) coincides with the inductive
limit topology of \((E_{n})_{n}\)~\cite{nariciTopologicalVectorSpaces2010}*{p.~428}.  (Sometimes the
term strict LF-space is used instead, while for general LF-spaces, the requirement on \(E_{n}\) is
relaxed to being Fr\'echet in a topology coarser than the induced one.)

\begin{theorem}(Schwartz \cite{schwartzRadonMeasuresArbitrary1973}*{p.~112 Thm.~7})
  \label{thm:schwartz}
  If \(E\) and \(F\) are separable LF-spaces, then \(\mathcal{L}_{c}(E,F)\)---the space of
  continuous linear operators \(E \to F\) equipped with the compact-open topology---is a Luzin
  space.
\end{theorem}

\begin{remark}
  Theorem~\ref{thm:schwartz} is stated in~\cite{schwartzRadonMeasuresArbitrary1973}*{p.~112 Thm.~7}
  in a more general form, where \(E\) is assumed to be the inductive limit of a sequence of
  separable Fr\'echet spaces \(E_{n}\) satisfying:
  \begin{enumerate}[leftmargin=1cm, label={\sf \arabic*}., ref={\sf \arabic*}]
  \item \(E_{n}\) increases with \(n\), and
  \item every compact subset of \(E\) is a compact subset of \(E_{n}\), for some \(n\).
  \end{enumerate}
  When \(E\) is a separable LF-space, these conditions are automatically satisfied
  by~\cite{nariciTopologicalVectorSpaces2010}*{Thm.~12.1.7(c)}.  Additionally, the separability of
  \(E\) implies the separability of the Fr\'echet subspaces
  \(E_{n}\)~\cite{vidossichCharacterizationSeparabilityLF1968} (see
  also~\cite{lohmanSeparabilityLinearTopological1974} for a more general argument).  Furthermore,
  the conditions on \(F\) can be relaxed to requiring that \(F\) is a Hausdorff topological vector
  space that is the countable union of images (under linear continuous maps) of separable Fr\'echet
  spaces.

  For the purposes of this work, the generality of Theorem~\ref{thm:schwartz} is sufficient.
\end{remark}

\begin{corollary}
  \label{cor:lf-dual-lf-standard}
  The Borel \(\sigma\)-algebras of separable LF-spaces and their weak* duals are standard.
\end{corollary}
\begin{proof}
  Let \(F\) be the field of scalars.  An LF-space \(E\) and its dual \(E^{*}\) with the weak*
  topology can be viewed as spaces of linear operators \(\mathcal{L}(F,E)\) and \(\mathcal{L}(E,F)\)
  endowed with the pointwise convergence topology.  The topology of pointwise convergence is coarser
  than the compact-open topology, yet it remains Hausdorff; the Hahn–Banach theorem is used to
  verify this for \(\mathcal{L}(E,F)\).  Theorem~\ref{thm:schwartz} implies that
  \(\mathcal{L}(F,E)\) and \(\mathcal{L}(E,F)\) are Luzin, and thus their Borel \(\sigma\)-algebras
  are standard.
\end{proof}

\begin{remark}
  \label{rem:lf-classification}
  The standardness of the Borel \(\sigma\)-algebras of separable LF-spaces also follows from a
  deeper classification theorem due to Mankiewicz~\cite{mankiewiczTopologicalLipschitzUniform1974},
  which states that, up to homeomorphism, there are only three distinct infinite-dimensional
  separable LF-spaces.
\end{remark}

We note that LF-spaces are typically
non-metrizable~\cite{nariciTopologicalVectorSpaces2010}*{12.1.8}.

\subsection{Continuous and smooth functions}
\label{sec:comp-supp-funct}
Let \(\smooth{\Rd}\) denote the space of infinitely differentiable real-valued functions on~\(\Rd\).
For a multi-index \(\alpha = (\alpha_{1}, \ldots, \alpha_{d})\), let \(\partial^{\alpha}\) stand for
the corresponding partial derivative operator
\(\dfrac{\partial^{\alpha}}{\partial x_{1}^{\alpha_{1}} \cdots \partial x_{d}^{\alpha_{d}}}\). For a
compact set \(K \in \vietoris{\Rd}\), define the sup-norm of the \(\alpha\)th derivative over \(K\)
as \(\norm{f}_{\alpha,K} = \sup_{x \in K}|\partial^{\alpha}f(x)|\).  The family of seminorms
\(\norm{\cdot}_{\alpha, K}\) endows \(\smooth{\Rd}\) with the structure of a separable Fr\'echet
space~\cite{trevesTopologicalVectorSpaces2006}*{p.~85}.

The space of continuous functions on \(\Rd\) is denoted by \(C(\Rd)\). The family of seminorms
\(\| f \|_K = \sup_{x \in K} |f(x)|\), indexed by compact subsets \(K \subset \Rd\), endows
\(C(\Rd)\) with a Fr\'echet space structure.

Let \(\contcs{\Rd}\) denote the space of compactly supported continuous functions on~\(\Rd\).
Consider an exhaustion \((K_{n})_{n}\) of \(\Rd\) by a sequence of compact sets. Then
\(\contcs{\Rd}\) can be expressed as the increasing union of spaces
\[\{f \in \contcs{\Rd} : \supp f \subseteq K_{n}\},\]
each of which is a Banach space with respect to the sup-norm. Equipped with the inductive limit
topology, \(\contcs{\Rd}\) becomes an LF-space~\cite{trevesTopologicalVectorSpaces2006}*{p.~131}.

Let \(\smoothcs{\Rd}\) denote the space of compactly supported smooth functions on \(\Rd\), also
known as the \emph{space of test functions}.  For a compact set \(K \subseteq \Rd\), the subspace
\(\{ f \in \smoothcs{\Rd} : \supp (f) \subseteq K\}\) is closed in \(\smooth{\Rd}\) and thus
inherits the structure of a separable Fr\'echet space.  The space of test functions, equipped with
the inductive limit topology of the union
\[
\bigcup_{n}\{f \in \smoothcs{\Rd} : \supp(f) \subseteq K_{n}\},
\]
is also an LF-space~\cite{trevesTopologicalVectorSpaces2006}*{p.~132}.  The inductive topologies on
\(\contcs{\Rd}\) and \(\smoothcs{\Rd}\) are not
metrizable~\cite{nariciTopologicalVectorSpaces2010}*{12.1.8}, but they are
separable~\cite{nariciTopologicalVectorSpaces2010}*{12.109}. Consequently, their Borel
\(\sigma\)-algebras are standard by Corollary~\ref{cor:lf-dual-lf-standard}.  Furthermore, the
inclusion \(\smoothcs{\Rd} \subseteq \contcs{\Rd}\) is continuous, and \(\smoothcs{\Rd}\) is
therefore a Borel subset of \(\contcs{\Rd}\).

\subsection[d-bar problem]{\texorpdfstring{\(\bar{\partial}\)-problem}{d-bar problem}}
\label{sec:d-bar-problem}

We now discuss the existence of Borel equivariant inverses to the \(\bar{\partial}\) map.  Let
\(\smoothCC\) denote the space of complex-valued smooth functions on \(\C\).  We can identify
\(\smoothCC\) with \(\smooth{\C} \times \smooth{\C}\) through the mapping
\((u,v) \mapsto u + \im v\), where \(\smooth{\C}\) represents the space of real-valued smooth
functions on \(\C\), as discussed in Section~\ref{sec:comp-supp-funct}.  The
\(\bar{\partial}\)-operator on this space is defined as
\(\frac{\partial}{\partial \bar{z}} = \frac{1}{2}(\frac{\partial}{\partial x} + \im
\frac{\partial}{\partial y})\), i.e.,
\[\frac{\partial}{\partial \bar{z}}(u + \im v) = \frac{1}{2}\Bigl(\frac{\partial u}{\partial x} -
  \frac{\partial v}{\partial y}\Bigr) + \frac{1}{2}\im \Bigl(\frac{\partial u}{\partial y} +
  \frac{\partial v}{\partial x}\Bigr).\]
In particular, the kernel of \(\bar{\partial}\) consists of the functions that satisfy the
Cauchy--Riemann equations, allowing us to identify it with the space \(\holom\) of entire functions.
Moreover, the topology induced on \(\holom\) by \(\smoothCC\) is precisely the topology of uniform
convergence on compact sets.

The \(\bar{\partial}\)-problem on the space \(\smoothCC\) involves finding a function
\(f \in \smoothCC\) that satisfies the equation \(\bar{\partial} f = g\) for a given
\(g \in \smoothCC\).  It is known that
\[
\bar{\partial} : \smoothCC \to \smoothCC
\]
is surjective, ensuring that a solution to the \(\bar{\partial}\)-problem exists for any smooth function \(g\).
This follows from the fact that linear partial differential operators with constant coefficients are
surjective as maps of the space \(\distrib{\C}\) of distributions into itself
(\cite{hormanderAnalysisLinearPartial2003}*{Cor.~3.6.2}), together with Weyl's lemma, which implies
that distributional solutions \(\bar\partial u=f\) with \(f\) smooth are themselves smooth
(\cite{hormanderAnalysisLinearPartial2003}*{Cor.~4.1.2}).

Corollary~\ref{cor:main-corollary} shows that, on the free part \(\free(\smoothCC)\), solutions to
the d-bar equation can be constructed in a Borel and equivariant manner.  More precisely, let
\(G = \C\) and let \(H = \holom\) be the group of entire functions.  The class \(\mfR\) is defined,
as before, to be the collection \(\mfD_{2}\) of compact sets of \(\C\) diffeomorphic to the unit
disk. The action of \(\C \) on \( \holom\) is via the argument shift, and the semidirect product
\(\holom \rtimes \C \) acts on the space of smooth functions \(\smoothCC\) as
\[
((f,z) \cdot h)(w) = f(w+z) + h(w+z).
\]
The family of seminorms \(\unifn\) with the
\(\mfD_{2}\)-Runge property is given, as earlier, by \(\norm{f}_{K} = \sup_{z \in K}|f(z)|\).  Let
\(Y\) also denote the space \(\smoothCC\) of complex-valued smooth functions on \(\C\), with the
same action \(\C \acts Y\) by the argument shift and let \(\pi\) be the \(\bar{\partial}\) operator.
Corollary~\ref{cor:main-corollary} applies and yields the following result.

\begin{theorem}
  \label{thm:d-bar-problem}
  There exists a Borel equivariant map \[\psi : \free(\smoothCC) \to \smoothCC\]
  which is a right-inverse to \(\bar{\partial}\).
\end{theorem}

The corresponding statement holds true for the Laplacian on \(\Rd\) (as will be discussed in
Section~\ref{sec:poisson-equation}) as well as for any other elliptic partial differential operator
\(P\) with constant coefficients.  We will not pursue the details here, but the key analytic results
are surjectivity of the operator \(P\) on the space of non-periodic smooth functions
(\cite{hormanderAnalysisLinearPartial2003}*{Cor.~3.6.2} and
\cite{hormanderAnalysisLinearPartial2003}*{Cor.~4.1.2}) and the Lax--Malgrange approximation theorem
for \(\ker P\) (\cite{hormanderAnalysisLinearPartial2003}*{Thm.~4.3.1}), which replaces Runge's
theorem.

\subsection{Radon measures and distributions}
\label{sec:radon-meas-distrib}
We will next turn to proving an equivariant version of the Brelot--Weierstrass theorem, which is an
analogue of Weierstrass theorem for subharmonic functions. We begin by introducing the relevant
spaces and maps.

The Riesz representation theorem~\cite{rudinRealComplexAnalysis1987}*{2.14} establishes an
isomorphism between the weak*-dual \(\contcs{\Rd}'\) of the space of compactly supported continuous
functions and the space \(\meas{\Rd}\) of (generalized) \emph{Radon measures} on \(\Rd\) endowed
with the topology of weak convergence (sometimes also called vague convergence), i.e., the weakest
topology under which maps
\[
\meas{\Rd} \ni \mu \mapsto \int_{\Rd} f\, d\mu
\]
are continuous for every compactly supported continuous function \(f\).  A basis of neighborhoods
for a measure \(\mu\) is given by finite collections \(f_{1}, \ldots, f_{n} \in \contcs{\Rd}\) and
\(\epsilon > 0\):
\[U(\mu; f_{1}, \ldots, f_{n}, \epsilon) = \Bigl\{\nu \in \meas{\Rd} : \Bigl| \int f_{i}\, d\mu -
  \int f_{i}\, d\nu \Bigr| < \epsilon \textrm{ for } i = 1, \ldots, n\Bigr\}.\]
This topology makes \(\meas{\Rd}\) a separable, non-metrizable, locally convex topological vector
space, and the weak*-dual to the separable LF-space \(\contcs{\Rd}\).  Hence, its Borel
\(\sigma\)-algebra is standard, as shown by Corollary \ref{cor:lf-dual-lf-standard}.

For a measure \(\mu \in \meas{\Rd}\) and a Borel set \(A \subseteq \Rd\), we say that \(A\) is
\emph{\(\mu\)-null} if \(\mu(A') = 0\) for every Borel subset \(A' \subseteq A\).  Notably, an open
set \(U\) is \(\mu\)-null if and only if \(\int f\, d\mu = 0\) for all \(f \in \contcs{\Rd}\) with
\(\supp f \subseteq U\).  The \emph{support} of a measure \(\mu \in \meas{\Rd}\) is defined as the
closed set \(\supp \mu = \Rd \setminus \bigcup U\), where the union is taken over all open
\(\mu\)-null sets \(U\).

The map \(\meas{\Rd} \ni \mu \mapsto \supp \mu \in \effros{\Rd}\) is Borel.  To verify this,
consider a countable basis of bounded open sets \((U_{n})_{n}\) in \(\Rd\).  For each \(n\), select
a countable family \((f_{n,m})_{m}\) of continuous functions supported on \(U_{n}\) that is dense in
the space \(\{f \in \contcs{\Rd} : \supp f \subseteq U_{n}\}\).  Recall that the Effros Borel
structure is generated by sets of the form \(\{F \in \effros{\Rd} : F \cap U \ne \varnothing\}\),
where \(U \subseteq \Rd\) is open.  Since
\[\supp \mu \cap U \ne \varnothing \iff \exists n\ \exists m\ \Bigl[U_{n} \subseteq U \textrm{ and}
  \int f_{n,m}\, d\mu \ne 0\Bigr], \]
the set \(\{\mu \in \meas{\Rd}: \supp \mu \cap U \ne \varnothing\}\) is open, hence the map
\(\mu \mapsto \supp \mu\) is Borel.

For any bounded Borel set \(A \subseteq \Rd\), the evaluation \(\mu \mapsto \mu(A)\) is Borel.  To
see this, consider a bounded open set \(U\) and define continuous compactly supported functions
\(f_{n}(x) = \min\{1, n\cdot \dist(x, \Rd \setminus U)\}\), where \(\dist\) is any proper metric on
\(\Rd\).  Since \((f_{n})_{n}\) converges monotonically to the characteristic function
\(\bbone_{U}\) of \(U\), the monotone convergence theorem implies
\[ \int_{\Rd} f_{n}\, d\mu \to \int_{\Rd} \bbone_{U}\, d\mu = \mu(U). \]  Thus,
\(\mu \mapsto \mu(U)\) is a pointwise limit of Borel functions and is therefore Borel
\cite{Kechris}*{11.2i}.  The class of Borel subsets \(A\) of \(U\) for which \(\mu \mapsto \mu(A)\)
is Borel includes all open subsets and is closed under complements and countable disjoint unions,
forming a Dynkin system.  By the \(\pi\)-\(\lambda\) theorem (see, for instance,
\cite{Billingsley1995-wf}*{Thm.~3.2}), this class contains all Borel subsets of \(U\).

Let \(\measp{\Rd}\) denote the subspace of \(\meas{\Rd}\) consisting of non-negative functionals,
i.e., \(\phi \in \contcs{\Rd}'\) such that \(\phi(f) \ge 0\) whenever \(f\) is non-negative.  The
Riesz representation theorem identifies \(\measp{\Rd}\) with the space of positive Radon measures on
\(\Rd\).  When restricted to \(\measp{\Rd}\), the weak convergence topology becomes metrizable and,
in fact, Polish (see, for instance,~\cite{kallenbergRandomMeasuresTheory2017}*{Thm.~4.2}).

The weak*-dual of \(\smoothcs{\Rd}\) is known as the space of \emph{distributions} and is denoted by
\(\distrib{\Rd}\).  The continuity of the inclusion \(\smoothcs{\Rd} \hookrightarrow \contcs{\Rd}\)
implies that the restriction
\[
\meas{\Rd} = \contcs{\Rd}' \ni \mu \mapsto \mu|_{\smoothcs{\Rd}} \in \distrib{\Rd}
\]
is continuous.  Since \(\smoothcs{\Rd}\) is dense in \(\contcs{\Rd}\), the restriction map is
injective.  Consequently, \(\meas{\Rd}\) is a Borel subset of \(\distrib{\Rd}\).

\subsection{Harmonic and subharmonic functions}
\label{sec:subh-funct}

Another space relevant to this work is the space of subharmonic functions on \(\Rd\), denoted by
\(\sharm{\Rd}\).  A function \(u : \Rd \to [-\infty, \infty)\) is called \emph{subharmonic} if it is
upper semicontinuous and satisfies, for all \(x\in \Rd\) and \(r > 0\),
\begin{equation}
  \label{eq:subharmonic}
  u(x) \le \frac{1}{\omega(\S^{d-1})}
  \int\limits_{\mathclap{\S^{d-1}}} u(x + r\omega)\, d\omega,
\end{equation}
where \(\omega\) denotes the spherical area measure on the sphere \(\S^{d-1}\).  We exclude the
constant function \(-\infty\) from the class of subharmonic functions.  For an introduction to
subharmonic functions, see \cite{hormanderAnalysisLinearPartial2003}*{Ch.~4}.

Let \(\lloc = \lloc(\Rd)\) denote the space of locally Lebesgue integrable real-valued functions on
\(\Rd\).  Equipped with the seminorms \(\norm{f}_{B_{n}} = \int_{B_{n}} |f|\, d\lambda\), where
\(B_{n}\) is the ball of radius \(n\) centered at the origin, \(\lloc\) forms a separable Fr\'echet
space. Every subharmonic function belongs to \(\lloc\), and distinct elements of \(\sharm{\Rd}\)
correspond to distinct elements of \(\lloc\)~\cite{hormanderAnalysisLinearPartial2003}*{Thm.~4.1.8}.
Thus, \(\sharm{\Rd}\) can be viewed as a subset of \(\lloc\), inheriting a separable metrizable
topology.  Moreover, \(\sharm{\Rd}\) is closed in \(\lloc\), making it a Polish space.  To see this,
consider a sequence \(u_{n} \in \sharm{\Rd}\) converging in \(\lloc\) to some \(u \in \lloc\).  Then
\(u_{n} \to u\) as distributions in \(\distrib{\Rd}\).  Since the Laplacian of any subharmonic
function is non-negative, \(\Delta u_{n} \ge 0\) for all \(n\), and \(\Delta\) is continuous on
\(\distrib{\Rd}\), it follows that \(\Delta u \ge
0\). By~\cite{hormanderAnalysisLinearPartial2003}*{Thm.~4.1.8}, this implies that \(u\) is
subharmonic. Hence, \(\sharm{\Rd}\) is a Polish space under the topology of \(\lloc\) convergence.

A function \(u\) is called \emph{harmonic} if it satisfies \(\Delta u = 0\).  Equivalently, \(u\) is
harmonic if and only if both \(u\) and \(-u\) are subharmonic.  The space of harmonic functions on
\(\Rd\) is denoted by \(\harm{\Rd}\).  On \(\harm{\Rd}\), the \(\lloc\) topology coincides with the
topology of uniform convergence on compact sets, making \(\harm{\Rd}\) a separable Fr\'echet space
and, consequently, a Polish group.

\subsection{Poisson equation}
\label{sec:poisson-equation}

A distinct application of Theorem~\ref{thm:main-theorem} arises from PDE.  Consider the Poisson
equation \(\Delta f = g\) in \(\Rd\), \(d \ge 2\).  Multiple natural spaces of potential
solutions can be considered.  For instance, we can view \(\Delta : \smooth{\Rd} \to \smooth{\Rd}\)
as a map between the spaces of smooth functions or as a map
\(\Delta : \distrib{\Rd} \to \distrib{\Rd}\) between the spaces of distributions.  By a very special
case of the Ehrenpreis--Malgrange theorem
\cite{hormanderAnalysisLinearPartial2003}*{Theorem~7.3.10}, both maps are surjective\footnote{This
  can be also proven directly, following the proof of the Weierstrass theorem.}. Another common
alternative considered in potential theory is to view \(\Delta : \sharm{\Rd} \to \measp{\Rd}\) as a
map between the space of subharmonic functions and (non-negative) Radon measures on \(\Rd\).  The
latter is also surjective (see Hayman--Kennedy \cite{haymanSubharmonicFunctionsVol1976}*{Thm.~4.1}).
In view of Weyl's lemma, the kernel of the Laplacian \(\Delta\) can be identified with the space of
harmonic functions in all these cases.  Consequently, Theorem~\ref{thm:main-theorem} can be applied
for all these choices of the solution space.

Let \(H = \harm{\Rd}\) be the additive group of harmonic functions on \(\Rd\).  It acts on
\(\sharm{\Rd}\) by \((h \cdot f)(z) = f(z) + h(z)\), and \(\Delta f_{1} = \Delta f_{2}\) if and only
if \(f_{2} - f_{1}\) is harmonic.  The family of seminorms on \(H\) is given by
\[\norm{f}_{K} = \sup_{z \in K}|f(z)|\]
and, by the Walsh theorem~\cite{armitageClassicalPotentialTheory2001}*{Cor.~2.6.5}, it satisfies the
\(\mfD_{d}\)-Runge property for the class \(\mfD_{d}\) of compact subsets diffeomorphic to the closed
unit ball.  Corollary~\ref{cor:main-corollary} gives the following.

\begin{theorem}
  \label{thm:equivariant-poisson}
In each of the following cases:
\begin{itemize}[itemsep=-2pt]
\item[{\rm (i)}]\( \Delta\colon \distrib{\Rd} \to \distrib{\Rd} \),

\item[{\rm (ii)}]\(\Delta: \smooth{\Rd} \to \smooth{\Rd}\),

\item[{\rm (iii)}]\(\Delta: \sharm{\Rd} \to \measp{\Rd}\),
\end{itemize}
\noindent
there exists a Borel \(\Rd\)-equivariant
right-inverse to \(\Delta\) on the free part of the corresponding target space.
\end{theorem}

We now turn our attention to the existence of equivariant right-inverses to \(\Delta\) on the
non-free part of \(\measp{\Rd}\).  Suppose that \(\Gamma\) is a proper closed subgroup of \(\Rd\).
Note that \(\Gamma\) is isomorphic to a group of the form \(\Z^{p} \times \R^{q}\), and
consequently, the quotient group \(\Rd/\Gamma\) is isomorphic to \(\R^{d-p-q} \times \T^{p}\),
where \(\T\) denotes the circle group \(\R/\Z\).  Let \(d' = d - p -q\).

Define
\begin{displaymath}
  \begin{aligned}
    Y &= \measp[\Gamma]{\Rd} = \{\mu \in \measp{\Rd} : \stab(\mu) = \Gamma\}, \\
    Z &= \sharm[\Gamma]{\Rd} = \{f \in \sharm{\Rd} : \stab(f) = \Gamma\}.
  \end{aligned}
\end{displaymath}
Let also \(\harm[\Gamma]{\Rd}\) be the group of harmonic functions whose stabilizer
\emph{contains}~\(\Gamma\).  We have the natural action \(\harm[\Gamma]{\Rd} \acts Z\) given by
\((h,f) \mapsto f + h\).  Observe that the argument shift actions of \(\Rd\) on \(Y\) and \(Z\)
induce \emph{free} actions of \(G = \R^{d'} \times \T^{p}\) on these spaces.

\begin{theorem}
  \label{thm:equivariant-poisson-not-free}
  Let \(\Gamma \le \Rd\) be a closed subgroup such that \(\Rd/\Gamma\) is isomorphic to
  \(\R^{d'} \times \T^{p}\).  If \(d' \ge 2\) or, equivalently, \(\dim \Gamma \le d-2\), then there
  exists a Borel equivariant map \(\psi: \measp[\Gamma]{\Rd} \to \sharm[\Gamma]{\Rd}\) such that
  \(\Delta\circ\psi={\rm id}_{\measp[\Gamma]{\Rd}}\).
\end{theorem}

\begin{proof}
  By the dimension of \(\Gamma\) we mean the dimension of the vector space spanned by \(\Gamma\).
  The previous discussion can be adapted to show the existence of such a map \(\psi\) for
  \(d' \ge 2\).  Here, we have \(G = \R^{d'} \times \T^{p}\) and \(H = \harm[\Gamma]{\Rd}\), which
  is naturally isomorphic to \(\harm{G}\).  The class \(\mfR\) comprises the family
  \(\mfD_{d'} \times \T^p\) of compact sets of the form \(K \times \T^p\) for \(K \in \mfD_{d'}\),
  where \(\mfD_{d'}\) is the class of compact subsets of~\(\R^{d'}\) diffeomorphic to the unit ball.
  Note that if \(K\in \mfC_{d'}\)---the class of compact subsets of \(\Rd\) with connected
  complements---then \(K\times \T^{p}\) has connected complement in \(\R^{d'}\times \T^{p}\). We
  need a Runge-type theorem for the harmonic functions in \(H\) for sets in the class
  \(\mfD_{d'} \times \T^{p}\). This result can be derived from the Lax--Malgrange approximation
  theorem, see~\cite{narasimhanAnalysisRealComplex1985}*{Sec.~3.10}.  A simple direct proof is
  presented in Appendix~\ref{sec:rung-theor-peri}.

  Finally, in order to utilize Corollary~\ref{cor:main-corollary}, we need to argue that free
  \(\R^{d'} \times \T^{p}\)-actions admit Borel \(\mfD_{d'} \times \T^{p}\)-toasts.  This was
  established in Lemma~\ref{lem:toasts-for-direct-products}.
\end{proof}

\begin{remark}
  \label{rem:converse-Laplacian}
  Conversely, if \(d' = 0\) or \(d' = 1\) in the statement of
  Theorem~\ref{thm:equivariant-poisson-not-free}, there exist Borel equivariant maps \(\phi\) for
  which the corresponding equivariant Borel liftings \(\psi\) fail to exist. This conclusion will be
  substantiated in Section~\ref{sec:subh-funct-riesz}.
\end{remark}

\begin{remark}
  \label{rem:combine-periodic}
  Theorem~\ref{thm:equivariant-poisson-not-free} guarantees the existence of a Borel equivariant map
  \(\psi_{\Gamma} : \measp[\Gamma]{\Rd} \to \sharm[\Gamma]{\Rd}\) for each subgroup \(\Gamma\) of
  dimension \(\dim \Gamma \le d-2\).  While we do not encounter any apparent obstructions to
  combining these individual maps \(\psi_{\Gamma}\) into a single Borel map \(\psi\) defined on the
  set \(\{\mu : \dim (\stab(\mu)) \le d-2\}\), we do not pursue this direction in the current work.
\end{remark}

\subsection{Dirichlet problem for harmonic functions in half-spaces}
\label{sec:Dirichlet}

Let \(d\ge 1\) and \(\R^{d+1} = \{(x, y)\colon x\in\Rd, y\in\R\}\). We let
\(\R^{d+1}_\pm\) be the half-spaces \(\Rd \times \{\pm y>0\}\).
Given \(f\in C(\Rd)\), consider the Dirichlet problem
\[
\begin{cases}
\Delta h = 0 & {\rm in\ } \R^{d+1}_+, \\
h = f & {\rm on\ } \Rd\times\{0\}.
\end{cases}
\]
The boundary values of \(h\) are understood in the classical sense, that is, \(h\) is assumed to be continuous
in the closed half-space \(\overline{\R}^{d+1}_+ = \{(x, y)\colon y\ge 0\}\), and its boundary values are equal to \(f\).

We denote by \( \harm{\R^{d+1}_+}\) the space of harmonic functions in \( \R^{d+1}_+\) continuous up
to the boundary, and by \( R\colon \mathcal H(\R^{d+1}_+) \to C(\Rd)\) the boundary trace map,
\( (Rh)(x)=h(x, 0)\).  It is well-known that this map is surjective.  Nevanlinna proved this for
\(d=1\) by modifying the Poisson kernel \cite{Nevanlinna1925}; for the general case, see
\cite{FinkelsteinScheinberg}*{Remark~5} and \cite{Gardiner81}\footnote{Streamlined proofs based on
  different ideas can be found in \cite{Arakelyan}, and (only for \(d=1\)) in
  \cite{BHM}*{Proposition~5.3.4}.}.  Next, we note that the linear space
\( \mathcal H(\R^{d+1}_+) \) endowed with the locally uniform topology is Fr\'echet. This follows,
for instance, from the fact that \( \mathcal H(\R^{d+1}_+) \) is a closed subspace of the Fr\'echet
space \( C(\overline{\R}_{+}^{d+1}) \).  We also observe that \( \Rd \) acts continuously on
\( \mathcal H(\R^{d+1}_+) \) and \( C(\Rd) \) by translations in \( x \), and that the map \( R \)
respects this action.

The role of the group \( H \) is played by the additive subgroup
\( \mathcal H_0(\R^{d+1}_+) \triangleleft \mathcal H(\R^{d+1}_+) \) of harmonic functions in
\( \R^{d+1}_+ \) vanishing continuously on the hyperplane \( y=0 \). By the reflection
principle, these functions can be harmonically continued to the whole \( \R^{d+1} \) by
\( h(x, y) = - h(x, -y) \), \( x\in\Rd \), \( y<0\). We call harmonic functions satisfying this
condition \emph{odd-symmetric}.

The group \( \mathcal H_0(\R^{d+1}_+)\) is endowed with the seminorms
\(\unifn = (\norm{\cdot}_{K})_{K \in \vietoris{\Rd}}\)
\begin{displaymath}
  \begin{aligned}
    \norm{f}_{K} &= \sup\{|f(x,y)| : (x,y) \in \tilde{K}\},  \\
    \tilde{K}    &= K \times [-\diam(K), \diam(K)] \in \vietoris{\R^{d+1}},
  \end{aligned}
\end{displaymath}
where \(\tilde{K}\) is a bounded vertical cylinder with base \(K\) and height \(2 \diam(K)\).  The
choice of \(\diam(K)\) ensures that these cylinders form a cofinal family in
\(\vietoris{\R^{d+1}}\), and therefore induce the topology of locally uniform convergence.  This is
a \(\tau\)-family of seminorms in the sense of Definition~\ref{def:U-family-seminorms}, where
\(\tau\) denotes the translation action of \(\Rd\) on \(H\).

Corollary~\ref{cor:main-corollary} is applied to the class \(\mfR = \vietoris{\Rd}\) of all compact
subsets of \(\Rd\).  To check the \(\mathfrak R\)-Runge property, note that if
\(K_{1}, \ldots, K_{m} \in \vietoris{\Rd}\) are pairwise disjoint, then
\(\R^{d+1} \setminus \bigcup_{i} \tilde{K}_{i}\) is connected.  Pick \(h_{i} \in \harm{\R^{d+1}}\),
\(1 \le i \le m\), to be odd-symmetric; the function \(h : \bigcup_{i} \tilde{K}_{i} \to \R\) given
by \(h|_{\tilde{K}_{i}} = h_{i}|_{\tilde{K}_{i}}\) is harmonic on \(\bigcup_{i} \tilde{K}_{i}\).  By
the Walsh Theorem \cite{GardinerHarmonicApprox}*{Theorem~1.7}, for any \(\epsilon > 0\) there is a
harmonic function \(h \in \harm{\R^{d+1}}\) such that
\begin{equation}
  \label{eq:h}
  \norm{h - h_{i}}_{\tilde{K}_{i}} < \epsilon \quad \text{for all } 1 \le i \le m.
\end{equation}
The function \(h\) may not be odd-symmetric and therefore may not correspond to an element of
\( \mathcal{H}_0(\R^{d+1}_+)\). To rectify this, we define
\( h_o(x, y) = \frac{1}{2}\, \bigl( h(x, y) - h(x, -y) \bigr) \) and observe that \( h_o \) is
odd-symmetric. Due to the odd-symmetry of \( h_j \), we have:
\begin{displaymath}
  h_o(x, y) - h_j(x, y) = \frac{1}{2}\, \bigl( (h(x, y) - h_j(x, y)) - (h(x, -y) - h_j(x, -y)) \bigr),
\end{displaymath}
which ensures that Eq.~\eqref{eq:h} also holds for \(h_o\). We conclude that the family of seminorms
\(\unifn\) satisfies the \(\vietoris{\Rd}\)-Runge property.

Applying Corollary~\ref{cor:main-corollary}, we obtain the following result.

\begin{theorem}
\label{thm:Dirichlet}
There is a Borel \(\Rd\)-equivariant map
\(\psi\colon \free(C(\Rd)) \to \harm{\R^{d+1}_+}\) that is a right-inverse to the boundary trace map
\( R \).
\end{theorem}

\subsection{Heat equation}
\label{sec:heat-equation}

Our final application pertains to the existence of equivariant solutions for the inhomogeneous heat
equation
\begin{equation}
  \label{eq:heat-equation}
  \frac{\partial u}{\partial t} - \Delta u = g.
\end{equation}
We consider the heat operator
\((\frac{\partial }{\partial t} - \Delta) : \smooth{\R^{d+1}} \to \smooth{\R^{d+1}}\), \(d \ge 2\),
on the space of smooth functions.  Being a linear partial differential operator with constant
coefficients, it is a surjective map on the space of smooth functions
(\cite{hormanderAnalysisLinearPartial2003}*{Thm.~7.3.10}).  Elements of the kernel of
\((\frac{\partial }{\partial t} - \Delta)\) form an abelian Polish group and are known as
\emph{caloric} functions.

To apply Theorem~\ref{thm:main-theorem}, we set \(Z = \smooth{\R^{d+1}}\), \(G = \R^{d+1}\), and
\(H = \ker (\frac{\partial }{\partial t} - \Delta)\).  As usual, \(G\) acts by the argument shift
and \(H \acts \smooth{\R^{d+1}}\) via \((h,\phi) \mapsto \phi + h\).  The topology induced by
\(\smooth{\R^{d+1}}\) on \(H\) coincides with the (seemingly weaker) topology of uniform convergence
on compact subsets.  Therefore, the family of seminorms on \(H\) can be expressed using the same
formula \(\norm{f}_{K} = \sup_{z \in K}|f(z)|\).

The final component needed for the application of Theorem~\ref{thm:main-theorem} involves selecting
a class of compact sets \(\mfR \subseteq \vietoris{\R^{d+1}}\) such that the action
\(\R^{d+1} \acts H\) satisfies \(\mfR\)-Runge property.  Choosing \(\mfR\) to be the class of
compact sets with connected complements will not suffice, since the Runge-type approximation
property does not hold for caloric functions on such sets.  Instead, we define \(\mfR\) to be the
class of compact \(K \in \vietoris{\R^{d+1}}\) that satisfy the following stronger condition: for
any hyperplane \(P \subseteq \R^{d+1}\) orthogonal to the time axis, the complement
\(P \setminus K\) is connected.  A caloric function on such a set \(K\) can be approximately
extended to a caloric function on all of \(\R^{d+1}\)~\cite{jonesApproximationTheoremRunge1975} (see
also~\cite{diazRungeTheoremSolutions1980} for a complete classification of Runge pairs for the heat
operator).  Borel \(\mfR\)-toasts do exist for all free Borel \(\R^{d+1}\)-action.  The proof is
given in Theorem~\ref{thm:existence-of-slice-filled-toasts} of Appendix~\ref{sec:borel-toasts-heat}.

\begin{theorem}
  \label{thm:equivariant-heat}
  There is a Borel \(\R^{d+1}\)-equivariant map
  \[\psi : \free(\smooth{\R^{d+1}}) \to \smooth{\R^{d+1}}\]
  that is a right-inverse to the heat operator \((\frac{\partial}{\partial t} - \Delta) \).
\end{theorem}

A natural question, similar to the one discussed in Section~\ref{sec:Dirichlet} for the Poisson
equation, is whether the initial value problem
\begin{equation}
  \label{eq:heat-IVP}
  \begin{cases}
    \partial_t u=\Delta u & \text{on } \Rd\times\R^{>0},\\
    u=f&\text{on }\Rd\times\{0\}
  \end{cases}
\end{equation}
admits a Borel equivariant solution. That is, does there exist a Borel map
\[
  \psi: \free(C^\infty(\Rd\times\{0\}))\to C^\infty(\Rd\times\R^{\ge 0})
\]
which is equivariant with respect to \(\Rd\times\{0\}\)-shifts, such that \(u=\psi(f)\) solves the
problem \eqref{eq:heat-IVP}?  There is no obvious Runge property to make use of, and we do not know
what the answer is.

\section{Lack of continuous equivariant inverses}
\label{sec:continuous-inverses}
In Section~\ref{sec:appl-main-theor}, we showed the existence of Borel equivariant right-inverses to
several maps, including \(\dm : \free(\divisp) \to \holom[\ne 0]\), and
\(\bar{\partial} : \free(C^{\infty}(\C,\C)) \to C^{\infty}(\C,\C)\). Both the domains and the
co-domains of these functions have natural Polish topologies.  A considerable strengthening would be
therefore to find \emph{continuous} equivariant inverses.  The purpose of this section is to show
that this is not possible.

\subsection{Divisors of entire functions}
\label{sec:no-cont-equiv-dvs}

Let \(\free(\divisp)\) denote the space of non-periodic positive divisors.  We claim that there are
no continuous equivariant functions \(\xi : \free(\divisp) \to \holom[\ne 0]\) such that
\(\dm(\xi(d)) = d\) for all \(d \in \free(\divisp)\).  Note that a periodic entire function has a
periodic divisor, so any such \(\xi\) would necessarily take values in
\(\free(\holom) = \free(\holom[\ne 0])\).

The easiest obstruction is, perhaps, the fact that \(\free(\divisp)\) has plenty of non-trivial
\(\C\)-invariant compact subsets, whereas \(\free(\holom)\) has no such subspaces.

\begin{lemma}
  \label{lem:compact-invariant-divisors}
  There are (non-empty) compact invariant subsets of \(\free(\divisp)\).
\end{lemma}

\begin{proof}
  The simplest way to construct such compact invariant subsets is, arguably, by taking a
  non-periodic almost periodic subset of \(\C\) with respect to the uniform
  distance~\cite{kolbasinaPropertyDiscreteSets2008} (see also
  \cite{favorovAlmostPeriodicDiscrete2010a}).  Recall that given two discrete sets \(\{a_{n}\}_{n}\)
  and \(\{b_{n}\}_{n}\) in \(\C\), the (extended) uniform transportation distance
  \(\dist(\{a_{n}\}_{n}, \{b_{n}\}_{n})\) between them is given by
  \(\inf_{\sigma} \sup_{n} |a_{n} - b_{\sigma(n)}|\), where the infimum is taken over all bijections
  \(\sigma : \N \to \N\).  An element \(z \in \C\) is an \(\epsilon\)-period of \(\{a_{n}\}_{n}\) if
  \(\dist(\{a_{n}\}_{n}, \{a_{n} + z\}_{n}) < \epsilon\).  Finally, \(\{a_{n}\}_{n}\) is
  \emph{uniformly almost periodic} if for any \(\epsilon > 0\) the set of \(\epsilon\)-periods is
  relatively dense in \(\C\). That is, the distance from any point in \(\C\) to the set of
  \(\epsilon\)-periods is uniformly bounded.

  A discrete set \(\{a_{n}\}_{n}\) can be viewed as a \(\{0,1\}\)-valued divisor.  If it is almost
  periodic in the above sense, the closure of its orbit, call it \(K\), is compact even in the much
  stronger topology of uniform convergence on all of \(\C\), and is therefore also compact in the
  topology of uniform convergence on bounded sets.  Furthermore, if \(\{a_{n}\}_{n}\) is not
  periodic then \(K \subseteq \free(\divisp)\), for if \(z\) is a period of some element of~\(K\),
  then for any \(\epsilon > 0\), \(z\) is an \(\epsilon\)-period of \(\{a_{n}\}_{n}\), hence
  \(\{a_{n}\}_{n}\) itself is periodic.

  Here is an explicit example of an almost periodic non-periodic subset \(\C\).  For integers
  \(m,n \in \Z\), let \(\alpha(m,n)\) denote the power of \(2\) in the prime decomposition of
  \(\gcd(m,n)\).  In other words, \(\alpha = \alpha(m,n) \in \N\) satisfies \(m = 2^{\alpha}m_{0}\),
  \(n = 2^{\alpha}n_{0}\), where either \(n_{0}\) or \(m_{0}\) is odd.  For an element
  \(m + \im n \in \Z + \im \Z\), let
  \[g(m+\im n) = m+ \im n + 2^{-1}\sum_{i=1}^{\alpha(m,n)}2^{-i} = m+ \im n + 1/2 -
    2^{-\alpha(m,n)-1}\]
  and set \(Z = \{ g(m+\im n) : m,n \in \Z\}\).  Vectors of the form \((2^{k}m, 2^{k}n)\),
  \(m,n \in \Z\) form \(2^{-k}\)-periods of \(Z\), hence it is almost periodic.
\end{proof}

\begin{lemma}
  \label{lem:compact-invariant-entire-const}
  If \(K \subseteq \holom\) is invariant and compact then \(K \subseteq \{\textrm{constants}\}\).
\end{lemma}

\begin{proof}[Proof]
  Let \( K \subseteq \holom\) be compact and invariant. The evaluation map
  \[
    \holom \ni f \mapsto f(0) \in \C
  \]
  is a continuous function, and therefore, by compactness, \(\sup_{f \in K}|f(0)| \le A\) for some
  \(A\).  Since \(K\) is invariant under the shift of the argument, this is equivalent to
  \(\sup_{f \in K} \sup_{z \in \C} |f(z)| \le A\).  We conclude that all \(f \in K\) are bounded
  entire functions, and therefore must be constant by the Liouville theorem.
\end{proof}

\begin{corollary}
  \label{cor:no-compact-invariant-entire}
  \(\holom \setminus \{\textrm{constants}\}\) has no non-trivial compact invariant subsets.
\end{corollary}

\begin{theorem}
  \label{thm:no-continuous-equivariant}
  There are no continuous equivariant maps
  \(\beta : \free(\divisp) \to \holom \setminus \{\textrm{constants}\}\).  In particular, there are
  no continuous equivariant right-inverses to \(\dm\) defined on all of \(\free(\divisp)\).
\end{theorem}

\begin{proof}
  Suppose such a function \(\beta\) did exist. By Lemma~\ref{lem:compact-invariant-divisors}, there
  is a non-empty compact invariant \(K \subseteq \free(\divisp)\).  Its image, \(\beta(K)\), would
  be a non-empty, compact (by the continuity of \(\beta\)), and invariant (by the equivariance of
  \(\beta\)) subset of \(\holom \setminus \{\textrm{constants}\}\), contradicting
  Corollary~\ref{cor:no-compact-invariant-entire}.
\end{proof}

\subsection[\texorpdfstring{\(\bar\partial\)}{d-bar}-problem]
{\texorpdfstring{\(\bar{\partial}\)-problem}{d-bar problem}}
\label{sec:d-bar-problem-continuous}
There are no continuous equivariant right-inverses to the operator
\[
\bar{\partial} : \free(C^{\infty}(\C, \C)) \to C^{\infty}(\C, \C)
\]
either, but this argument requires a little more work.
The reason for the difference is that there are plenty of compact
invariant sets in \(\free(C^\infty(\C,\C))\).

Before we proceed to prove this claim, we record an elementary topological transitivity property in
the space \(C^{\infty}(\C, \C)\times C^{\infty}(\C, \C)\). For brevity, we write simply
\(C^{\infty}\) for \(C^{\infty}(\C, \C)\).

\begin{lemma}
  \label{lem:C-infty-transitive}
  There is a pair \((f_1,f_2)\in C^\infty\times C^\infty\) whose \(\C\)-orbit is dense in
  \(C^\infty\times C^\infty\).
\end{lemma}

\begin{proof}
  Choose a sequence \(K_n\subset \C\) of compact sets and a sequence of complex numbers
  \(|w_n|\ge 2\), such that
  \begin{itemize}[leftmargin=.6cm]
  \item \(K_n\cap \left(w_n\cdot K_n\right)=\varnothing\),
  \item \(K_n\cup \left(w_n\cdot K_n\right)\subseteq K_{n+1}\),
  \item \(\bigcup_{n}K_n=\C\).
  \end{itemize}
  Here and below, \(w\cdot K\) and \(w\cdot f\) denote the translated set \(\{z-w:z\in K\}\) and
  function \(w\cdot f(z)=f(z+w)\), respectively.  We also fix a dense sequence of elements
  \((h_{1,n}, h_{2,n})_n\) in \(C^\infty\times C^\infty\), such that each element of the sequence
  occurs infinitely many times.

  Next, we construct a sequence \(f_n=(f_{1,n},f_{2,n})_n\) which converges to an element
  \((f_1,f_2)\) with dense \(\C\)-orbit.  We start with an arbitrary pair \((f_{1,0},f_{2,0})\).
  Supposing that \((f_{1,n},f_{2,n})\) has been chosen, we pick \(f_{i,n+1}\in C^\infty\),
  \(i=1,2\), with the properties that
  \[
    f_{i,n+1}=
    \begin{cases}
      f_{i,n}& \text{on }K_n\\
      (-w_n)\cdot h_{i,n} & \text{on } w_n \cdot K_n.
    \end{cases}
  \]
  Then, for \(n>m\), \(f_{i,n}=f_{i,m}\) on \(K_m\), and hence the sequence converges to a smooth
  function \(f_i\). It remains to establish that the limiting pair \((f_1,f_2)\) has dense orbits.

  To this end, fix \(K_m\), a natural number \(p\in\N\), and an arbitrary element
  \(h=(h_1,h_2)\in C^\infty\times C^\infty\).  For \(g=(g_1,g_2)\in C^\infty\times C^\infty\), we
  denote by \(\lVert g\rVert_{m,p}\) the norms
  \[
    \lVert g\rVert_{m,p}=\max_{i=1,2}\max_{0\le j\le p}\max_{K_m}|g^{(j)}_i|.
  \]

  We will show that
  \begin{equation}
    \label{eq:dens-wn-f}
    \liminf_{n\to\infty}\lVert w_n\cdot f-h\rVert_{m,p}=0.
  \end{equation}
  Given \(\epsilon>0\), we let \(\mathcal{I}\subset\N\) be all indices \(n\) for which
  \(\lVert h_n-h\rVert_{m,p}<\epsilon\). By the density of the sequence \((h_n)_n\) and by the fact
  that each element occurs infinitely often in the sequence, we conclude that \(\mathcal{I}\) is an
  infinite set. For any index \(n\) we have that
  \[
    \lVert w_n\cdot f - h\rVert_{m,p}\le \lVert w_n\cdot (f-f_{n+1})\rVert_{m,p} + \lVert w_n\cdot
    f_{n+1} - h_{n}\rVert_{m,p} + \lVert h_n - h\rVert_{m,p}.
  \]
  For \(n>m\), we have that \(w_n\cdot f_{n+1}=h_n\) on \(K_n\supseteq K_m\), so the second term
  vanishes. For \(n>m\) with \(n\in \mathcal{I}\), the third term is moreover bounded above by
  \(\epsilon\), so for such indices we get
  \[
    \begin{aligned}
      \lVert w_n\cdot f - g\rVert_{m,p}&\le \lVert w_n\cdot (f-f_{n+1})\rVert_{m,p}+\epsilon \\
                                       & \le \lVert f-f_{n+1}\rVert_{m+1,p}+\epsilon.
    \end{aligned}
  \]
  Taking the limit \(n\to\infty\) along \(\mathcal{I}\), the claim \eqref{eq:dens-wn-f} follows.
\end{proof}

We proceed to the main result of this section.

\begin{theorem}
  \label{thm:no-dbar-inverse}
  There are no continuous equivariant maps
  \[\xi : \free(C^{\infty}(\C, \C)) \to C^{\infty}(\C, \C)\] such that
  \(\bar{\partial}(\xi(f)) = f\) for all \(f \in \free(C^{\infty}(\C, \C))\).
\end{theorem}

\begin{proof}
  Suppose towards a contradiction that \(\xi : \free(C^{\infty}) \to C^{\infty}\) is continuous,
  equivariant, and satisfies \(\bar{\partial}(\xi(f)) = f\) for all \(f \in \free(C^{\infty})\).  We
  note that there is some freedom in the choice of \(\xi\): we may add any Borel map
  \(c:C^\infty\to \C\) which is constant on the orbits of the action \(\C\acts C^\infty\), and
  obtain a new map with the same properties.

  Consider the map \(\alpha(f_{1}, f_{2}) = \xi(f_{1} + f_{2}) - \xi(f_{1}) - \xi(f_{2})\),
  defined for those pairs \((f_{1}, f_{2}) \in \free(C^{\infty})^{2}\) that satisfy
  \(f_{1} + f_{2} \in \free(C^{\infty})\).  Note that \(\alpha\) is continuous and equivariant with
  respect to the diagonal action \(\C \acts \free(C^{\infty}) \times \free(C^{\infty}) \).  Since
  \(\bar{\partial}\) is linear, we have \(\bar{\partial}(\alpha(f_{1}, f_{2})) = 0\), showing that
  \(\alpha(f_{1}, f_{2}) \in \holom\) for all \((f_{1}, f_{2}) \in \dom \alpha\).

  The proof will consist of three main steps:
  \begin{enumerate}[leftmargin=1cm, label={\sf \arabic*}., ref={\sf \arabic*}]
  \item Show that the map \(\alpha\) is constant, so with an appropriate choice of the additive
    constant, \(\xi\) is linear.
  \item Show that the map \(\xi\) has to be given by the Cauchy transform
    \[
      \xi(f)(\zeta) = \frac{1}{2\pi \im} \int_{\C} \frac{f(z)}{z-\zeta} dz \wedge d\bar{z}.
    \]
  \item Finally, we argue that the Cauchy transform cannot be continuously extended to the whole of
    \(\free(C^\infty)\), thus reaching a contradiction.
  \end{enumerate}

  We proceed with this plan. Suppose that \((f_{1}, f_{2}) \in \dom \alpha\) is such that the
  closure \(K\) of its orbit under the diagonal action is compact and satisfies
  \(K \subseteq \dom \alpha\).  Then \(\alpha(K)\) is also compact and shift invariant, hence
  \(\alpha(K) \subseteq \{\textrm{constants}\}\) by Liouville's theorem.  One can take for such
  \(f_{1}, f_{2}\) non-periodic uniformly almost periodic functions and check that such pairs are
  dense in \(\dom \alpha\).  We conclude that \(\alpha(f_{1}, f_{2}) \in \{\textrm{constants}\}\)
  for all \((f_{1}, f_{2}) \in \dom \alpha\).

  Lemma~\ref{lem:C-infty-transitive} implies that the action
  \(\C \acts C^{\infty} \times C^{\infty}\) is topologically transitive.  If \(\omega_{0} \in \C\)
  corresponds to the constant function such that \(\alpha(f_{1}, f_{2}) = \omega_{0}\), then
  \(\alpha(z \cdot f_{1}, z \cdot f_{2}) = z \cdot \alpha(f_{1}, f_{2}) = w_{0}\) for all \(z\).
  Hence, \(\alpha(g_{1}, g_{2}) = \omega_{0}\) for all \(g_{1}, g_{2}\) in the orbit of
  \((f_{1}, f_{2})\) and therefore also for all \((g_{1}, g_{2}) \in \dom \alpha \) by continuity.
  We conclude that \(\alpha\) is identically equal to the constant function \(\omega_{0}\).

  Changing \(\xi\) to \(\xi'(f) = \xi(f) + \omega_{0}\) we may therefore assume without loss of
  generality that \(\xi\) is additive in the sense that
  \(\xi(f_{1} + f_{2}) = \xi(f_{1}) + \xi(f_{2})\) holds whenever all the functions belong to the
  free part \(\free(C^{\infty})\).

  If \(f \in C^{\infty}\) has compact support and \(z_{n} \to \infty\), then \(z_{n} \cdot f \to 0\)
  in the topology of \(C^{\infty}\).  In particular,
  \(\overline{\C \cdot f} = (\C \cdot f) \cup \{0\}\) is compact in \(C^{\infty}\).  Unfortunately,
  \(0 \not \in \free(C^{\infty})\), so \(\xi(0)\) is not defined.  We claim, however, that for any
  compactly supported \(f \in C^{\infty}\) and any sequence \(z_{n} \to \infty\) one has
  \(\xi(z_{n} \cdot f) \to 0\), showing that \(\xi\) can be extended to a continuous function on
  \(\overline{\C \cdot f}\) by setting \(\xi(0) = 0\).  It suffices to show that any
  \(z_{n} \to \infty\) has a subsequence satisfying \(\xi(z_{n_{k}} \cdot f) \to 0\).  Take \(h\) to
  be a non-periodic uniformly almost periodic function on \(\C\).  Then
  \(h, h+ f \in \free(C^{\infty})\).  Passing to a subsequence and using the almost periodicity of
  \(h\), we may ensure that \(\lim_{n} (z_{n} \cdot h)\) exists; let us denote it by
  \(\widetilde{h}\).  Using the established additivity and continuity of \(\xi\), we get
  \[\lim_{n}(\xi(z_{n}\cdot h) + \xi(z_{n}\cdot f)) = \lim_{n}\xi(z_{n}\cdot h + z_{n} \cdot f) =
    \xi(\widetilde{h}) = \lim_{n}\xi(z_{n} \cdot h),\]
  and therefore \(\lim_{n}\xi(z_{n} \cdot f) = 0\), as claimed.

  For any \(f \in C^{\infty}\) with compact support, \(\xi(f)\) is necessarily a bounded function.
  Indeed, \(\C \cdot f\) is precompact, and we have just shown that so is
  \(\C \cdot \xi(f) = \xi(\C \cdot f)\).  The ``evaluation at \(0\)'' map \(f \mapsto f(0)\) is
  continuous on \(C^{\infty}\), hence bounded on \(\overline{\xi(\C \cdot f )}\).  Since the latter
  is shift-invariant, this is equivalent to \(\sup_{h}\sup_{z \in \C} |h(z)| < \infty\), where the
  first supremum is over \(h \in \overline{\xi(\C \cdot f )}\).  Thus, \(\xi(f)\) is a bounded
  function.

  There is a ``natural candidate'' for what the map \(\xi\) could be. Given a compactly supported
  \(f \in C^{\infty}\), let
  \[
    \widetilde{\xi}(f)(\zeta) = \frac{1}{2\pi \im} \int_{\C} \frac{f(z)}{z-\zeta} dz \wedge
    d\bar{z}.
  \]
  One has \(\widetilde{\xi}(f) \in C^{\infty}\) and \(\bar{\partial} (\widetilde{\xi}(f)) = f\)
  (see~\cite{hormanderIntroductionComplexAnalysis1990}*{1.2.1}).  Furthermore, we have that
  \(\lim_{\zeta \to \infty}\widetilde{\xi}(f)(\zeta) = 0\).  Indeed, if
  \(M = \max_{z \in \supp f}|f(z)|\), then
  \begin{equation}
    \label{eq:zero-at-infinite}
    |\widetilde{\xi}(f)(\zeta)| \le \frac{\mu(\supp f)}{2 \pi}  \cdot \frac{M}{\dist(\zeta, \supp
      f)} \to 0 \textrm{ as } \zeta \to \infty.
  \end{equation}
  In particular, \(\widetilde{\xi}(f)\) is also a bounded function.

  Since \(\bar{\partial}(\widetilde{\xi}(f)) = f\) and \(\bar{\partial} (\xi(f)) = f\) for all compactly
  supported \(f\), we conclude that \(\xi(f) - \widetilde{\xi}(f)\) is entire.  Since each of
  \(\xi(f), \widetilde{\xi}(f)\) is bounded, so is their difference
  \(c(f) = \widetilde{\xi}(f) - \xi(f)\) which thus must be constant.

  In fact, \(c(f) \) is identically zero.  Indeed, if \(z_{n} \to \infty\), then
  \[
    c(f) = z_{n}\cdot c(f) = z_{n} \cdot \widetilde{\xi}(f) - z_{n} \cdot \xi(f) = z_{n} \cdot
    \widetilde{\xi}(f) - \xi(z_{n} \cdot f).
  \]
  Since \(\xi(z_{n} \cdot f) \to 0\), we conclude that \(z_{n} \cdot \widetilde{\xi}(f) \to c(f)\)
  whenever \(z_{n} \to \infty\).  However, Eq.~\eqref{eq:zero-at-infinite} gives
  \[
  (z_{n} \cdot \widetilde{\xi}(f))(0) = \widetilde{\xi}(f)(z_{n}) \to 0,
  \]
  hence \(c(f) = 0\). In other words, \(\xi(f) = \widetilde{\xi}(f)\) for all compactly supported \(f\).

  To show that continuous equivariant right-inverse to \(\bar{\partial}\) cannot exist, it therefore
  suffices to verify that \(\widetilde{\xi}\) does not extend continuously to all of
  \(\free(C^{\infty})\).  Let \(f_{n}\) be an element of \(C^{\infty}\) such that
  \begin{enumerate}[leftmargin=1cm, label={\sf \arabic*}., ref={\sf \arabic*}]
  \item \(f_{n}(z) = z\) for all \(z \in n\D\)
  \item \(f_{n}(z) = 0\) for all \(z \in \C \setminus (n+1) \D\)
  \item \(|f_{n}(z)| \le n+1\) for all \(z \in \C\).
  \end{enumerate}

  Clearly \(f_{n}\) converges to the identity map \(z \mapsto z\).  However,
  \(\widetilde{\xi}(f_{n})\) diverges, and in fact, \(\widetilde{\xi}(f_{n})(0) \to \infty\) as
  \(n \to \infty\).  Indeed,
  \begin{displaymath}
    \begin{aligned}
      |2\pi \im\, \widetilde{\xi}(f_{n})(0)|
      &= \Bigl|\int_{n\D}\frac{z}{z} + \int_{(n+1)\D \setminus n \D}
        \frac{f_{n}}{z}\Bigr| \ge \pi n^{2} - \int_{(n+1)\D \setminus n
        \D}\frac{n+1}{|z|} \\
      &\ge \pi n^{2} -
        2(\pi(n+1)^{2} - \pi n^{2}) \\
      &= \pi n^{2} - 4\pi n - 2\pi \to \infty \textrm{ as } n \to \infty.
    \end{aligned}
  \end{displaymath}
  This completes the proof.
\end{proof}

\section{Lack of periodic equivariant inverses}
\label{sec:borel-inverses-1-periodic}

In most of the applications considered so far, we have concentrated on free \(\Rd\)-actions, with
Section~\ref{sec:poisson-equation} being a notable exception. The situation with the periodic parts
of the actions discussed in Section~\ref{sec:appl-main-theor} is more subtle.  In this section we
show that equivariant right-inverses to \(\dm\), \(\Delta\) and \(\bar{\partial}\) generally do not
exist on the periodic parts of their domain.

\subsection{Divisors and entire functions}
\label{sec:divis-merom-funct}

The stabilizer of a divisor under the argument shift action must be a closed subgroup of \(\C\).
Furthermore, since the support of any divisor is a discrete subset of \(\C\), the stabilizer of a
non-trivial divisor is a discrete subgroup of \(\C\), hence isomorphic either to \(\Z\) or to
\(\Z^{2}\).  Let \(\divisp[1]\) be the collection of those non-negative divisors whose stabilizer is
of the form \(\lambda \Z\) for some non-zero \(\lambda \in \C\) and \(\divisp[2]\) be those
non-trivial non-negative divisors that have two independent periods.  We therefore get a partition
\[\divisp = \free(\divisp) \sqcup \divisp[1] \sqcup \divisp[2] \sqcup \{\varnothing\}.\]
Similarly, the stabilizer of a non-constant entire function must also be a discrete subgroup of
\(\C\), which gives us a similar partition of \(\holom\) into
\[\holom = \free(\holom) \sqcup \holom[1] \sqcup \holom[2] \sqcup \{\textrm{constants}\}.\]
Any equivariant right-inverse \(\xi\) to \(\dm\) must respect these partitions and thus satisfy
\(\xi(\free(\divisp)) \subseteq \xi(\free(\holom)) \) and
\(\xi(\divisp[i]) \subseteq \xi(\holom[i])\), \(i = 1,2\).  However, \(\holom[2]\)---the space of
non-constant doubly periodic entire functions---is empty, whereas \(\divisp[2]\) is not.  In
particular, there can't be any equivariant right-inverses to \(\dm\) on \(\divisp[2]\).

The set \(\holom[1]\) is non-empty, and the divisor map
\[
  \dm:\holom[1]\to\divisp[1]
\]
is a Borel surjection. To see this, let \(\holom[1,1]\) and \(\divisp[1,1]\) denote the subspaces of
(non-constant) holomorphic functions and divisors, respectively, which have period \(1\). These are
Polish subspaces of \(\holom[1]\) and \(\divisp[1]\), correspondingly. It suffices to see that the
restriction \(\dm:\holom[1,1]\to\divisp[1,1]\) is surjective. The exponential map
\(w=e^{2\pi \im z}\) transfers the question to surjectivity of the divisor map for analytic
functions in the punctured plane \(\C\setminus\{0\}\), which is classical for arbitrary domains
(\cite{MR1483074}*{Ch.~4}).

A non-measurable equivariant inverse to \(\dm\colon \holom[1]\to\divisp[1]\) can easily be
constructed using the axiom of choice. It turns out, however, that contrary to the free part of the
action, there is no Borel equivariant inverse.

\begin{theorem}
  \label{thm:no-borel-equivariant-inverse-1-periodic}
  There is no Borel equivariant map \(\xi : \divisp[1] \to \holom[1]\) such that \(\dm(\xi(d)) = d\)
  for all \(d \in \divisp[1]\).
\end{theorem}

We begin with a lemma which is an easy consequence of the Lindelöf maximum principle.  For the
reader's convenience, we include the proof.

\begin{lemma}
  \label{lem:maximum-principle-strip}
  Let \(s < 0 < t\) be reals, \(S_{s,t}\) be the strip \(s < \Im(z) < t\), and
  \(f : \bar{S}_{s,t} \to \C\) be a bounded continuous function holomorphic on \(S_{s,t}\).  If
  \(|f(z)| \le M\) for all \(z \in \partial S_{s,t}\), then \(|f(z)| \le M\) for all
  \(z \in S_{s,t}\).
\end{lemma}

\begin{proof}
  Set \(S = S_{s,t}\) and let \(z_{0} \in \C \setminus \bar{S}\) be such that \(|z-z_{0}| \ge 1\)
  for all \(z \in S\).  For instance, we may take \(z_{0} = \im(t + 1)\).  Given \(\epsilon > 0\),
  let \(g_{\epsilon}(z) = \frac{f(z)}{(z-z_{0})^{\epsilon}}\), where we take the branch cut from
  \(z_{0}\) parallel to the imaginary axes that avoids the strip \(S\).  The function
  \(g_{\epsilon}(z)\) is holomorphic on \(S\) and continuous on \(\bar{S}\).  Furthermore, the
  choice of \(z_{0}\) ensures that \(|g_{\epsilon}(z)| \le M\) for all \(z \in \partial S\).  Since
  \(f\) is bounded by assumption, \(\lim_{|z| \to \infty}|g_{\epsilon}(z)| = 0\).  In particular,
  for a sufficiently large \(R > 0\) and \(S^{R} = S \cap \{z : |\Re(z)| \le R\}\), we have
  \(|g_{\epsilon}(z)| \le M\) for all \(z \in \partial S^{R}\).  Maximum modulus principle applies
  and shows that \(|g_{\epsilon}(z)| \le M\) for all \(z \in S^{R}\) and therefore also for all
  \(z \in S \) as \(R\) can be arbitrarily large.  We conclude that
  \(|f(z)| \le M |z-z_{0}|^{\epsilon}\) for all \(\epsilon > 0\) and all \(z \in S\).  Sending
  \(\epsilon \to 0\) gives the desired \(|f(z)| \le M\) on all of~\(S\).
\end{proof}

\begin{lemma}
  \label{lem:1-periodic-limit}
  Let \(f : \C \to \C\) be a non-constant periodic entire function with period \(z = 1\).  Let
  \[M_{f}(s) = \max\{|f(\im s + t)| : t \in \R\} = \max\{|f(\im s + t)| : t \in [0,1]\}\]
  be the maximum modulus of \(f\) on the line \(\{\im s + t : t \in \R \}\), where \(s \in
  \R\). Then
  \[\textrm{either}\quad \lim_{s \to +\infty} M_{f}(s) = +\infty \quad \textrm{or} \quad
    \lim_{s \to -\infty} M_{f}(s) = +\infty.\]
\end{lemma}

\begin{proof}
  Suppose towards a contradiction that there are \(s_{n} \to -\infty\) and \(t_{n} \to +\infty\)
  such that \(M_{f}(s_{n}) \le M\) and \(M_{f}(t_{n}) \le M\) for some \( M \in \R^{\ge 0}\) and all
  \(n\).  Lemma~\ref{lem:maximum-principle-strip} applies to strips \(S_{s_{n},t_{n}}\) and
  functions \(f|_{\bar{S}_{s_{n},t_{n}}}\).  We conclude that \(|f(z)| \le M\) on each string
  \(S_{s_{n}, t_{n}}\) and therefore also on all of \(\C = \bigcup_{n} S_{s_{n},t_{n}}\).
  Liouville's theorem assures that \(f\) must be constant, contrary to the assumption.
\end{proof}

\begin{proposition}
  \label{prop:action-on-H1-smooth}
  The action \(\C \acts \holom[1,1]\) has a Borel transversal, i.e., a Borel set
  \(T \subset \holom[1,1]\) that intersects each orbit in exactly one point.
\end{proposition}

\begin{proof}
  Consider a \(1\)-periodic \(f \in \holom[1,1]\) and let \(k\) be large enough to ensure that the
  set \(W_{f,k} = \{s \in \R : M_{f}(s) \le k\}\) is non-empty.  Note that \(W_{f,k}\) is
  necessarily closed, whereas Lemma~\ref{lem:1-periodic-limit} guarantees that it must be bounded
  from above or from below, and possibly both. Let's take the set
  \(\widetilde{T}_{k} \subseteq \holom[1,1]\) of \(1\)-periodic entire functions for which
  \(M_{f}(0) \le k\) and \(s = 0\) is either the smallest or the largest element of \(W_{f,k}\).
  This set is Borel.  In fact,
  \(\widetilde{T}_{k} = (\widetilde{T}^{0}_{k} \cap \widetilde{T}^{+}_{k}) \cup
  (\widetilde{T}^{0}_{k} \cap \widetilde{T}^{-}_{k})\), where
  \begin{displaymath}
    \begin{aligned}
      \widetilde{T}^{0}_{k} &= \Bigl\{f \in \holom[1,1] : \forall r \in [0,1] \cap \Q \quad |f(r)| \le k \Bigr\},  \\
      \widetilde{T}^{+}_{k} &= \Bigl\{ f \in \holom[1,1] : \forall n\ \exists m\ \forall s \in
                              \Q^{>1/n}\ \exists r \in [0,1] \cap \Q \quad
                              |f(\im s +r)|> k + \frac{1}{m} \Bigr\},\\
      \widetilde{T}^{-}_{k} &= \Bigl\{ f \in \holom[1,1] : \forall n\ \exists m\ \forall s \in
                              \Q^{< -1/n}\ \exists r \in [0,1] \cap \Q \quad
                              |f(\im s +r)|> k + \frac{1}{m} \Bigr\}.\\
    \end{aligned}
  \end{displaymath}

  The set \(\widetilde{T}_{k}\) is invariant under the real shifts: if \(f \in \widetilde{T}_{k}\)
  then \(z \mapsto f(z + r) \in \widetilde{T}_{k}\), \(r \in \R\).  By \(1\)-periodicity, we have
  the action of \(\T = [0,1]\) on \(\widetilde{T}_{k}\).  Since Borel actions of compact groups
  admit Borel transversals, we can pick a Borel transversal \(S_{k} \subseteq \widetilde{T}_{k}\)
  for the action of \(\T\).  Note that \(S_{k}\) intersects each \(\C\)-orbit of \(f \in W_{f,k}\)
  in one or two points, depending on whether \(W_{f,k}\) is bounded from only one side or from both.
  Choosing in a Borel way either of the two points on those orbits where
  \(|S_{k} \cap (\C \cdot f)| = 2\), gives a Borel transversal \(T_{k} \subseteq S_{k}\) for the
  restriction of the action onto \(Z_{k} = \{f \in \holom[1,1] : W_{f,k} \ne \varnothing\}\).

  It only remains to glue the sets \(T_k\) together into a single Borel transversal by setting
  \(T=\bigcup_{k} \left(T_k\setminus \bigcup_{\ell<k}Z_{\ell}\right)\).
\end{proof}

\begin{remark}
  \label{rem:1-periodic-borel-transversal}
  In fact, the action of \(\C \acts \holom[1]\) on all of \(\holom[1]\) admits a Borel transversal.
  Let \(\alpha : \holom[1,1] \to Y\) be a Borel reduction of \(E_{\C \acts \holom[1,1]}\) to the
  equality on some standard Borel space \(Y\).  There is a Borel
  \(E_{\C \acts \holom[1]}\)-invariant function \(\pi : \holom[1] \to \C\setminus \{0\}\) such that
  \(\pi(f)\) is a generator for the group of periods of \(f\) (see, for
  instance,~\cite{Slutsky}*{Cor.~5.3}).  Note that the function \(\widetilde{f}(z) = f(\pi(f)z) \)
  belongs to \(\holom[1,1]\) and therefore
  \[\widetilde{\alpha} : \holom[1] \to Y \times (\C\setminus \{0\}),\] given by
  \(\widetilde{\alpha}(f) = (\alpha(\widetilde{f}), \pi(f))\), is a reduction from
  \(E_{\C \acts \holom[1]}\)to the equality on \(Y \times (\C\setminus \{0\})\). Thus,
  \(E_{\C \acts \holom[1]}\) is concretely classifiable and therefore admits a Borel transversal.
\end{remark}

The action \(\C \acts \divisp[1,1]\), on the other hand, does not admit Borel transversals.  Indeed,
one can embed the free part of the Bernoulli shift on \(2^{\Z}\) into \(E_{\C \acts \divisp[1,1]}\)
by viewing a sequence \(x \in 2^{\Z}\) as a \(\{0,1\}\)-valued divisor along \(\im \Z\).  For the
benefit of the readers unfamiliar with this type of argument, we carefully spell out the details of
the following standard ideas.

\begin{lemma}
  \label{lem:D11-no-transversal}
  Consider the action \(\Z \acts \divisp[1,1]\), where the generator \(\tau\) acts by
  \(d(z) \mapsto d(z-\im)\).  The action \(\Z\acts \divisp[1,1]\) does not admit a Borel
  transversal.
\end{lemma}

\begin{proof}
  Let \(\mathrm{Free}(2^{\Z})\) denote the set of non-periodic binary sequences indexed by~\(\Z\),
  and denote by \(s\) the (left) shift map \((s\cdot x)(i) = x(i-1)\), \(x\in\free(2^\Z)\).  The
  following two items ensure the lack of Borel transversals.
  \begin{enumerate}[leftmargin=1cm, label={\sf \arabic*}., ref={\sf \arabic*}]
  \item\label{item:bernoulli-no-transversal} The action \(s\acts \free(2^{\Z})\) does not admit a
    Borel transversal.
  \item\label{item:bernoulli-embedding} There exists a continuous equivariant injection
    \(\varphi:\mathrm{Free}(2^{\Z})\to \divisp[1,1]\).
  \end{enumerate}
  Any Borel transversal \(T\) for \(\Z \acts \divisp[1,1]\), had it existed, would then pull back to
  a Borel transversal \(\phi^{-1}(T)\) for \(s \acts \free(2^{\Z})\),
  violating~\eqref{item:bernoulli-no-transversal}.

  It thus remains to establish Properties~\eqref{item:bernoulli-no-transversal}
  and~\eqref{item:bernoulli-embedding} above.  Suppose towards a contradiction that there exists a
  Borel transversal \(B\) for \(s \acts \free(2^{\Z})\).  The sets \(B_{n} = s^{n} \cdot B\) are
  pairwise disjoint (by the freeness of the action) and form a partition of \(\free(2^{\Z})\).
  However, the shift map has invariant probability measures.  For instance, the product measure
  \(\prod_{k \in \Z}\mu_{0}\), where \(\mu_{0}\) is the Bernoulli measure on \(\{0,1\}\),
  \(\mu_{0}(0) = 1/2 = \mu_{0}(1)\).  Sets \(B_{n}\) must therefore all have the same measure, which
  implies \(\mu(B_{n}) = 0\) by finiteness of \(\mu\).  The conclusion \(\mu(X) = 0\) leads to a
  contradiction.  The proof of Property~\eqref{item:bernoulli-no-transversal} is complete.

  It remains to establish the existence of the equivariant Borel injection \(\varphi\). Given
  \(x \in 2^{\Z}\), let \(d_{x}\) be the divisor given by
  \begin{displaymath}
    d_{x}(z) =
    \begin{cases}
      x(k) & \textrm{if \(z = k\im + m\) for some \(k,m \in \Z\)}, \\
      0    & \textrm{otherwise}.
    \end{cases}
  \end{displaymath}
  Clearly, \(d_{x} \in \divisp[1,1] \sqcup \divisp[2]\) for all \(x \in 2^{\Z}\) and
  \(d_{x} \in \divisp[1,1]\) precisely when \(x \in \free(2^{\Z})\).  The map \(\phi\) is then given
  by
  \[
    \mathrm{Free}(2^{\Z}) \ni x \mapsto \varphi(x)=d_{x} \in \divisp[1,1].
  \]
  Continuity and injectivity of \(\varphi\) are both clear, and a direct computation verifies that
  \(d_{s\cdot x}=\tau d_x\), so \(\varphi\) is equivariant.
\end{proof}

\begin{proposition}
  \label{prop:action-on-D1plus-not-smooth}
  The action \(\C \acts \divisp[1,1]\) does not have Borel transversals.
\end{proposition}

\begin{proof}
  Let \(K = \{s + \im t: 0 \le s, t < 1\} \subset \C\).  If \(T \subseteq \divisp[1,1]\) were a
  Borel transversal for \(\C \acts \divisp[1,1]\), then \(K \cdot T\) would be a Borel transversal
  for \(\Z \acts \divisp[1,1]\), which contradicts Lemma~\ref{lem:D11-no-transversal}.
\end{proof}

\begin{proof}[Proof of Theorem~\ref{thm:no-borel-equivariant-inverse-1-periodic}]
  Suppose towards a contradiction that there existed a Borel equivariant map
  \(\xi:\divisp[1]\to\holom[1]\) such that \(\dm(\xi(d))=d\). Since \(\dm\) maps \(\holom[1,1]\)
  onto \(\divisp[1,1]\) and \(\xi\) is equivariant, \(\xi\) maps \(\divisp[1,1]\) to
  \(\holom[1,1]\).

  Let \(T\) be a Borel transversal for \(\holom[1,1]\), whose existence is ensured by
  Proposition~\ref{prop:action-on-H1-smooth}. Then the pre-image \(\xi^{-1}(T)\) is a Borel
  transversal for \(\divisp[1,1]\), contradicting Lemma~\ref{lem:D11-no-transversal}.
\end{proof}

\begin{remark}
  \label{rem:yosida-type-normal-meromorphic}
  We note that, in contrast to Proposition~\ref{prop:action-on-H1-smooth}, the action
  \(\C \acts \merom[1]\) on the space of 1-periodic meromorphic functions is not concretely
  classifiable.  In fact, there are compact invariant subsets of \(\C \acts \merom[1]\).

  Consider a 1-periodic Yosida-type normal meromorphic function\footnote{ This is the class of
    meromorphic functions \(f\) on \(\C\) such that the family of translates \(\{f(z+w):w\in \C\}\)
    is normal with respect to locally uniform convergence and such that no limit function is
    constant. This class was thoroughly studied by Favorov in \cite{MR2403439}.}, which is not
  doubly-periodic.  For instance, let
  \[S_{a, b}(z) = e^{\pi \im (a-b)} \times \frac{\sin\pi (z-a)}{\sin\pi (z-b)},\]
  where \(a, b \in [0,1]\). This function is meromorphic, \(1\)-periodic, and satisfies
  \[
    S_{a, b}(z)=1 + O(1)e^{-2\pi |y|}
  \]
  for \(|y|\ge c>0\).  Hence, the infinite product \[F(z) = \prod_{k\in \Z} S_{a_k, b_k} (z- \im k)\]
  converges very fast to a 1-periodic meromorphic function. If the points \(a_k, b_k\) are uniformly
  separated, \(|a_k-b_k| \ge c >0\) for all \(k \in \Z\), then the function \(F\) is normal, i.e.,
  its orbit is a compact subset of \(\merom\), and each limiting meromorphic function is not
  constant. The freedom in the choice of \(a_k, b_k\) allows us to avoid doubly-periodic signed
  divisors in the limit, and hence, we can ensure that the closure of the orbit of \(F\) is a subset
  of \(\merom[1]\).

  The existence of compact invariant subsets of \(\merom[1]\) guarantees that \(\C \acts \merom[1]\)
  is not concretely classifiable. Moreover, the action admits invariant probability measures, which
  can be obtained via the classical Krylov--Bogolyubov construction.

  These invariant subspaces of \(\merom[1]\) demonstrate that the restriction to the free part in
  Theorem~\ref{thm:merormophic-quotient-entire} cannot be removed. Specifically, let \(Y\) be a
  closed shift-invariant subspace of \(\merom[1,1]\), the space of meromorphic functions with exact
  period group \(\Z\).  If
  \(T \subseteq \holom[1,1]\) is a Borel transversal for \(\C \acts \holom[1,1]\) (which exists by
  Proposition~\ref{prop:action-on-H1-smooth}), then
  \[\{(f_{1}, f_{2}) : f_{1} \in T, f_{2} \in \holom[1,1]\}\]
  is a Borel transversal for the diagonal action
  \[\C \acts \holom[1,1] \times \holom[1,1], \quad w \cdot (f_{1}, f_{2}) = (w \cdot f_{1}, w \cdot
    f_{2}).\]
  However, the action \(\C \acts Y\) admits no Borel transversals. The same argument as in the proof
  of Theorem~\ref{thm:no-borel-equivariant-inverse-1-periodic} shows that no Borel
  \(\C\)-equivariant map \(\psi : Y \to \holom[1,1] \times \holom[1,1]\) can exist.
\end{remark}

\subsection{Subharmonic functions and Riesz measures}\label{sec:subh-funct-riesz}
As shown in Section~\ref{sec:poisson-equation}, if \(\Gamma\) is a closed subgroup of \(\Rd\) with
dimension \(\dim \Gamma \le d-2\), then there exists a Borel equivariant
\(\xi : \measp[\Gamma]{\Rd} \to \sharm[\Gamma]{\Rd} \) that is a right-inverse to \(\Delta\).  We
now show that such equivariant inverses do not exist when \(\dim \Gamma = d-1\).

The group \(\Gamma\) is isomorphic to a group of the form \(\R^{q} \times \Z^{p}\) for some unique
choice of \(p\) and \(q\) satisfying \(p + q \le d\), where \(p+q\) is \(\dim \Gamma\).  Moreover,
there exists a matrix \(B \in \mathrm{GL}_{d}(\R)\) such that
\(\Gamma = B(\R^{q} \times \Z^{p} \times \{0\}^{d-p-q})\).

We aim to prove the following theorem.

\begin{theorem}
  \label{thm:subharmonic-large-stabilizers}
  Suppose that \(\Gamma\) is a closed subgroup of \(\Rd\) of dimension \(\dim \Gamma \ge
  d-1\). There are no Borel equivariant maps \(\xi : \measp[\Gamma]{\Rd} \to \sharm[\Gamma]{\Rd}\)
  satisfying \(\Delta\xi(\mu) = \mu\) for all \(\mu \in \measp[\Gamma]{\Rd}\).  If
  \(\dim \Gamma = d-1\), then the argument shift action \(\Rd \acts \sharm[\Gamma]{\Rd} \) admits a
  Borel transversal. Furthermore, the only subharmonic functions whose stabilizer has dimension
  \(d\) are the constant functions.
\end{theorem}

\subsubsection{The case \texorpdfstring{\(\dim \Gamma = d-1\)}{dim = d-1}.} Choose a
\((d-1)\)-dimensional vector subspace \(V\) of \(\Rd\) and a \((d-1)\)-dimensional lattice
\(\Lambda\) such that \(\Lambda \subseteq \Gamma \subseteq V\).  For instance, if
\(\Gamma = B(\R^{q} \times \Z^{p} \times \{0\})\), we can take
\[
  V=B (\R^{d-1}\times \{0\})\quad \textrm{and} \quad \Lambda = B(\Z^{d-1}\times \{0\}).
\]
We denote by \(D_0=D_0(\Lambda)\) a fundamental domain of the lattice \(\Lambda\).

Let \(e\) be one of the two (parallel) unit vectors in \(V^\perp\).  Any element in \(\Rd\) is
therefore uniquely represented as \(x+se\), with \(x\in V\) and \(s \in \R\).  Given a function
\(u\in \sharm[\Gamma]{\Rd}\) and \(s \in \R\), define
\begin{equation}
  \label{eq:def-Ms}
  M_u(s)=\sup\big\{u(x+se):x\in V \big\}=\sup\big\{u(x+se):x\in V\cap D_0 \big\},
\end{equation}
where the last equality follows from the \(\Gamma\)-periodicity of \(u\).  Since the closure
\(\overline{V\cap D_0}\) is compact and \(u\) is upper semi-continuous, the supremum is finite for
all~\(s\).  The following lemma is analogous to Lemma~\ref{lem:1-periodic-limit} and is the key to
establishing the existence of Borel transversals of \(\Rd \acts \sharm[\Gamma]{\Rd}\).

\begin{lemma}\label{lem:G-inv-SH-growth}
  Let \(\mathrm{dim}(\Gamma)=d-1\). For any \(u\in \sharm[\Gamma]{\Rd}\) one has
  \begin{equation}
    \label{eq:growth-G-inv-f}
    \lim_{s\to+\infty} M_{u}(s)=+\infty\quad\text{or}
    \quad \lim_{s\to-\infty} M_{u}(s)=+\infty,
  \end{equation}
  not excluding the case that both limits hold simultaneously.
\end{lemma}

\begin{remark}
  Note that, in contrast to the planar case, there are plenty of bounded subharmonic functions in
  \(\Rd\) for \(d\ge 3\).  However, the lemma asserts that boundedness, even along
  sequences of affine hyperplanes of the form \[\{x + t_j e: x\in V\}, \quad \{x- s_j e: x\in V\},\]
  is incompatible with \(\Gamma\)-periodicity.
\end{remark}

Before we proceed with the proof, we recall a subharmonic counterpart of
Lemma~\ref{lem:maximum-principle-strip}.
\begin{proposition}\label{prop:Lindelof}
  Let \(u\) be an upper bounded subharmonic function in the domain
  \(D = \{x \in \Rd : |x_d| < 1 \}\), with \(\limsup_{y\to x}u(y)\le 0\) for \(x\in \partial D\).
  Then \(u(x)\le 0\) everywhere on \(D\).
\end{proposition}

\noindent The proof is a reduction to the classical maximum principle for subharmonic function on
bounded domains (see, for example,~\cite{haymanSubharmonicFunctionsVol1976}*{Thm.~2.4}).

\begin{proof}
  For \(\epsilon>0\), let
  \[
    u_\epsilon(x)=u(x)-\epsilon\left(x_1^2+\cdots+x_{d-1}^2\right) +\epsilon(d-1)x_d^2.
  \]
  Then \(\Delta u_\epsilon=\Delta u\), so the function \(u_\epsilon\) is also subharmonic.  On the
  boundary \(\partial D\), we have the upper bound \(u_\epsilon\le (d-1)\epsilon\) (understood in
  the sense of upper limits), and by taking \(R\) large enough we can ensure that the same upper
  bound holds on the whole of \(\partial (B(0,R)\cap D)\).  Hence, by the classical maximum
  principle we have \(u_\epsilon(x)\le (d-1)\epsilon\) for all \(x\in B(0,R)\cap D\).  Since \(R\)
  can be chosen arbitrarily large, we get \(u_\epsilon(x)\le (d-1)\epsilon\) throughout \(D\).  In
  terms of the original subharmonic function \(u\), this means that
  \[
    u(x)\le (d-1)\epsilon +\epsilon\left(x_1^2+\cdots+x_{d-1}^2\right) -\epsilon(d-1)x_n^2
  \]
  on \(D\), which in particular implies that
  \[
    u(x)=\mathrm{O}(\epsilon) \quad\text{ as }\epsilon\to 0,
  \]
  the error term estimate being uniform on compact subsets.  Letting \(\epsilon\to 0\) we obtain the
  desired conclusion.
\end{proof}

We turn to the proof of Lemma~\ref{lem:G-inv-SH-growth}.

\begin{proof}[Proof of Lemma~\ref{lem:G-inv-SH-growth}]
  For the sake of contradiction, assume that \eqref{eq:growth-G-inv-f} fails.  Then, there exists a
  constant \(M\) such that for some sequences \((s_j)_j\) and \((t_{j})_{j}\) tending to
  \(+\infty\), both \(M_{u}(-s_j)\) and \(M_{u}(t_j)\) are bounded from above by \(M\).  That is to
  say, the upper bound \(u\le M\) holds on the parallel affine hyperplanes
  \[
    \{x-s_je: x\in V\}\quad\text{and}\quad \{x+t_je: x\in V\}, \qquad \textrm{ for } j \in \N.
  \]
  Moreover, for each \(j\) individually, \(u\) is bounded from above in the domain
  \[
    F_j=\left\{x+se: x\in V,\, -s_j<s<t_j\right\}.
  \]
  But then \(u\le M\) on \(F_j\) by Proposition~\ref{prop:Lindelof}, and, since \(F_j\) exhaust
  \(\Rd\), \(u\) is upper bounded by \(M\) globally.  Let \(\mu=\Delta u\) be the Riesz mass of
  \(u\). By \cite{haymanSubharmonicFunctionsVol1976}*{Thm.~3.20}, we have that
  \begin{equation}
    \label{eq:Riesz-bd}
    \int_0^\infty \frac{\mu(B(0,t))}{t^{d-1}} dt<+\infty.
  \end{equation}

  We conclude the proof by showing that \eqref{eq:Riesz-bd} is incompatible with
  \((d-1)\)-periodicity.  Let \(u_1,\ldots,u_{d-1}\) be a linearly independent set of elements of
  \(\Gamma\), and consider the truncated rectangular cylinder
  \[
    Q_R=\biggl\{se+\sum_{j=1}^{d-1}x_ju_j:0\le x_j\le 1,\, |s|\le R\biggr\}.
  \]
  For large enough \(R\), \(\mu(Q_R)\ge c_0>0\), and by the \(\Gamma\)-invariance of \(\mu\), we
  have \(\mu(Q_R + x)=\mu(Q_R)\) for any \(x\) in the lattice generated by
  \(\Lambda=\{u_1,\ldots,u_{d-1}\}\).  Moreover, there exists a constant \(d_0>0\), depending only
  on \(R\) and \(\Lambda\), such that for \(t\) large enough the ball \(B(0,t)\) contains at least
  \(d_0 t^{d-1}\) disjoint translates of \(Q_R\) by elements of \(\Lambda\). Hence,
  \[
  \mu(B(0,t))\ge d_0t^{d-1}\mu(Q_R)\ge c_0d_0 t^{d-1}.
  \]
  But this contradicts the integrability \eqref{eq:Riesz-bd}, and we have reached a contradiction.
\end{proof}

Empowered by Lemma~\ref{lem:G-inv-SH-growth}, the action \(\Rd \acts \sharm[\Gamma]{\Rd}\) can be
seen to admit a Borel transversal, by an argument that parallels
Proposition~\ref{prop:action-on-H1-smooth}.  On the other hand, a construction similar to the one
given in Lemma~\ref{lem:D11-no-transversal} for divisors, shows that
\(\Rd \acts \measp[\Gamma]{\Rd}\) does not.  These two facts let us conclude that there are no
\(\Rd\)-equivariant Borel right-inverses to \(\Delta\) on \(\sharm[\Gamma]{\Rd}\).

\subsubsection{The case \texorpdfstring{\(\dim \Gamma = d\)}{dim = d}}
\noindent It is even simpler to exclude existence of equivariant inverses to \(\Delta\) on
\(\measp[\Gamma]{\Rd}\), when stabilizers have full dimension.  The key observation is the
following.

\begin{lemma}
  \label{lem:non-exist-SH-fulldim}
  Let \(\Gamma\) be a closed subgroup of \(\Rd\) of dimension \(d\).  Then any \(\Gamma\)-invariant
  subharmonic function is constant.
\end{lemma}

In contrast, \(\measp[\Gamma]{\Rd}\) contains non-zero elements whenever \(\Gamma\) is a proper
subgroup of \(\Rd\).  The proof of Lemma~\ref{lem:non-exist-SH-fulldim} therefore concludes the
proof of Theorem~\ref{thm:subharmonic-large-stabilizers}.

\begin{proof}
  Since any subharmonic function can be approximated by smooth subharmonic functions, it is enough
  to show that any smooth \(\Gamma\)-invariant subharmonic function is constant. If
  \(u\in C^2(\Rd)\), then for any bounded domain \(D\) with regular boundary, Green's formula gives
  \begin{equation}
    \label{eq:Green-on-G-box}
    \int_{D} \Delta u\,dm=\int_{\partial D}\nabla u\cdot {\rm n}\,dS,
  \end{equation}
  where \(\mathrm{d}S\) denotes the surface element and \({\rm n}\) is the outward unit normal to
  \(D\). We apply this with \(D\) being a box
  \[
    D=\biggl\{\sum_{j=1}^d x_ju_j: x\in (0,1)^d\biggr\}
  \]
  with \(\{u_j:1\le j\le d\}\) a set of linearly independent elements of \(\Gamma\). Then, by
  periodicity, the right-hand side of \eqref{eq:Green-on-G-box} vanishes (for any \(k\), the
  integrand coincides on the two faces \(\{\sum_j x_ju_j:x_k=0\}\) and \(\{\sum_j x_ju_j:x_k=1\}\),
  but the direction of the normal is reversed).  Hence \(u\) is harmonic, and, since it is
  \(\Gamma\)-periodic, it is bounded and thus constant by the Liouville's theorem for harmonic
  functions~\cite{nelsonProofLiouvillesTheorem1961}.
\end{proof}

\subsection[Non-existence of the d-bar equation]{Non-existence for the
  \texorpdfstring{\(\bar\partial\)-equation}{Non-existence of the d-bar equation}}
\label{s:d-bar}
\noindent Let \(\smooth[1]{\C,\C}\) denote the space of smooth complex-valued functions on
\(\C\) with a discrete \(1\)-dimensional stabilizer. The goal of this section is to prove
the following result.
\begin{theorem}
  \label{thm:non-exist-dbar}
  There does not exist any equivariant Borel right-inverse to \(\bar\partial\) on
  \(\smooth[1]{\C,\C}\).
\end{theorem}

\begin{proof}
  Informally speaking, the idea is to start with an appropriate \emph{non-stationary} random
  \(\smooth[1]{\C,\C}\)-function \(f\) such that \(g=\bar\partial f\) is stationary.  If an
  equivariant Borel right-inverse \(\Psi\) to \(\bar\partial\) existed, then \(F=\Psi(g)\) would be
  a \emph{stationary} random function such that \(F-f\) is entire.  But by the ergodic theorem
  combined with \(1\)-periodicity, \(F-f\) can be shown to be a constant, contradicting the
  non-stationarity of \(f\).

  \medskip

  Let us fill in the missing details. We let \(\xi\) be a stationary random element of
  \(\smooth[1]{\C,\C}\). For definiteness, we may take a stationary process \(\nu\) on
  \(\R\) with smooth and non-periodic trajectories, and put
  \[ \xi(x+\im y)=\sin(2\pi x)\nu(y). \]
  Then let \(f(x+\im y)=-2\im y+ \xi(x+\im y)\).  Since \(\bar\partial f=1+\bar\partial \xi\), the
  function \(f\) satisfies the desired properties.  Next, note that \(\partial(F-f)\) is a
  \(1\)-periodic stationary random entire function, and hence constant. The argument is similar to
  the one used in Section~\ref{sec:divis-merom-funct}.  Indeed, by the ergodic theorem, there is a
  constant \(M\) and sequences \(s_j,t_j\to+\infty\) such that \(|\partial(F-f)|\) is upper bounded
  by \(M\) on
  \[
    \left\{x-\im s_j:x\in [0,1]\right\}\cup \left\{x+\im t_j:x\in [0,1]\right\}.
  \]
  By Lindelöf's maximum principle we have that \(|\partial(F-f)(x+\im y)|\le M\) for
  \(y\in [-s_j,t_j]\). But then \(|\partial (F-f)|\) is bounded on \(\C\) and hence
  constant. It thus follows that \(F(z)=f(z)+az+b\), and by \(1\)-periodicity we can further claim
  that \(F(z)=f(z)+b\).

  \medskip

  To finish the proof, we claim that this is a contradiction. Note that there are non-stationary
  random functions which become stationary after adding a constant---take e.g., \(\nu(t)-\nu(0)\)
  with \(\nu\) a stationary process on \(\R\) as above.  However, this is not the case for
  our random function \(f\). By the Birkhoff ergodic theorem, we get that for fixed \(M>0\),
  \[
    \lim_{R\to\infty}\frac{\big|\big\{y\in[-R,R]: \max_{x\in[0,1]}|\xi(x+\im y)|\le
      M\big\}\big|}{2R} =\mathbb{E}\left(\bbone_{\{|\xi(x)|\le M\text{ for
        }x\in[0,1]\}}|\mathcal{I}\right),
  \]
  where \(\mathcal{I}\) is the \(\sigma\)-algebra of invariant events.  In particular, as
  \(M\to\infty\) the right-hand side tends to \(1\) in \(L^1\).  Looking at the corresponding limit
  for \(F\), we have in the same way that
  \begin{equation}
    \label{eq:F-aver}
    \lim_{M\to\infty}\lim_{R\to\infty}\frac{\big|\big\{y\in[-R,R]:
      \max_{x\in[0,1]}|F(x+\im y)|\le M\big\}\big|}{2R}=1
  \end{equation}
  in \(L^1\).  Now, if \(|\xi(x+\im y)|\le M\) and \(|y|>M\), then the reverse triangle inequality
  gives the lower bound \(|f(x+\im y)|\ge |2y|-|\xi|>M\) so that
  \[
    \Big\{\max_{x\in[0,1]}|f(x+\im y)|>M\Big\}\supseteq \Big\{\max_{x\in[0,1]}|\xi(x+\im y)|\le
    M\Big\}\cap\Big\{|y|>M\Big\}.
  \]
  As a consequence, the ergodic averages pertaining to \(f\) satisfy
  \begin{multline*}
    \limsup_{R\to\infty}\frac{\big|\big\{y\in[-R,R]:
      \max_{x\in[0,1]}|f(x+\im y)|\le M\big\}\big|}{2R}\\
    \le 1-\mathbb{E}\left(\bbone_{|\xi(x)|\le M\text{ for }x\in[0,1]}|\mathcal{I}\right).
  \end{multline*}
  Taking finally the limit as \(M\to\infty\), we have the \(L^1\)-convergence
  \begin{equation}
    \label{eq:f-aver}
    \lim_{M\to\infty}\lim_{R\to\infty}\frac{\big|\big\{y\in[-R,R]:
      \max_{x\in[0,1]}|f(x+\im y)|\le M\big\}\big|}{2R}=0.
  \end{equation}
  The conclusions \eqref{eq:F-aver} and \eqref{eq:f-aver} are clearly incompatible with \(F-f\)
  being constant, and hence we have reached a contradiction.
\end{proof}

\section{Non-existence condition for equivariant inverses}
\label{sec:lack-equiv-invers}

In this section, we continue our investigation of the lifting problem.  While
Theorem~\ref{thm:main-theorem} gives sufficient conditions for the existence of Borel liftings, here
we give several examples where liftings do not exist.

The prototype example is the Borel entire function \(F\) generated by the action
\(\C \acts \mathcal E\) by translation of the argument \(f\mapsto f(\cdot +z)\), and the map
\(\pi\colon \mathcal E \to \mathcal E\), \(f\mapsto f'\).  As it turns out, this map \(\pi\) does
not have a Borel equivariant (right) inverse, i.e., \(F\) does not have a Borel primitive.  There
are several other natural maps \(\pi\) without Borel equivariant right inverses.  The general
framework, that covers the derivative, as well as other maps \(\pi\), is given in
Theorem~\ref{thm:sufficient-condition} and uses the classical Birkhoff observation that the action
\(\C \acts \mathcal E\) by translation of the argument has a dense orbit, (i.e., is topologically
transitive) combined with the Baire category argument.

\subsection{Prerequisites}
\label{subsect:Prerequisites}

\subsubsection{The Baire property}
Recall that a subset \(A \subseteq Z\) of a Polish space \(Z\) is \emph{meager} (or of the first
category) if there are nowhere dense sets \(F_n\subseteq Z\), \(n\in\N\), such that
\(A\subseteq \bigcup_n F_n\). The set \(A\) is \emph{comeager} (or residual) if its complement
\(Z\setminus A\) is meager, that is, there exist dense open sets \(U_n\), \(n\in \N\), such that
\(A\supseteq \bigcap_n U_n\).  By the Baire category theorem, every comeager set is dense in \(Z\).

The set \(A \subseteq Z\) has the \emph{Baire property} if there exists an open set \(U\) such that
the set \(A\symd U\) is meager, i.e., \(A\) can be represented in the form \(A=U\symd M\), where
\(U\) is open, and \(M\) is meager.  Note that all open sets and all meager sets have the Baire
property.  We recall several basic facts about the sets with the Baire properties, which we will be
using.

\begin{enumerate}[leftmargin=.6cm, label={\sf \arabic*}., ref={\sf \arabic*}]
\item The class of sets having the Baire property is a \(\sigma\)-algebra containing all open
  sets. In particular, all Borel sets have the Baire property.
\item A somewhat deeper fact is that all analytic sets have the Baire
  property~\cite{Kechris}*{21.6}, \cite{MR1619545}*{Thm.~4.3.2}.  To see this, one needs to observe
  that closed sets possess the Baire property, and then to use the fact that the class of sets
  having the Baire property is closed under the Suslin scheme, and that every analytic set can be
  obtained by applying the Suslin scheme to closed sets.

\item For a set \(A\subseteq Z\) with the Baire property, put
  \[
    U(A) = \bigcup \bigl\{U:U\ {\rm open}, U\setminus A\ {\rm is\ meager} \bigr\}.
  \]
  Then,
  \begin{enumerate}[leftmargin=.8cm]
  \item[{\sf (a)}] \(A \symd U(A)\) is meager;
  \item[{\sf (b)}] the open set \(U(A)\) is \emph{regular}, i.e., it equals the interior of its
    closure;
  \item[{\sf (c)}] \(U(A)\) is the unique regular open set \(U\) such that \(A\symd U\) is meager.
  \end{enumerate}
  For the details, see \cite{Kechris}*{8G}.
\end{enumerate}

\subsubsection{Generic ergodicity}
Let \(G \acts Z\) be a continuous action of a Polish group \(G\) on a Polish space \(Z\).  The
action is called \emph{generically ergodic} if every \(G\)-invariant set in \(Z\) with the Baire
property is either meager, or comeager.

Note that if \(A \subseteq Z\) is a \(G\)-invariant set that has the Baire property, and
\(A = U \symd M\) for a regular open \(U\) and a meager \(M\), then \(U\) is \(G\)-invariant.
Indeed, for \(g\in G\),
\[
  U \symd M = A = gA = gU \symd gM
\]
implies \(U=gU\) since both \(U\) and \(gU\) are regular open, while \(M\) and \(gM\) are both
meager.

Furthermore, generic ergodicity is equivalent to the existence of a dense orbit (aka
\emph{topological transitivity}).  Indeed, if \(G \acts Z\) is generically ergodic and \((U_n)_n\)
is a countable basis of non-empty open sets for the topology of \(Z\), then the \(G\)-invariant set
\(\bigcap_n GU_n\) is comeager and therefore dense, and, in fact, the orbit of each
\(z\in \bigcap_n GU_n\) is dense. Conversely, let \(G \curvearrowright Z\) have a dense
orbit. Denote by \(A\) a \(G\)-invariant set with the Baire property.  Then \(A = U \symd M\), where
\(U\) is regular open and \(M\) is meager.  Then \(U\) is also \(G\)-invariant, so it contains a
dense orbit and is therefore dense (assuming \(U\) is non-empty). Consequently, either
\(U = \varnothing\) and \(A\) is meager, or \(U\) is a dense open set and \(A\) is comeager; in
other words, the action is generically ergodic.

\subsection{Non-existence condition}
\label{subsect:non-exist}
We are now ready to formulate our \emph{non-existence condition} for Borel equivariant liftings.
Here is the setting of our result:

\paragraph{1.}
Let \(G\acts Z\) be a continuous action by a Polish group \(G\) on a Polish space \(Z\), and
\(G\acts Y\) be a Borel action on a standard Borel space \(Y\).

\paragraph{2.}
Let \(\pi: Z\to Y\) a Borel equivariant map.  Our aim is to give conditions under which \(\pi\) does
not have a Borel equivariant right inverse.

\paragraph{3.}
Suppose that \(H\acts Z\) is a free continuous action by a Polish group \(H\), whose orbits are
classified by \(\pi\):
\begin{equation}\label{eq:pi-H}
  \pi (z_1)=\pi(z_2) \Longleftrightarrow z_1 E_H z_2.
\end{equation}

\paragraph{4.}
Let \(\tau : G \acts H\) be a continuous action by automorphisms, and let \(H \rtimes_{\tau} G\) be
the corresponding semidirect product.  We assume that \(H \rtimes_{\tau} G\) acts continuously on
\(Z\), and that this action is compatible with the actions \(G \acts Z\) and \(H \acts Z\).

\bigskip

There are some differences with the setting of the main result,
cf.~Section~\ref{sec:standing_assumptions}. Most notably, \(Z\) is assumed to be a Polish space and
\(H \rtimes_{\tau} G \acts Z\) is assumed to be continuous.

\begin{theorem}
  \label{thm:sufficient-condition}
  If \(G \acts Z\) is generically ergodic while \(\tau:G\acts H\) is not, then there are no Borel
  \(G\)-equivariant inverses to \(Z \to \pi(Z)\).  Moreover, if \(Y' \subseteq Y \cap \pi(Z)\) is a
  Borel \(G\)-invariant set such that \(Z' = \pi^{-1}(Y')\) is comeager in \(Z\), then there are no
  Borel \(G\)-equivariant inverses to \(\pi|_{Z'} : Z' \to Y'\) either.
\end{theorem}

\begin{proof}
  Suppose that \(\xi : Y' \to Z'\) is Borel \(G\)-equivariant, satisfies \(\pi(\xi(y)) = y\) for all
  \(y \in Y'\), and that \(Z'\) is comeager in \(Z\).  Such a map \(\xi\) is necessarily injective.
  Consider the set \(T=\xi(Y')\). It is Borel by the Luzin--Suslin theorem (\cite{Kechris}*{Thm.\
    15.1}, \cite{MR1619545}*{Prop.~4.5.1}), it is \(G\)-invariant (by the equivariance of \(\xi\)),
  intersects each \(H\)-orbit at most once (by Eq.~\eqref{eq:pi-H}), and its saturation
  \(H \cdot T = \pi^{-1}(Y') = Z'\) is comeager in \(Z\).

  Suppose that the action \(\tau\) is not generically ergodic.  Then there exists a \(G\)-invariant
  Borel set \(A\subset H\) such that both \(A\) and \(H\setminus A\) are non-meager subsets of
  \(H\).  Consider the disjoint subsets \(A\cdot T\) and \((H\setminus A)\cdot T\) of \(Z\). We
  claim that these sets are (i) \(G\)-invariant, (ii) Borel, (iii) non-meager. This contradicts the
  generic ergodicity of the action \(G \acts Z\).  It suffices to check these properties for the set
  \(A\cdot T\).

  The \(G\)-invariance of \(A\cdot T\) follows from the relation \(g(hz)=(\tau^g h)(gz)\) combined
  with the \(G\)-invariance of \(A\) and \(T\).  To check this relation, we note that, by
  assumption, \(G\) and \(H\) act on \(Z\) as follows: \(gz=(e_H, g)z\), \(hz=(h, e_G)z\), where
  \(e_H\) and \(e_G\) are the units in \(H\) and \(G\). Thus,
  \[ g(hz)=g((h, e_G)z)=(e_H, g)(h, e_G)z=(e_H \tau^g h, g)z = (\tau^g h, g)z, \] and therefore,
  \[
    (\tau^g h)(gz) = (\tau^g h, e_G)(e_H, g)z = (\tau^g h, g)z = g(hz).
  \]

  Borelness of \(A\cdot T\) follows from applying the Luzin--Suslin theorem to the continuous
  injective map \(\psi\colon (h, t)\mapsto ht\).

  To see that the set \(A\cdot T\) is non-meager, first note that since \(A\) is non-meager in \(H\)
  and \(H\) is Polish, there exists a non-empty open set \(O\subseteq H\) such that \(A\cap O\) is
  comeager in \(O\).  Then \((A\cap O)\times T\) is comeager in \(O\times T\), and therefore,
  \(A \times T\) cannot be meager in \(H\times T\). It remains to note that
  \(A\cdot T = \psi (A\times T)\), where the map \(\psi\colon A\times T \to Z'\) is continuous and
  injective, so it cannot map non-meager sets onto meager ones. Thus \(A\cdot T\) is non-meager in
  \(Z'=H\cdot T\), and since \(Z'\) is comeager in \(Z\), \(A\cdot T\) remains non-meager in \(Z\).
\end{proof}

\subsection{Applications}
\label{subsect:applications}

Theorem~\ref{thm:sufficient-condition} has multiple applications when \(G = \C\) acts on the space
of entire functions \(Z = \holom\) via the argument shift.  Note that, by Birkhoff's theorem,
\(\C \acts \holom\) has dense orbits and is thus generically ergodic.  In all of the following
examples \(\tau\) is the trivial action (that is, \(\tau^g h = h\) for all \(g\) and \(h\)), and the
semidirect product is really just the direct product \(\C \times H\).  In particular,
\(\tau : \C \acts H\) is automatically not generically ergodic provided that \(H\) is non-trivial.
In the following examples, elements of \(Y\) are functions, and \(\C\) acts on \(Y\) via the
argument shift.  For each of the following maps \(\pi\), there are no Borel equivariant inverses
\(\xi : \pi(Z) \to Z\).
\begin{enumerate}[leftmargin=.6cm, label={\sf \arabic*}., ref={\sf \arabic*}]
\item The derivative \(\frac{d}{dz} : \holom \to \holom\).  Here \(H = \C\) is viewed as the space
  of constant functions and \(H \acts \holom\) by \((w \cdot f)(z) = f(z) + w\).
\item The logarithmic derivative \(L : \holom[\ne 0] \to \merom\), \(f \mapsto f'/f\).  Here \(H\)
  is the multiplicative group \(\C^{\times}\), and \(H \acts \holom[\ne 0]\) by multiplication,
  \((w \cdot f)(z) = wf(z)\).
\item The exponentiation \(\exp : \holom \to \holomnz\).  Here \(H = \Z\) is identified with
  constant functions of the form \(2\pi \im n\), \(n \in \Z\) and \(H \acts \holom\) by
  \((n \cdot f)(z) = f(z) + 2\pi \im n\).
\item The absolute value map \(|\cdot | : \holom \to C(\C)\).  Here \(H\) is the circle \(\T\),
  identified with the multiplicative group of complex numbers of absolute value \(1\).
  \(H \acts \holom\) via \((\alpha \cdot f)(z) = e^{\im \alpha}f(z)\) for \(\alpha \in [0,2\pi)\).
\item The Schwarzian derivative \(S : \holom \setminus \{\textrm{constants}\} \to \merom\), where
  \[Sf = \Bigl(\frac{f''}{f'}\Bigr)' - \frac{1}{2}\Bigl(\frac{f''}{f'}\Bigr)^{2}.\]
  Here \(H\) is the group of M\"obius transformations \(\mathrm{PGL}(2,\C)\) acting by
  post-composition:
  \[f \mapsto \frac{af + b}{cf+d}.\]
\item The spherical derivative
  \((\,\cdot\,)^{\#} : \holom \setminus \{\textrm{constants}\} \to C(\C)\) where
  \[f^{\#} = \frac{2|f'|}{1+|f|^{2}}.\]
  Here \(H\) is the special unitary group \(\mathrm{SU}(2)\) acting on \(\holom\) by
  post-composition:
  \[f \mapsto \frac{\alpha f - \bar{\beta}}{\beta f + \bar{\alpha}} \qquad \alpha, \beta \in \C
    \textrm{ and } |\alpha|^{2} + |\beta|^{2} = 1.\]
\end{enumerate}

An application of Theorem~\ref{thm:sufficient-condition} with a non-trivial action \(\tau\) is given
by the difference operator \(D : \holom \to \holom\) defined by \(D(f)(z) = f(z+1)-f(z)\).  Here
\(H = \holom[1]\) is the space of \(1\)-periodic entire functions, \(H \acts \holom\) via
\((h\cdot f)(z) = f(z) + h(z)\), and \(\tau\) is the argument shift. The action
\(\tau : \C \acts \holom[1]\) has a Borel transversal as was argued in
Remark~\ref{rem:1-periodic-borel-transversal}. Hence, it cannot be generically ergodic by the
following standard fact, whose proof is included for the reader's convenience.

\begin{lemma}
  \label{lem:smooth-not-generically-ergodic}
  Consider a continuous action \(G \acts Z\) of a Polish group \(G\) on a Polish space \(Z\). Assume
  that no orbit of the action is comeager. If the action has a Borel transversal, then it cannot be
  generically ergodic.
\end{lemma}

\begin{proof}
  First, consider a simpler case with countably many orbits. Note that every orbit is an analytic
  set (as the image of \(G\times\{z\}\) under the continuous map \((g, z)\mapsto gz\), \(g\in G\),
  \(z\in Z\), and therefore, has the Baire property).  All the orbits cannot be meager (otherwise,
  \(Z\) becomes meager in itself, which contradicts the Baire category theorem). There could not be
  a unique non-meager orbit---in that case, its complement is meager as a countable union of meager
  sets, and the orbit becomes comeager, which contradicts the assumption. Hence, there are at least
  two non-meager orbits, and we can decompose \(Z=Z_1 \sqcup Z_2\), where the sets \(Z_1\) and
  \(Z_2\) are \(G\)-invariant and have the Baire property.  Each of them contains a non-meager
  subsets, hence, the action cannot be generically ergodic.

  Now, we assume that the action has uncountably many orbits, and let \(T \subseteq Z\) be the Borel
  transversal of the action.  By the assumption, the set \(T\) is uncountable.  Fix a Borel
  bijection \(\alpha : 2^{\N} \to T\) between \(T\) and the Cantor set \(2^{\N}\). For each finite
  string \(s \in 2^{<\N}\), define \(Z_{s} = G \cdot \alpha(N_{s})\), where
  \(N_{s} \subseteq 2^{\N}\) consists of all infinite binary extensions of \(s\). Each \(Z_{s}\) is
  \(G\)-invariant and satisfies \(Z_{s} = Z_{s \concat 0} \sqcup Z_{s \concat 1}\) (the sign
  \say{\( \concat \)} denotes concatenation of finite strings). Furthermore, the sets \(Z_{s}\) are
  analytic (as the images of Borel sets \(G\times N_s\) by the Borel map
  \((g, x) \mapsto g \alpha (x)\), \(g\in G\), \(x\in N_s\)), and therefore possess the Baire
  property.

  Assume towards a contradiction that \(G \acts Z\) is generically ergodic. Since
  \(Z = Z_{0} \sqcup Z_{1}\) and both \(Z_{0}\) and \(Z_{1}\) are \(G\)-invariant with the Baire
  property, one of them must be comeager. Let \(i_{1} \in \{0,1\}\) be such that \(Z_{i_{1}}\) is
  comeager. Proceeding inductively, at each step \(n\) we have
  \(Z_{i_{1}\cdots i_{n}} = Z_{i_{1}\cdots i_{n}0} \sqcup Z_{i_{1}\cdots i_{n}1}\), and we choose
  \(i_{n+1}\) such that \(Z_{i_{1}\cdots i_{n+1}}\) is comeager.

  This yields a sequence \(a = (i_{n})_{n} \in 2^{\N}\) for which each \(Z_{i_{1}\cdots i_{n}}\) is
  \(G\)-invariant and comeager. Consequently, \(\bigcap_{n} Z_{i_{1}\cdots i_{n}}\) is comeager as
  well. However, this intersection equals the single orbit \(G\alpha(a)\), contradicting the
  assumption that no orbit is comeager.
\end{proof}

\subsection{Two remarks}
\label{subsect:two-remarks}

{\sf 1.} All our non-existence applications pertained to the action
\(\C \curvearrowright \mathcal E\) by the argument shift, i.e., to the Borel entire function
\(\C\curvearrowright\free(\mathcal E) \stackrel{\rm id}\to\C\curvearrowright\mathcal E\).  This does
not exclude the possibility that for other Borel measurable entire functions
\(\C\curvearrowright X \stackrel{\varphi}\to\C\curvearrowright\mathcal E\), the same map \(\pi\)
might have a Borel equivariant inverse.  For one, we have the tautological example: If
\(F:X\to \holom\) is a Borel entire function, then the Borel entire function \(F'\) given by
\(x\mapsto F'_x\) evidently admits the Borel entire primitive \(F\). Moreover, a modification of the
original Weiss's construction~\cite{MR1422707} can be used to build Borel entire functions that admit
primitives of all orders.

\medskip\noindent{\sf 2.}  It is also worth mentioning that, for some of the maps \(\pi\) discussed
above, it is possible to formulate a cohomological criterion for the existence of the Borel
equivariant inverse to \(\pi\).  For instance, this can be done when \(\pi = {\rm d}/{\rm d}z\) is
the derivative map. With any Borel entire function \(\varphi\colon X\to\mathcal E\), we associate
the \emph{cocycle} \(\alpha\colon X\times \C \to \C\),
\[
  \alpha_\varphi (x, z) = \int_0^z \varphi_x(w) \, {\rm d}w.
\]
Then, \emph{a Borel entire function \(\varphi\) admits a Borel primitive if and only if
  \(\alpha_\varphi\) is a coboundary}, that is, there is a Borel function \(g\colon X\to \C\) such
that \(\alpha (x, z) = g(z\cdot x) - g(x)\) for all \(x\in X\) and \(z\in \C\). If such \(g\)
exists, then \(g(x)=\varphi_x(0)\), the evaluation of \(\varphi\) at the origin.

\bigskip \bigskip

\appendix
\section*{Appendices}

\section{Growth of Borel continuous functions}
\label{sec:growth-introduction}

The purpose of this appendix is to highlight the differences in the growth rates between measurable
and Borel entire functions by showing that for some free Borel \(\C\)-actions, any nowhere constant
Borel entire function will have arbitrarily fast growth.  In fact, this phenomenon has little to do
with the structure of entire functions, and applies to arbitrary continuous functions.

The main result, namely Theorem~\ref{thm:growth-rate-limsup}, is derived from the following theorem
from~\cite{gaoForcingConstructionsCountable2022}*{Thm.~1.1}.
\begin{theorem}
  \label{thm:gjks}
  Let \(\Zd \acts X'\) be a continuous action on a compact Polish space.  Let also \((S'_{n})_{n}\),
  \(S'_{n} \subseteq X'\), be a sequence of Borel complete sections\footnote{A section is complete
    if it intersects each orbit of the action.} and \((A_{n})_{n}\), \(A_{n} \subseteq \Zd\), be an
  increasing sequence of finite subsets of \(\Zd\) such that \(\bigcup_{n}A_{n} = \Zd\).  Then there
  is \(x \in X'\) such that \((A_{n} \cdot x) \cap S'_{n} \ne \varnothing\) for infinitely many
  \(n\).
\end{theorem}

Let \((\Omega, \norm{\cdot})\) be a separable Banach space and \(C(\Rd, \Omega)\) be the space of
\(\Omega\)-valued continuous functions on \(\Rd\) endowed with the topology of uniform convergence
on compact subsets.  Naturally, we have the argument shift action \(\Rd\acts C(\Rd, \Omega)\).  Let
\(\Rd \acts X\) be a free Borel action.  Given a Borel \(\Rd\)-equivariant map
\(\phi : X \to C(\Rd, \Omega)\) let \(M_{\phi, x} : \R^{> 0} \to \R^{\ge 0}\) be given by
\[M_{\phi, x}(R) = \max_{|r|\le R} \norm{\phi(x)(r)}.\]
We say that \(\phi\) is \emph{everywhere unbounded} if for all \(x \in X\) one has
\(M_{\phi,x}(R) \to +\infty\) as \(R \to +\infty\).  By a \emph{rate function} we mean a
non-decreasing \(f : \R^{\ge 0} \to \R^{\ge 0}\) satisfying \(\lim_{R \to \infty}f(R) = +\infty\).

\begin{theorem}
  \label{thm:growth-rate-limsup}
  For each \(d\), there is a free Borel action \(\Rd \acts X\) such that for any everywhere
  unbounded Borel equivariant \(\phi : X \to C(\Rd, \Omega)\) and any rate function \(f\) there is a
  point \(x \in X\) such that
  \[\limsup_{R \to \infty} \dfrac{M_{\phi, x}(R)}{f(R)} = +\infty.\]
\end{theorem}

\begin{proof}
  Let \(K \subseteq \Rd\) be a compact symmetric subset so large that \(\Zd + K = \Rd\), and let
  \(\diam K\) denote its diameter.  For instance, \( K = [-1/2,1/2]^{d}\) with
  \(\diam K = \sqrt{d}\) will do.  Note that \((r + K) \cap \Zd \ne \varnothing\) for any
  \(r \in \Rd\).  It is notationally more convenient to prove the weaker inequality
  \begin{equation}
    \label{eq:rate-ge-one}
    \limsup_{R \to \infty}\dfrac{M_{\phi, x}(R+ \diam K)}{f(R)} \ge 1.
  \end{equation}
  The latter, however, is equivalent to the conclusion of Theorem~\ref{thm:growth-rate-limsup}, for
  Eq.~(\ref{eq:rate-ge-one}) applied to \(f_1(R) = R f(R+ \diam K)\) yields the conclusion of
  Theorem~\ref{thm:growth-rate-limsup} for~\(f\).

  Let \(\Zd \acts X'\) be a free continuous action of \(\Zd\) on some compact Polish space~\(X'\).
  We turn it into a free action of \(\Rd\) on \(X = X' \times [0,1)^{d}\) in the natural way:
  \[r \cdot (x,s) = (\lfloor s+r \rfloor \cdot x, s + r - \lfloor s+r \rfloor),\]
  where \(\lfloor r \rfloor \in \Zd\) is the coordinate-wise integer part of the vector
  \(r \in \Rd\).  Note that~\(X'\) can be identified with the \(\Zd\)-invariant subset
  \(X' \times \{0\} \subseteq X\).

  Fix an everywhere unbounded Borel equivariant \(\phi : X \to C(\Rd, \Omega)\) and a growth rate
  function \(f\). Our goal is to find an \(x \in X\) such that Eq.~(\ref{eq:rate-ge-one}) holds. Let
  \(\Phi : X \to \Omega\) be given by \(\Phi(x) = \phi(x)(0)\) and let
  \[B(R) = \{a \in \Omega : \norm{a} < R\} \subseteq \Omega \]
  denote the open ball of radius \(R\) around the origin. Given \(n \in \N\), set
  \[S_{n} = \Phi^{-1}(\Omega \setminus B(f(n))) = \{ x \in X : \norm{\Phi(x)} \ge f(n)\}. \]
  Everywhere unboundedness of \(\phi\) is equivalent to the assertion that each \(S_{n}\) is
  complete, i.e., it intersects every orbit of the action.  Furthermore, the sets \(S_{n}\) are
  nested, \(S_{n} \supseteq S_{n+1}\), and vanish, \(\bigcap_{n} S_{n} = \varnothing\).  In fact,
  the latter holds in the following stronger sense: for any compact \(L \subseteq \Rd\) and any
  \(x \in X\) there is \(N \in \N\) such that for all \(n \ge N\) one has
  \((L \cdot x) \cap S_{n} = \varnothing\).

  Set \(S_{n}' = (K \cdot S_{n}) \cap X'\) and note that each \(S_{n}'\) is a Borel complete section
  for the action \(\Zd \acts X'\).  One may now apply Theorem~\ref{thm:gjks} to sets \(S'_{n}\) and
  \(A_{n} = \{ a \in \Zd : \norm{a} \le n\}\), \(n \in \N\). It yields an \(x \in X'\) such that
  \((A_{n} \cdot x ) \cap S_{n}' \ne \varnothing\) for infinitely many \(n\).  In other words, there
  are \(n_{k}\), \(a_{k} \in A_{n_{k}}\), \(y_{k}\in S_{n_{k}}\) and \(s_{k} \in K\) such that
  \(\norm{a_{k}} \le n_{k}\) and \(a_{k}x = s_{k}y_{k}\).  Note that
  \begin{displaymath}
    \begin{aligned}
      \norm{\rho(x, y_{k})}
      &\le \norm{\rho(x,a_{k}x)} + \norm{\rho(a_{k}x, y_{k})} \\
      &= \norm{\rho(x,a_{k}x)} + \norm{\rho(s_{k}y_{k}, y_{k})} \\
      &= \norm{a_{k}} + \norm{s_{k}} \le n_{k} + \diam K.
    \end{aligned}
  \end{displaymath}
  In particular,
  \begin{displaymath}
    \begin{aligned}
      M_{\phi, x}(n_{k} + \diam K) &\ge \norm{\phi(x)(\rho(x,y_{k}))}
                                     = \norm{\Phi(\rho(x,y_{k})x)} \\
                                   &=
                                     \norm{\Phi(y_{k})} \ge f(n_{k}),
    \end{aligned}
  \end{displaymath}
  and therefore \(R = n_{k}\), \(k \in \N\), witness that Eq.~(\ref{eq:rate-ge-one}) does indeed
  hold.
\end{proof}

Specializing to entire functions, we get the following corollary.

\begin{corollary}
  \label{cor:borel-entire-unbounded-growth}
  There is a free Borel action \(\C \acts X\) such that for any rate function \(f\) and any Borel
  equivariant \(\phi : X \to \holom \setminus \{\textrm{constants}\}\) there is a point \(x \in X\)
  such that
  \[\limsup_{R \to \infty} \dfrac{M_{\phi, x}(R)}{f(R)} = +\infty.\]
\end{corollary}

\begin{proof}
  Apply Theorem~\ref{thm:growth-rate-limsup} to \(\R^{2} = \C\), \(\Omega = \C\) so that
  \(\holom \subseteq C(\R^{2}, \Omega)\), and note that any
  \(\phi: X \to \holom \setminus \{\textrm{constants}\}\) is necessarily everywhere unbounded by
  Liouville's theorem.
\end{proof}

\section{Topologies on meromorphic functions}
\label{topol-divis-merom}

As discussed in Section~\ref{sec:merom-entire-functions}, the space of meromorphic functions
\(\merom\) is typically endowed with the topology \(\tau_{\merom}\) of uniform convergence on
compact subsets of \(\C\) with respect to, say, the spherical metric on the Riemann sphere
\(\Cbar\). However, the algebraic operations on \(\merom\) are not continuous with respect to
\(\tau_{\merom}\).  This appendix aims to demonstrate that there is no Polish group topology on the
multiplicative group \(\meromnz\) of non-zero meromorphic functions.  This result is established in
Corollary~\ref{cor:no-polish-topology}, which proves a slightly more general statement.

\subsection{Dudley's theorem}
\label{sec:dudley-thm}

Our argument relies on a theorem by Dudley~\cite{dudleyContinuityHomomorphisms1961} concerning the
automatic continuity of homomorphisms into groups equipped with norms having linear growth.  We
begin by recalling the relevant definitions.  (Our terminology differs from that used
in~\cite{dudleyContinuityHomomorphisms1961}.) In contrast to
Section~\ref{sec:equivariant-borel-liftings}, all seminorms in this appendix take values in \(\N\).

\begin{definition}
  \label{def:seminorm}
  A \emph{seminorm} on a group \(H\) is a map \(\norm{\cdot} : H \to \N\) such that for all
  \(h, h_{1}, h_{2} \in H\) one has
  \begin{enumerate}[leftmargin=.6cm, label={\sf \arabic*}., ref={\sf \arabic*}]
  \item \(\norm{e_{H}} = 0\);
  \item \(\norm{h} = \norm{h^{-1}}\);
  \item \(\norm{h_{1} h_{2}} \le \norm{h_{1}} + \norm{h_{2}}\).
  \end{enumerate}
\end{definition}

Given a seminorm \(\norm{\cdot}\), its \emph{kernel} is
\(\ker(\norm{\cdot}) = \{h \in H : \norm{h} = 0\}\).  Note that \(\ker(\norm{\cdot})\) is
necessarily a subgroup of \(H\).  A seminorm is a \emph{norm} if its kernel is trivial.  We say that
\(\norm{\cdot}\) has \emph{linear growth} if the inequality \(\norm{h^{n}} \ge \max\{n, \norm{h}\}\)
holds for all \(n \ge 1\) and all \(h \in H \setminus \ker(\norm{\cdot})\).  Since
\(\norm{h^{n}} \le n \norm{h}\) holds for all seminorms, no seminorm can grow faster than linearly
in this sense.  All seminorms considered in this appendix are, in fact, \emph{additive}, meaning
they satisfy the stronger condition \(\norm{h^{n}} = n\norm{h}\) for all \(h \in H\) and
\(n \in \N\), and thus have linear growth.

Given a group \(G\) and a topology \(\tau_{G}\) on it, we say that \(\tau_{G}\) is a \emph{semigroup
  topology} if the multiplication is \(\tau_{G}\)-continuous (with no assumption on the continuity
of the inverse map).

The following theorem is stated in~\cite{dudleyContinuityHomomorphisms1961}*{Thm.~1} for group
topologies on \(G\).  However, continuity of the inverse map is not used in the proof, and we will
make use of the more general statement later on.  For the reader's convenience, the proof of
Dudley's theorem is reproduced below.

\begin{theorem}[Dudley~{\cite{dudleyContinuityHomomorphisms1961}*{Thm.~1}}]
  \label{thm:dudley-theorem}
  Suppose a group \(H\) admits a norm with linear growth.  Let \(G\) be a group equipped with a
  completely metrizable semigroup topology.  Any homomorphism from \(G\) to \(H\) is continuous with
  respect to the discrete topology on \(H\).
\end{theorem}

\begin{proof}
  Let \(\phi : G \to H\) be a homomorphism, \(d\) be a complete metric on \(G\) that generates a
  semigroup topology, and \(\norm{\cdot}\) be a norm on \(H\) with linear growth.  We claim that
  \(\ker \phi = \{g \in G : \phi(g) = e_H\}\) contains a neighborhood of the identity.

  Suppose towards a contradiction that this is not the case.  We construct a sequence
  \((g_{n})_{n} \subseteq G \setminus \ker \phi\) such that for
  \(r_{n} = n + \sum_{i=0}^{n}\norm{\phi(g_{i})}\) and \(h_{m,n} \in G\) given by
  \[h_{m,n} = g_{m}(g_{m+1}(\cdots (g_{n-1}(g_{n})^{r_{n-1}})^{r_{n-2}}\cdots
    )^{r_{m+1}})^{r_{m}},\]
  the sequence \((h_{m,n})_{n=m}^{\infty}\) is \(d\)-Cauchy for each \(m\).  Note that
  \begin{equation}
    \label{eq:hnm-relation}
    h_{m,n} = g_{m}(h_{m+1,n})^{r_{m}}.
  \end{equation}

  Take \(g_{0}\) to be any element of \(G \setminus \ker \phi\).  Suppose that \(g_{m}\),
  \(m \le n\), and therefore also \(r_{m}\), \(h_{m,n}\), have been constructed for all \(m \le n\).
  Consider the function \(\xi_{m} : G \to G\) given by
  \[\xi_{m}(x) = g_{m}(g_{m+1}(\cdots (g_{n-1}(g_{n} \cdot x^{r_{n}})^{r_{n-1}})^{r_{n-2}}\cdots
    )^{r_{m+1}})^{r_{m}}.\]
  This map is continuous by the continuity of the multiplication and \(\xi_{m}(e) = h_{m,n}\) for
  all \(m \le n\).  Therefore, for a small enough neighborhood of the identity \(U \subseteq G\) the
  inequality \(d(\xi_{m}(g), h_{m,n}) < 2^{-n}\) holds for all \(m \le n\) and all \(g \in U\).
  Since \(\ker \phi\) does not contain any neighborhoods of the identity by our assumption, we may
  pick for \(g_{n+1}\) an element of \(U \setminus \ker \phi\).  This ensures that
  \(d(h_{m,n+1}, h_{m,n}) < 2^{-n}\) and \((h_{m,n})_{n=m}^{\infty}\) is thus \(d\)-Cauchy.

  Set \(h_{m} = \lim_{n \to \infty} h_{n,m}\) and note that \(h_{m} = g_{m}h_{m+1}^{r_{m}}\) holds
  for all \(m\) by Eq.~\eqref{eq:hnm-relation} and the continuity of the multiplication.  A priori,
  we have little control as to whether \(h_{m}\)'s are elements of \(\ker \phi\) or not.
  Nonetheless, we claim that
  \begin{equation}
    \label{eq:hmrm-length-bound}
    \norm{\phi(h_{m}^{r_{m-1}})} \ge r_{m-1}
  \end{equation}
  holds for all \(m\).  Indeed, if \(\norm{\phi(h_{m+1})} \ne 0\), then
  \(\norm{\phi (h^{r_{m}}_{m+1})} \ge r_{m} \) by the linearity of the length function and
  \begin{displaymath}
    \begin{aligned}
      \norm{\phi(h^{r_{m-1}}_{m})}
      &\ge \norm{\phi (h_{m})} = \norm{\phi(g_{m}h_{m+1}^{r_{m}})} \\
      &\ge\norm{\phi (h_{m+1}^{r_{m}})} - \norm{\phi(g_{m})} \ge r_{m} - \norm{\phi(g_{m})}
        = r_{m-1} +1 > r_{m-1}.
    \end{aligned}
  \end{displaymath}
  If on the other hand \(\norm{\phi(h_{m+1})} = 0\), then \(\phi(h_{m+1}) = e_{H}\) and so
  \begin{displaymath}
    \begin{aligned}
      \norm{\phi(h_{m}^{r_{m-1}})}
      &= \norm{\phi((g_{m}h_{m+1}^{r_{m}})^{r_{m-1}})} = \norm{\phi(g_{m}^{r_{m-1}})} \ge r_{m-1},
    \end{aligned}
  \end{displaymath}
  where the last inequality relies on the choice of \(g_{m} \in G \setminus \ker \phi \).  This
  justifies Eq.~\eqref{eq:hmrm-length-bound}, which, in particular, implies that \(h_{m}\) must be
  an element of \(G \setminus \ker \phi\).

  Iterating Eq.~\eqref{eq:hmrm-length-bound} yields a contradiction.  Indeed
  \begin{displaymath}
    \begin{aligned}
      \norm{\phi(h_{0})}
      &= \norm{\phi(g_{0}h_{1}^{r_{0}})} \ge \norm{\phi(h_{1}^{r_{0}})} - \norm{\phi(g_{0})}
        \ge \norm{\phi(h_{1})} - \norm{\phi(g_{0})} \\
      &= \norm{\phi(g_{1}h_{2}^{r_{1}})} - \norm{\phi(g_{0})}
        \ge \norm{(h_{2}^{r_{1}})} - \norm{\phi(g_{1})} - \norm{\phi(g_{0})} \\
      &\ge  \norm{\phi(h_{2})} - \norm{\phi(g_{1})} - \norm{\phi(g_{0})}\\
      &= \ \cdots \\[-5pt]
      &= \norm{\phi(g_{m}h_{m+1}^{r_{m}})} - \sum_{i=0}^{m-1}\norm{\phi(g_{i})}
        \ge \norm{\phi(h_{m+1}^{r_{m}})} - \sum_{i=0}^{m}\norm{\phi(g_{i})} \\
      \textrm{Eq.~\eqref{eq:hmrm-length-bound}} &\ge r_{m} - \sum_{i=0}^{m}\norm{\phi(g_{i})} = m
    \end{aligned}
  \end{displaymath}
  is true for all \(m\), i.e., \(\norm{\phi(h_{0})} = \infty\), which is impossible.

  We conclude that \(\ker \phi\) contains a neighborhood of the identity.  From this it is easy to
  see that \(\ker \phi\) must be open.  Recall that \(\ker\phi\) is a subgroup of \(G\).  If
  \(U \subseteq \ker\phi\) is open, then we can write \(\ker\phi = \bigcup_{g} gU\), where the union
  is taken over all \(g \in \ker \phi\).  The topology of the metric \(d\) is only a semigroup
  topology and in a general topological semigroup the translation map \(x \mapsto gx\) is continuous
  but not necessarily open.  However, in our setup \(G\) is a group, so \(x \mapsto g^{-1}x\) is a
  continuous inverse to the former map, hence all translations are homeomorphisms and \(gU\) is thus
  open for all \(g \in G\).  This shows that \(\ker\phi\) is open.

  Finally, given any \(A \subseteq H\), \(\phi^{-1}(A) = \bigcup_{g}g\ker\phi\), where the union is
  over \(g \in \phi^{-1}(A)\), which shows that \(\phi^{-1}(A)\) is open.  We conclude that \(\phi\)
  is continuous with respect to the discrete topology on \(H\).
\end{proof}

\begin{corollary}
  \label{cor:dudley-theorem-open-kernel}
  Suppose \(\norm{\cdot}\) is a seminorm of linear growth on a group \(G\), and
  \(\ker(\norm{\cdot})\) is a normal subgroup of \(G\).  Then \(\ker(\norm{\cdot})\) is open in any
  completely metrizable semigroup topology on \(G\).
\end{corollary}

\begin{proof}
  Let \(N = \ker(\norm{\cdot})\) be normal in \(G\) and consider the factor group \(H = G/N\).  Let
  \(\pi : G \to H\) be the quotient map.  It is straightforward to verify that
  \(\norm{g_{1}} = \norm{g_{2}}\) whenever \(g_{1}N = g_{2}N\), allowing us to define a function
  \(\norm{\cdot}_{H} : H \to \N\) via \(\norm{gN}_{H} = \norm{g}\).  Furthermore,
  \(\norm{\cdot}_{H}\) is a norm with linear growth, \(\ker(\norm{\cdot}_{H}) = \{e_{H}\}\).  Let
  \(\tau_{G}\) be a completely metrizable semigroup topology on \(G\).  By
  Theorem~\ref{thm:dudley-theorem}, \(\pi\) is \(\tau_{G}\)-continuous with respect to the discrete
  topology on \(H\).  It follows that \(N = \pi^{-1}(e_{H})\) must be \(\tau_{G}\)-open, as claimed.
\end{proof}

\subsection{Applications to meromorphic functions}
\label{sec:appl-merom-funct}

Here is an example of a linearly growing seminorm relevant to our work.

\begin{exmpl}
  \label{exmpl:divisors-seminorm}
  Let \(\divis\) be the Abelian group of divisors on \(\C\).  For a compact \(K \subseteq \C\),
  define \(\norm{\cdot}_{K} : \divis \to \N\) by \(\norm{d}_{K} = \sum_{z \in K}|d(z)|\).  By the
  definition of a divisor, \(d(z)\) is non-zero for only finitely many \(z \in K\), so the sum is
  well-defined.  One can easily verify that \(\norm{\cdot}_{K}\) is an additive, and thus linearly
  growing, seminorm.  The group \(\ker(\norm{\cdot}_{K})\) consists of divisors \(d\) satisfying
  \(d(z) = 0\) for all \(z \in K\).  Since \(\divis\) is Abelian, all its subgroups are normal.
  Corollary~\ref{cor:dudley-theorem-open-kernel} therefore implies that \(\ker(\norm{\cdot}_{K})\)
  is open in any completely metrizable semigroup topology on \(\divis\).
\end{exmpl}

\begin{exmpl}
  \label{exmpl:meromorphic-seminorm}
  A seminorm on a group \(H\) can be pulled back to a group \(G\) through any homomorphism
  \(G \to H\).  For instance, consider the multiplicative group \(\meromnz\) of non-zero meromorphic
  functions and the homomorphism \(\dm : \meromnz \to \divis\) given by the divisor map.  Given a
  compact set \(K \subseteq \C\), define a seminorm \(\norm{\cdot}'_K\) on \(\meromnz\) by
  \(\norm{f}'_{K} = \norm{\dm(f)}_{K}\) for \(f \in \meromnz\), where \(\norm{\cdot}_{K}\) is the
  seminorm from Example~\ref{exmpl:divisors-seminorm}. Note that \(\ker(\norm{\cdot}'_K)\) consists
  of meromorphic functions with no poles or zeros in \(K\).
  Corollary~\ref{cor:dudley-theorem-open-kernel} shows that \(\ker(\norm{\cdot}'_K)\) is open with
  respect to any completely metrizable semigroup topology on \(\meromnz\).
\end{exmpl}

Given a set \(K\), let \(\freeab(K)\) denote the free Abelian group with generators \(K\).  If
\(K \subseteq K'\), there is a natural surjective homomorphism \(\pi : \freeab(K') \to \freeab(K)\)
defined on the generators by
\[\pi(x) =
  \begin{cases}
    x & \textrm{if } x \in K,\\
    0 & \textrm{otherwise}.
  \end{cases}
\]
The family of compact subsets of \(\C\) is directed under inclusion, allowing us to form the inverse
limit \(\varprojlim \freeab(K)\) of the groups \(\freeab(K)\) as \(K\) ranges over compact subsets
of \(\C\).  Observe that \(\varprojlim \freeab(K)\) is naturally isomorphic to the group of divisors
\(\divis\).  We endow each \(\freeab(K)\) with the discrete topology and let \(\tau_{\divis}\) be
the inverse limit topology on \(\divis\).  For \(K \subseteq \C\), let
\(\pi_{K} : \divis \to \freeab(K)\) be the corresponding projection.

\begin{proposition}
  \label{prop:divis-inductive-topology-automatic-continuity}
  Let \(G\) be a group, \(\tau_{G}\) a completely metrizable semigroup topology on \(G\), and
  \(\alpha : G \to \divis\) a homomorphism.
  \begin{enumerate}[leftmargin=.6cm, label={\sf \arabic*}., ref={\sf \arabic*}]
  \item\label{item:alpha-continuous} The homomorphism \(\alpha\) is
    \((\tau_{G}, \tau_{\divis})\)-continuous.
  \item\label{item:not-separable} If the range of \(\pi_{L} \circ \alpha : G \to \freeab(L)\) is
    uncountable for some compact \(L \subseteq \C\), then \(\tau_{G}\) is not separable.
  \end{enumerate}
\end{proposition}

\begin{proof}
  \eqref{item:alpha-continuous}~By the universal property of the inverse limit topology, it suffices
  to show that for each compact \(K \subseteq \C\), the homomorphism
  \[\alpha_{K} = \pi_{K} \circ \alpha : G \to \freeab(K)\] is continuous with respect to the
  discrete topology on \(\freeab(K)\). Define a seminorm \(\norm{\cdot}_{K} : G \to \N\) by
  \(\norm{g}_{K} = \sum_{z \in K} |\alpha(g)(z)|\).  This seminorm is additive, and its kernel
  \(\ker(\norm{\cdot}_{K})\) is normal in \(G\).  By Corollary~\ref{cor:dudley-theorem-open-kernel},
  \(\ker(\norm{\cdot}_{K})\) is \(\tau_{G}\)-open, and thus \(\alpha_{K}\) is continuous.

  \eqref{item:not-separable}~Let \(L\) be such that the range of \(\pi_{L} \circ \alpha\) is
  uncountable.  By item~\eqref{item:alpha-continuous}, \(\pi_{L} \circ \alpha : G \to \freeab(L)\)
  is continuous, and the sets \((\pi_{L} \circ \alpha)^{-1}(z)\) for
  \(z \in \ran (\pi_{L} \circ \alpha) \) are \(\tau_{G}\)-open and pairwise disjoint.  Since
  \(\ran (\pi_{L} \circ \alpha)\) is uncountable, \(\tau_{G}\) cannot be separable and hence is not
  Polish.
\end{proof}

Applying this proposition to the identity map on \(\divis\) and the homomorphism
\(\dm : \meromnz \to \divis\), we obtain the following corollaries.

\begin{corollary}
  \label{cor:divis-inductive-topology-coarsest}
  Any completely metrizable semigroup topology on \(\divis\) is finer than \(\tau_{\divis}\).  The
  homomorphism \(\dm : \meromnz \to \divis \) is \(\tau_{\divis}\)-continuous with respect to any
  completely metrizable semigroup topology \(\tau\) on \(\meromnz\).
\end{corollary}

\begin{corollary}
  \label{cor:no-polish-topology}
  There is no Polish semigroup topology on \(\divis\). There is no Polish semigroup topology on
  \(\meromnz\).
\end{corollary}

\begin{proof}
  Since \(\dm : \meromnz \to \divis\) is surjective,
  Proposition~\ref{prop:divis-inductive-topology-automatic-continuity} implies that no semigroup
  topology on \(\divis\) or \(\meromnz\) is separable.
\end{proof}

Recall that a topological algebra is an algebra \(A\) equipped with a topology \(\tau\) that turns
\(A\) into a topological vector space and with respect to which the multiplication on \(A\) is
continuous.  A complete topological algebra is a topological algebra that is complete as a
topological vector space.

Corollary~\ref{cor:no-polish-topology} rules out the existence of Polish topologies on \(\merom\)
that turn it into a topological algebra.  However, it is now easy to show that \(\merom\) is not a
complete topological algebra under any (not necessarily separable) completely metrizable topology.
Recall that a subset \(Z\) of a vector space \(X\) is \emph{absorbing} if for any \(x \in X\) there
is \(r > 0\) such that for all scalars \(|c| \le r\) one has \(cx \in Z\).  In a topological vector
space, all neighborhoods of the origin are absorbing.

\begin{theorem}
  \label{thm:no-completely-metrizable-topological-algebra}
  There is no metrizable topology on \(\merom\) that turns it into a complete topological algebra.
\end{theorem}

\begin{proof}
  Suppose, towards a contradiction, that \(\tau\) is a completely metrizable topology with respect
  to which \(\merom\) is a topological algebra.  The space \(\meromnz = \merom \setminus \{0\} \) is
  an open subset of \(\merom\) and is therefore completely metrizable in the induced topology.  Let
  \(\norm{\cdot}' = \norm{\cdot}'_{\{0\}}\) be the seminorm on \(\meromnz\) from
  Example~\ref{exmpl:meromorphic-seminorm} corresponding to \(K = \{0\}\).  Since \(\tau\) is an
  algebra topology, multiplication is continuous, and Corollary~\ref{cor:dudley-theorem-open-kernel}
  implies that the set \(\ker(\norm{\cdot}')\) is \(\tau\)-open in \(\meromnz\) and hence in
  \(\merom\).

  By construction, \(\ker(\norm{\cdot}')\) is a subset of \(\meromnz\), so
  \(0 \not \in \ker(\norm{\cdot}')\).  However, the constant function \(1 \in \ker(\norm{\cdot}')\),
  and thus \(0 \in \ker(\norm{\cdot}') - 1\).  Since \(\tau\) is a vector space topology,
  translations are \(\tau\)-open, and \(\ker(\norm{\cdot}') - 1\) is therefore an open neighborhood
  of the origin, hence absorbing.  However, if \(h(z) = 1/z\) and \(c \ne 0\), then \(ch + 1\) has a
  pole at the origin, so \(ch \not \in \ker(\norm{\cdot}') - 1\). This is a contradiction.
\end{proof}

Grosse-Erdmann~\cite{grosse-erdmannLocallyConvexTopology1995} noted that, as a consequence of
Arens's extension~\cite{arensLinearTopologicalDivision1947} of the Gelfand--Mazur theorem, there is
no Hausdorff locally convex topology on \(\merom\) that is a field topology.  It was also asked
in~\cite{grosse-erdmannLocallyConvexTopology1995} whether such a topology exists if we drop either
the local convexity or the continuity of the inverse requirement, and, furthermore, whether such a
topology can be metrizable.  Theorem~\ref{thm:no-completely-metrizable-topological-algebra} answers
the latter part of this question in the negative.

\begin{corollary}{\rm (cf.~\cite{grosse-erdmannLocallyConvexTopology1995}*{p.~302})\phantom{}}
  \begin{enumerate}[leftmargin=.6cm, label={\sf \arabic*}., ref={\sf \arabic*}]
  \item There is no completely metrizable vector space and field topology on \(\merom\).
  \item There is no completely metrizable locally convex algebra topology on \(\merom\).
  \end{enumerate}
\end{corollary}

Finally, we consider vector space topologies on \(\merom\).  Of course, it would be too much to
expect \(\merom\) not to have any Polish vector space topologies.  Being a vector space with a Hamel
basis of size continuum, \(\merom\) is (abstractly) isomorphic to an infinite-dimensional separable
Banach space and thus admits many distinct Polish topologies.  None of these, however, generate the
Borel \(\sigma\)-algebra of \(\tau_{\merom}\).  We recall that all the field operations on
\(\merom\) are Borel.

\begin{theorem}
  \label{thm:no-vector-space-topology-same-borel}
  There is no Polish vector space topology on \(\merom\) whose Borel \(\sigma\)-algebra equals
  \(\borel{\tau_{\merom}}\).
\end{theorem}

\begin{proof}
  Suppose, towards a contradiction, that such a topology \(\tau\) exists.  For \(n \in \N\), define
  \begin{displaymath}
    \begin{aligned}
      M_{n}&=\{f \in \merom : 0 \textrm{ is not a pole of } z^{n}f\} \\
           &= \{0\} \cup \{f \in \meromnz :
             \dm(f)(0) \ge -n\}.
    \end{aligned}
  \end{displaymath}
  Note that \(M_{n}\) is Borel and thus has the Baire property with respect to \(\tau\).  Since
  \(\merom = \bigcup_{n} M_{n}\), the Baire category theorem ensures that \(M_{n_{0}}\) is
  non-meager for some \(n_{0}\).  However, each \(M_{n}\) is a vector subspace of \(\merom\), which
  forces \(\merom = M_{n_{0}}\). This is absurd, as \(1/z^{n_{0}+1} \not \in M_{n_{0}}\).
\end{proof}

\section{Runge's theorem for periodic harmonic functions}
\label{sec:rung-theor-peri}

Let \(\Gamma\) be a discrete subgroup of \(\Rd\), \(d\ge 3\), of dimension at most \(d-2\). We are
given a closed set \(K\subset \Rd\) which is invariant with respect to the action of
\(\Gamma\), such that \(\Rd\setminus K\) is connected.  We assume throughout that
\(K\cap D_0\) is compact, where \(D_0=D_0(\Gamma)\) is a (closed) fundamental domain of \(\Gamma\),
and fix a \(\Gamma\)-invariant neighborhood \(U\) of \(K\).

\begin{theorem}
  \label{thm:periodic-Runge}
  Given a harmonic function \(h\) on \(U\) with \(\stab(h)=\Gamma\) and a number \(\epsilon>0\),
  there exists a harmonic function \(H\) on \(\Rd\) with
  \(\mathrm{stab}(H)\supseteq\Gamma\), such that \(\sup_{x\in K}|H(x)-h(x)|\le \epsilon\).
\end{theorem}

\begin{remark}
  The above result is true whenever \(\Rd/\Gamma\) is a non-compact manifold, which also
  holds in the case \(\mathrm{dim}(\Gamma)=d-1\). However, the following simple proof does not apply
  (the periodization of the Green function in Lemma~\ref{lem:periodic-fundamental-sol}
  diverges). The more general theorem is a special case of the Lax--Malgrange approximation theorem,
  see Section~3.10 in~\cite{narasimhanAnalysisRealComplex1985}.
\end{remark}

The first step towards Theorem~\ref{thm:periodic-Runge} is the following lemma.
\begin{lemma}
  \label{lem:periodic-fundamental-sol}
  There exists a function \(G_\Gamma\) on \(\Rd\) with \(\mathrm{stab}(G_\Gamma)=\Gamma\),
  such that
  \[
    -\Delta G_\Gamma=\sum_{x\in\Gamma}\delta_x.
  \]
\end{lemma}

\begin{proof}
  We obtain a fundamental solution \(G(x)\) by \(G(x)=\frac{1}{(d-2)\omega_d}f(x)\) where
  \(\omega_d\) is the volume of the unit ball in \(\Rd\) and where \(f\) is the function
  \[
    f(x)=\frac{1}{|x|^{d-2}}+ \sum_{y\in
      \Gamma\setminus\{0\}}\left(|x-y|^{-(d-2)}-|y|^{-(d-2)}\right).
  \]
  Indeed, in an annular shell of bounded width and inner radius \(r\), there are about \(r^{d-3}\)
  lattice points. Moreover, for fixed \(x\) and \(r\) large, the terms are of order \(r^{-(d-1)}\)
  and hence the series converges absolutely. Thus, \(f(x)\) defines a superharmonic function on
  \(\Rd\) with Newtonian singularities at the lattice points, with
  \[
    -\Delta f=(d-2)\omega_d\sum_{y\in \Gamma}\delta_{y}.
  \]
  It is not immediately clear, however, that \(f\) and \(G\) are \(\Gamma\)-periodic.  To
  investigate this, let \(a\in \Gamma\) and consider the difference \(f(x+a)-f(x)\). A direct
  computation yields, for \(x\notin \Gamma\),
  \begin{align*}
    f(x+a)-f(x)&=\frac{1}{|x+a|^{d-2}}-\frac{1}{|x|^{d-2}} \\
               &\qquad +\sum_{y\in \Gamma\setminus\{0\}}
                 \left(\frac{1}{|x-(y-a)|^{d-2}}-\frac{1}{|x-y|^{d-2}}\right)\\
               &=\sum_{y\in\Gamma}\left(\frac{1}{|x-(y-a)|^{d-2}}-\frac{1}{|x-y|^{d-2}}\right).
  \end{align*}
  We claim that the right-hand side vanishes identically. Indeed, it holds that
  \[
    f(x+a)-f(x)=\sum_{y\in\Gamma\cap B(0,R)}
    \left(\frac{1}{|x-(y-a)|^{d-2}}-\frac{1}{|x-y|^{d-2}}\right)+\mathrm{o}(1)
  \]
  as \(R\to+\infty\), and in the sum on the right-hand side all but \(\mathrm{O}(R^{(d-3)})\) terms
  cancel exactly.  As the remaining terms are of order \(R^{-(d-1)}\), the right-hand side tends to
  \(0\) as \(R\to+\infty\). This shows that \(f(x+a)=f(x)\), so \(G_\Gamma(x) = G(x)\) is
  \(\Gamma\)-invariant.

  It only remains to show that \(G_\Gamma\) is not invariant under any larger subgroup
  \(G'\subset\Rd\). This, however, is immediate from the fact that
  \(\mathrm{stab}(\Delta G)=\Gamma\). This completes the proof.
\end{proof}

With the periodic fundamental solution \(G_\Gamma\) at hand, the necessary periodic approximation
property is standard. First, we establish a representation of smooth functions \(g\) with
\(\mathrm{stab}(g)=\Gamma\) and suitable conditions on the support.  We denote by
\(D_0=D_0(\Gamma)\) the closed fundamental domain of \(\Gamma\).

\begin{lemma}
  \label{lem:representation}
  Assume that \(g\) is smooth and \(\Gamma\)-invariant and that \(\mathrm{supp}(g)\cap D_0\) is
  compact. Then we have
  \begin{equation}
    \label{eq:green}
    g(x)=-\int_{D_{0}}G_\Gamma(x-y)\Delta g(y) dm(y).
  \end{equation}
\end{lemma}

\begin{proof}
  If we denote the right-hand side by \(g_0(x)\), then \(g_0\) is a well-defined function with
  \(\mathrm{stab}(g_0)=\Gamma\), and \(g_0(x)\to 0\) as \(x\to\infty\), \(x\in D_{0}\).  It follows
  that \((g-g_0)(x)\) is a \(\Gamma\)-periodic harmonic function which tends to zero as
  \(x\to\infty\) along \(D_0\). Any such function is bounded, and hence constant by Liouville's
  theorem. Due to the limiting behavior at infinity, we must have \(g=g_0\).
\end{proof}

We turn to the proof of the periodic Runge theorem stated at the beginning of this appendix.

\begin{proof}[Proof of \(\Gamma\)-periodic Runge (Theorem~\ref{thm:periodic-Runge})]
  We follow one of the standard textbook proofs of Runge's theorem, which is based on the
  Hahn-Banach theorem, and can be found for instance in the
  paper~\cite{bagbyUniformHarmonicApproximation1994}.

  Denote by \(C_\Gamma(\Rd)\) the space of continuous functions with stabilizer \(\Gamma\),
  and endow it with the topology of locally uniform convergence in \(D_0\) (that is, the topology
  corresponding to locally uniform convergence on the quotient manifold \(\Rd/\Gamma\)
  where the \(\Gamma\)-invariant functions naturally live).  We use the notation
  \(\mathcal{H}_\Gamma(K)\) for the subspace of \(C_\Gamma(\Rd)\) consisting of functions
  harmonic on some neighborhood of the closed set \(K\). If \(\mathcal{O}\) is open and
  \(\Gamma\)-invariant, we write \(\mathcal{H}_\Gamma(\mathcal{O})\) for those functions in
  \(C_\Gamma(\Rd)\) harmonic on \(\mathcal{O}\).

  It is sufficient to show that for an arbitrarily large open \(\Gamma\)-invariant set
  \(\mathcal{O}\), \(\mathcal{H}_\Gamma(\mathcal{O})\) is dense in \(\mathcal{H}_\Gamma(K)\).  This,
  in turn, will follow if we show that whenever \(\mu\) is a finite signed Borel measure on
  \(K\cap D_0\) whose associated linear functional
  \[
    g\mapsto \int_{D_0}g\, d\mu
  \]
  annihilates \(\mathcal{H}_\Gamma(\mathcal{O})\), then \(\mathcal{H}_\Gamma(K)\) is annihilated as
  well.

  To that end, let \(\mu\) be such a measure and let \(f\in\mathcal{H}_\Gamma(K)\).  Form the
  function
  \begin{equation}
    \label{eq:nu-def}
    \nu(x)=\int_{D_0} G_\Gamma(x-y)\,d\mu(y).
  \end{equation}
  The function \(\nu\) is harmonic outside \(K\) (since for each \(y\), \(x\mapsto G_\Gamma(x-y)\)
  is harmonic outside \(\Gamma\) and since \(\mu\) is supported on \(K\cap D_0\)). Moreover, we
  claim that \(\nu\) vanishes outside \(\mathcal{O}\), and since \(\Rd\setminus K\) is
  connected the unique continuation property of harmonic functions implies that \(\nu\) is in fact
  supported on \(K\).  To see why \(\nu\) vanishes outside \(\mathcal{O}\), note that
  \(x\in\mathcal{O}^c\), \(x-y\in\Gamma\) implies that \(y\in \mathcal{O}^c\) as well, so
  \(y\mapsto G_\Gamma(x-y)\) is an element of \(\mathcal{H}_\Gamma(\mathcal{O})\).  But \(\mu\)
  annihilates \(\mathcal{H}_\Gamma(\mathcal{O})\), so we conclude from the definition
  \eqref{eq:nu-def} that \(\nu=0\) on \(\mathcal{O}^c\).

  Now, for the trick. We let \(N\) be a neighborhood of \(K\) such that \(f\) is smooth and harmonic
  on \(N\), and let \(\varphi\) be a smooth \(\Gamma\)-invariant cut-off function supported in \(N\)
  which equals \(1\) on a neighborhood of \(K\). We have that
  \[
    \int f(x)d\mu(x)=\int f(x)\varphi(x)d\mu(x) =-\int\int G_\Gamma(x-y)\Delta (f\varphi)(y)
    dm(y)d\mu(x)
  \]
  by Lemma~\ref{lem:representation} applied to \(f\varphi\).  Applying Fubini's theorem and Green's
  formula, we get
  \[
    \int f(x)d\mu(x)=-\int \Delta (f\varphi) (y)\nu(y) dm(y).
  \]
  But \(\Delta (f\varphi)\equiv 0\) on a neighborhood of \(K\), while \(\nu=0\) outside
  \(K\). Hence, \(\int f(x)d\mu(x)\) vanishes, and the proof that
  \(\mathcal{H}_\Gamma(\mathcal{O})\) is dense in \(\mathcal{H}_\Gamma(K)\) is complete.

  \medskip

  A standard iterative application of the above version of Runge with respect to an exhaustive
  family \(K_j\) such that \(K_j\cap D_0\) is compact completes the proof (cf.\ the proof
  in~\cite{bagbyUniformHarmonicApproximation1994}).
\end{proof}

\section{Borel toasts for the heat equation}
\label{sec:borel-toasts-heat}

In this appendix, we construct Borel toasts for free \(\R^{d+1}\)-actions with regions that
are Runge domains for the heat equation. While Theorem~\ref{thm:existence-of-slice-filled-toasts}
presents a new result, the exposition is designed for those unfamiliar with toast
constructions---demonstrating, in particular, the construction of toasts for Euclidean space
actions.

\subsection{Topological hulls}
\label{sec:topological-hull}

Let \(K\) be a compact subset of \(\Rd\), \(d\ge 2\).
The \emph{topological hull} of \(K\), denoted
\(\hull(K)\), is defined as
\[
\hull(K) = K \cup U(K),
\]
where \(U(K)\) consists of all bounded connected components of \(\Rd \setminus K\). Intuitively,
\(\hull(K)\) is obtained by filling in all bounded holes of \(K\).

The hull operation satisfies the following properties:
\begin{itemize}
  \item \emph{Monotonicity}: If \(K \subseteq L\), then \(\hull(K) \subseteq \hull(L)\).
  \item \emph{Idempotence}: \(\hull(\hull(K)) = \hull(K)\).
\end{itemize}

The next lemma shows that the hull operation is additive for disjoint compact sets.
Although it looks to us like a standard result, we did not find it in the literature.
However, a related lemma was proven in \cite{BergerEtAlCMP}*{Prop.~C.1}.

\begin{lemma}
  \label{lem:topological-hull-union-two}
  Let \(K_{1}, K_{2} \in \vietoris{\Rd}\) be disjoint compact sets. Then
  \[
  \hull(K_{1} \cup K_{2}) = \hull(K_{1}) \cup \hull(K_{2}).
  \]
\end{lemma}

In the proof we will use some basic notions from algebraic topology, in
particular reduced singular homology (see, e.g.,
\cite{VassilievIntroTopology}*{Ch.~9--10}). For a compact set \(K\subset\Rd\) we put \(U=\Rd\setminus K\).
The reduced singular homology group \(\tilde{H}_{0}(U)\) may be thought of as a free Abelian group with one
generator for each of the (possibly countably many) bounded connected components of \(U\), i.e.,
the unbounded component is ``killed'' by the reduction.

\begin{proof}
  The inclusions \(\hull(K_{i}) \subseteq \hull(K_{1} \cup K_{2})\) for \(i = 1,2\) follow directly
  from the monotonicity of the hull operation. We now prove the reverse inclusion:
  \[ \hull(K_{1} \cup K_{2}) \subseteq \hull(K_{1}) \cup \hull(K_{2}). \]

  Let \(U_{i} = \Rd \setminus K_{i}\) for \(i = 1,2\). Since \(K_{1}\) and \(K_{2}\) are
  disjoint, \(\{U_{1}, U_{2}\}\) forms an open cover of \(\Rd\). Applying the
  Mayer--Vietoris sequence to the reduced singular homology groups of this cover yields:
  \[
    \tilde{H}_{1}(\Rd) \to \tilde{H}_{0}(U_{1} \cap U_{2}) \xrightarrow{(\iota_{1},
      \iota_{2})} \tilde{H}_{0}(U_{1}) \oplus \tilde{H}_{0}(U_{2}) \to \tilde{H}_{0}(\Rd),
  \]
  where \(\iota_{i}\), \(i=1,2\), are induced by the
  inclusions \(U_{1} \cap U_{2} \hookrightarrow U_{i}\).
  Since \(\tilde{H}_{k}(\Rd) = 0\) for all \(k \ge 0\), this sequence reduces to the short
  exact sequence:
  \[
    0 \to \tilde{H}_{0}(U_{1} \cap U_{2}) \xrightarrow{(\iota_{1}, \iota_{2})} \tilde{H}_{0}(U_{1})
    \oplus \tilde{H}_{0}(U_{2}) \to 0,
  \]
  showing that \((\iota_{1},\iota_{2})\) is an isomorphism.

  For each \(i = 1,2\), select a point \(x^{i}_{k}\) in every bounded connected component of
  \(U_{i}\), and a point \(y_{l}\) in every bounded connected component of \(U_{1} \cap U_{2}\). The
  classes \([x^{i}_{k}]\) and \([y_{l}]\) form bases for \(\tilde{H}_{0}(U_{i})\) and
  \(\tilde{H}_{0}(U_{1}\cap U_{2})\), respectively.

  Let \(V_{l}\) denote the bounded component of \(U_{1} \cap U_{2}\) containing \(y_{l}\). There
  exist unique connected components \(W_{1} \subseteq U_{1}\) and \(W_{2} \subseteq U_{2}\) such
  that \(V_{l} = W_{1} \cap W_{2}\). Define:
  \[
    c^{i} =
    \begin{cases}
      [x^{i}_{k}] & \text{if } W_{i} \text{ is bounded and } x^{i}_{k} \in W_{i},\\[4pt]
      0 & \text{if } W_{i} \text{ is unbounded}.
    \end{cases}
  \]
  Under the isomorphism \((\iota_{1},\iota_{2})\), the class \([y_{l}]\) is mapped to
  \((c^{1}, c^{2})\). Injectivity implies at least one \(c^{i}\) is non-zero, meaning at least one
  of \(W_{1}, W_{2}\) is bounded.

  Thus, every bounded connected component of \(U_{1} \cap U_{2}\) is contained in a bounded
  component of \(U_{1}\) or \(U_{2}\). Translating to hulls, we obtain the desired inclusion
  \(\hull(K_{1} \cup K_{2}) \subseteq \hull(K_{1}) \cup \hull(K_{2})\).
\end{proof}

\begin{proposition}
  \label{prop:topological-hull-finite-union}
  Let \(K_{1}, \ldots, K_{n} \in \vietoris{\Rd}\) be pairwise disjoint compact sets. Then the
  topological hull of their union equals the union of their topological hulls:
  \[ \hull\bigl(\bigcup_{i=1}^{n}K_{i}\bigr) = \bigcup_{i=1}^{n}\hull(K_{i}). \]
\end{proposition}

\begin{proof}
  The result follows by induction from Lemma~\ref{lem:topological-hull-union-two}.
\end{proof}

For our construction of Borel toasts for the heat equation, we introduce a variant of the
topological hull where the filling occurs slice-wise along hyperplanes orthogonal to the time axis.

Let \(\R^{d+1}\) be parametrized by coordinates \((x,t)\), where \(x \in \Rd\) and
\(t \in \R\). For each \(t \in \R\), we define the horizontal hyperplane at height \(t\) as
\[ P_{t} = \{(x,t): x \in \Rd\} \cong \Rd. \]
The \emph{topological slice hull} of \(K \in \vietoris{\R^{d+1}} \) is defined as the union of the
topological hulls of all horizontal slices:
\[ \shull(K) = \bigcup_{t \in \R} \hull_{t}(K),\]
where \(\hull_{t}(K)\) denotes the topological hull of \(K\cap P_{t}\) computed within \(P_{t}\). We
say that \(K\) is \emph{slice-wise filled} if \(K = \shull(K)\).

\begin{proposition}
  \label{prop:slice-hull-compact}
  For any \(K \in \vietoris{K}\), the topological slice hull \(\shull(K)\) is compact.
\end{proposition}

\begin{proof}
  For \(M > 0\), let \(B_{M} \) denote the closed \(\ell^{\infty}\) ball of radius \(M\) centered at
  the origin.  Note that \(\shull(B_{M}) = B_{M}\). Pick \(M\) large enough to ensure that
  \(K \subseteq B_{M}\), and therefore \(\shull(K) \subseteq B_{M}\).  In particular, the set
  \(\shull(K)\) is bounded, and it remains to show that it is closed.

  Suppose, towards a contradiction, that there exists a point
  \[ (x_{0},t_{0})\in\overline{\shull(K)}\setminus\shull(K). \]
  Let \(z_{\infty}\in \Rd \setminus B_{M+1}\). Then, for every \(t \in \R\),
  \((z_{\infty},t)\notin\hull_{t}(K)\). Furthermore, perturbation by any \(y \in \Rd\)
  with \(\norm{y}<1\) preserves this property:
  \begin{equation}
    \label{eq:z-stably-outside-hull}
    (z_{\infty}+y,t)\notin\hull_{t}(K) \quad \text{for all } t.
  \end{equation}

  Since \((x_{0}, t_{0}) \notin \shull(K)\), the point lies in the unbounded component of the
  relative complement of \(\hull_{t_{0}}(K)\) within the hyperplane \(P_{t_{0}}\).  Since
  connected open subsets of \(\Rd\) are path-connected, there exists a continuous path
  \[
    \tau : [0,1] \to \Rd, \quad \tau(0) = x_{0}, \quad \tau(1) = z_{\infty},
  \]
  such that \((\tau(r), t_{0}) \notin \hull_{t_{0}}(K)\) for all \(r \in [0,1]\). Lifting \(\tau\)
  to \(\R^{d+1}\) via \(\tau_{t}(r) = (\tau(r), t)\) yields a path \(\tau_{t_{0}}\) from
  \((x_{0}, t_{0})\) to \((z_{\infty}, t_{0})\) avoiding \(K\).

  For \(\epsilon > 0\), define the \(\epsilon\)-neighborhood of \(\tau_t\) as
  \[
    V_{\epsilon} = \{\tau_{t}(r) + (y,0) : r \in [0,1], y \in \Rd, \|y\| <
    \epsilon, |t - t_{0}| < \epsilon\}.
  \]
  Each \(V_{\epsilon}\) is open and contains \((x_{0}, t_{0})\). Since
  \((x_{0}, t_{0})\) lies in \(\overline{\shull(K)}\), there exists a sequence
  \[
    (x_{n}, t_{n}) \in \shull(K) \cap V_{1/n}, \quad n \geq 1,
  \]
  where each \((x_{n}, t_{n})\) can be expressed as
  \[
    (\tau(r_n) + y_{n}, t_{n}) \in \hull_{t_{n}}(K), \quad r_n \in [0,1], \quad
    \|y_{n}\| < 1/n.
  \]

  The translated curve \( [0,1] \ni r \mapsto (\tau(r) + y_{n}, t_{n}) \in P_{t_{n}}\)
  connects the points \((x_{0} + y_{n}, t_{n})\) and
  \((z_{\infty} + y_{n}, t_{n})\) while intersecting \(\hull_{t_{n}}(K)\). Since
  \(\norm{y_{n}} < 1/n \le 1\), the point \((z_{\infty} + y_{n}, t_{n})\) lies in
  the unbounded component of \(P_{t_{n}} \setminus \hull_{t_{n}}(K)\) by
  Eq.~\eqref{eq:z-stably-outside-hull}. Thus, there exists \(r'_{n} \ge r_{n}\) such that
  \[(\tau(r'_{n}) + y_{n}, t_{n}) \in K \cap P_{t_{n}}.\]
  Defining \(u_{n} = \tau(r'_{n}) + y_{n}\), we have \((u_{n}, t_{n}) \in K\), and
  \begin{equation}
    \label{eq:xn-close-to-tau}
    d(u_{n}, \tau) = \inf_{r \in [0,1]} \norm{u_{n} - \tau(r)} \le \norm{y_{n}} < 1/n.
  \end{equation}

  By the compactness of \(K\), we may assume (passing to a subsequence if necessary) that
  \((u_{n}, t_{n}) \to (u_{\infty}, t_{0}) \in K\). From \eqref{eq:xn-close-to-tau}, the
  limit satisfies \(d(u_{\infty}, \tau) = 0\), so \(u_{\infty} = \tau(r_{\infty})\) for
  some \(r_{\infty} \in [0,1]\). This implies that the path \(\tau_{t_{0}}\), connecting
  \((x_{0}, t_{0})\) to \((z_{\infty}, t_{0})\), intersects \(K\) at
  \((u_{\infty}, t_{0})\), contradicting the initial choice of \(\tau\). Therefore,
  \(\shull(K)\) is closed and, being bounded, compact.
\end{proof}

The slice hull operation shares the additivity property with the standard convex hull operation.

\begin{proposition}
  \label{prop:slice-hull-additivity}
  For any finite collection of pairwise disjoint sets
  \(K_{1}, \ldots, K_{n} \in \vietoris{\R^{d+1}}\), we have
  \[\shull\Bigl(\bigcup_{i=1}^{n}K_{i}\Bigr) = \bigcup_{i=1}^{n}\shull(K_{i}).\]
\end{proposition}

\begin{proof}
  The result follows by applying Proposition~\ref{prop:topological-hull-finite-union} to each slice
  \(P_{t}\) and the collection of sets \(K_{i} \cap P_{t}\), \(1 \le i \le n\).
\end{proof}

\subsection{Borel toasts with slice-wise filled regions}
\label{sec:borel-toasts-slice-filled}

Our construction of Borel toasts begins similarly to the approach in Appendix~A
of~\cite{marksBorelCircleSquaring2017}. Given a free \(\R^{d+1}\)-flow on a standard Borel space
\(X\) and a sufficiently rapidly growing sequence of positive reals \((r_n)_n\), one shows that
there exists a sequence of cross-sections \((\mathcal{C}_n)_n\) with the following properties:
\begin{enumerate}[leftmargin=.6cm, label={\sf \arabic*}., ref={\sf \arabic*}]
\item For any two distinct elements \(c_n, c_n' \in \mathcal{C}_n\) one has\footnote{The constants
    \(100\) and \(1000\) are by no means special; they simply ensure sufficient separation between
    elements relative to \(r_n\).}
  \[1000r_{n} > |\rho(c_n, c_n')| > 100r_n;\]
\item\label{item:boykin-jackson} For every \(x \in X\) and any \(\epsilon > 0\), there are
  infinitely many \(n\) for which \(d(x, \mathcal{C}_n) < \epsilon r_n\).
\end{enumerate}
The second condition is particularly crucial and represents a strong requirement. For our purposes,
the sequence \((r_{n})_{n}\) needs to grow fast enough to satisfy
\begin{equation}
  \label{eq:rn-growth-rate}
  r_n > n^{2}2^{n+2}r_{n-1} \quad \text{for all } n.
\end{equation}

By Lemma~2.4 of \cite{Slutsky}, we may further assume that \(\rho(c_n, c_m) \in \Q^{d}\) for
all \(c_m \in \mathcal{C}_m\) and \(c_n \in \mathcal{C}_n\). This ensures
condition~\eqref{item:toast-rational} from Definition~\ref{def:toast} is satisfied. Let \(\mfR\)
denote the family of compact sets \(K \in \vietoris{\R^{d+1}}\) that are slice-wise filled,
meaning \(\shull(K) = K\). Below, we explain the construction of a Borel \(\mfR\)-toast over the
cross-sections \(\mathcal{C}_n\).

Let \(B_{r_{n}}\) denote the cube in \(\R^{d+1}\) centered at the origin with side length
\(2r_{n}\).  Suppose that for each \(n\), the maps
\(\lambda_{n} : \mathcal{C}_{n} \to \vietoris{\R^{d+1}}\) satisfy
\(B_{r_{n}} \subseteq \lambda_{n}(c_{n})\) for all \(c_{n} \in \mathcal{C}_{n}\). Then
item~\eqref{item:boykin-jackson} above implies
conditions~\eqref{item:toast-directed}--\eqref{item:toast-exhaustive} of Definition~\ref{def:toast}.

Condition~\eqref{item:toast-layered} can be later ensured by inserting missing intermediate regions
as discussed in Remark~\ref{rem:layered-regions}. Thus, we focus on constructing maps
\(\lambda_{n} : \mathcal{C}_{n} \to \mfR\) satisfying conditions \eqref{item:toast-disjoint} and
\eqref{item:toast-coherent} of Definition~\ref{def:toast}, while maintaining
\(B_{r_{n}} \subseteq \lambda_{n}(c_{n})\) for every \(c_{n} \in \mathcal{C}_{n}\).

We define \(\kappa_{n} : \mathcal{C}_{n} \to \vietoris{\R^{d+1}}\) as the constant map
\(\kappa_{n}(c_{n}) = B_{r_{n}}\). Due to the lacunarity condition
\(|\rho(c_{n}, c'_{n})| > 100 r_{n}\) on the cross-sections, the regions \(\kappa_{n}(c_{n}) c_{n}\)
for \(c_{n} \in \mathcal{C}_{n}\) are pairwise disjoint. However, these regions may not satisfy the
coherence condition across different levels as in
Definition~\ref{def:toast}\eqref{item:toast-coherent}.

To resolve this, we construct the maps \(\lambda'_{n} : \mathcal{C}_{n} \to \vietoris{\R^{d+1}}\)
inductively. Set \(\lambda'_{0} = \kappa_{0}\) as the base case. For each
\(c_{n} \in \mathcal{C}_{n}\), consider all elements
\(c^{1}_{n_{1}}, \ldots, c^{m}_{n_{m}} \in \bigcup_{i<n}\mathcal{C}_{i}\) such that
\[(\kappa_{n}(c_{n})c_{n}) \cap (\lambda'_{n_{i}}(c^{i}_{n_{i}})c^{i}_{n_{i}}) \ne \varnothing.\]
By the lacunarity of the cross-sections, this collection is finite. We then define
\begin{equation}
  \label{eq:fusion}
  \lambda'_{n}(c_{n}) = \kappa_{n}(c_{n}) \cup \bigcup_{i=1}^{m}(\lambda'_{n_{i}}(c^{i}_{n_{i}}) + \rho(c_{n},c^{i}_{n_{i}})).
\end{equation}
In other words, the region \(R'_{n}(c_{n}) = \lambda'_{n}(c_{n})c_{n}\) is formed by augmenting
\(\kappa_{n}(c_{n})c_{n}\) with every prior region \(R'_{n_{i}}(c^{i}_{n_{i}})\) that intersects
it. This construction ensures the coherence of the \(R'_{n}\) regions as in
Definition~\ref{def:toast}\eqref{item:toast-coherent}.

An inductive argument yields the diameter bound for the \(\ell^{\infty}\) metric:
\begin{equation}
  \label{eq:fusion-diameter-bound}
  \diam(\lambda'_{n}(c_{n})) \le 2\sum_{k=0}^{n}2^{n-k}r_{k} = 2r_{n} + 2\sum_{k=0}^{n-1}2^{n-k}r_{k}.
\end{equation}
Since \(|\rho(c_{n},c'_{n})| > 100r_{n}\), the regions \(R'_{n}(c_{n})\) for
\(c_{n} \in \mathcal{C}_{n}\) remain disjoint as long as \(r_{n} \gg \sum_{k< n}2^{n-k}r_{k}\).

After adding the missing regions to satisfy Definition~\ref{def:toast}\eqref{item:toast-layered} as
per Remark~\ref{rem:layered-regions}, the sequence \((\mathcal{C}_{n}, \lambda'_{n})_{n}\) forms a
Borel toast. However, it may not be an \(\mfR\)-toast. For convenience, we say that the maps
\(\lambda'_{n}\) were obtained by \emph{fusing} the maps \(\kappa_{n}\). We will fuse another family
of maps shortly.

So far, the construction provides limited control over the shape of the resulting regions
\(R'_{n}\). While these regions are connected, they need not be filled---let alone slice-wise filled.

To construct an \(\mfR\)-toast from the sequence \((\mathcal{C}_{n}, \lambda'_{n})_{n}\), we define
\(\lambda_{n}\) by fusing (in the sense described above) the maps
\(c_{n} \mapsto \shull(\lambda'_{n}(c_{n}))\). That is, we first apply the topological slice hull
operation to each set \(\lambda'_{n}(c_{n})\). Observe that \(\lambda'_{n}(c_{n})\) lies within the
box \(B_{s_{n}}\), where \(s_{n} = \sum_{k\le n}2^{n-k}r_{k}\)
(cf.~\eqref{eq:fusion-diameter-bound}). By the monotonicity of \(\shull\) and the fact that cubes
are slice-wise filled, it follows that \(\shull(\lambda'_{n}(c_{n})) \subseteq B_{s_{n}}\). Thus,
for each level \(n\), the sets \(\shull(\lambda'_{n}(c_{n}))\), \(c_{n} \in \mathcal{C}_{n}\),
remain pairwise disjoint. However, taking the \(\shull\) operation may disrupt coherence, as
illustrated in Fig.~\ref{fig:slice-fusion-coherence}.

\begin{figure}
  \def\angle{7}
  \centering
  \begin{tikzpicture}[line cap=round,line join=round,isometric view,rotate around z=5*\angle-45]
    \draw[fill=gray!30,canvas is yz plane at x=0.5] (-1.5,-0.5)     rectangle ++(1,1);
    \draw[fill=gray!15,canvas is xy plane at z=-0.5] (-1.5+1,-0.5-1) rectangle ++(1,1);
    \draw[fill=gray!50,canvas is xz plane at y=-0.5] (-0.5,-0.5) rectangle ++(1,1);

    \draw[canvas is xy plane at z=0,] (0,3) -- (0,-1.5);

    \draw[canvas is xy plane at z= 1.5,fill=white] (-1.5,-1.5) rectangle (1.5,0.5);
    \draw[canvas is xz plane at y=-1.5,fill=white,even odd rule] (-1.5,-1.5) rectangle (1.5,1.5)
          (-0.5,-0.5) rectangle (0.5,0.5);
    \draw[canvas is yz plane at x=-1.5,fill=white] (-1.5,-1.5) rectangle (0.5,1.5);

    \draw[canvas is xy plane at z=0.2, fill=white] (-0.2,-2.5) rectangle (0.2,-1.3);
    \draw[canvas is xz plane at y=-2.5, fill=white] (-0.2,-0.2) rectangle (0.2,0.2);
    \draw[canvas is yz plane at x=-0.2, fill=white] (-2.5,-0.2) rectangle (-1.3,0.2);

    \draw[canvas is xy plane at z=0,->] (0,-2.5) -- (0,-3.8)
         node[anchor=east, xshift=-1mm, yshift=-0.5mm] {\(t\)};
       \end{tikzpicture}
  \caption{The slice hull operation yields regions in general position.}
  \label{fig:slice-fusion-coherence}
\end{figure}

Therefore, we fuse these regions once again to ensure coherence. A diameter estimate for the
resulting regions, similar to the one obtained earlier in~\eqref{eq:fusion-diameter-bound}, gives
\begin{displaymath}
  \begin{aligned}
    \diam(\lambda_{n}(c_{n}))
    &\le 2s_{n} + 2\sum_{k=0}^{n-1}2^{n-k}s_{k} \\
    &= 2r_{n} + 2\sum_{j = 0}^{n-1}2^{n-j}r_{j} +
      2\sum_{k=0}^{n-1}2^{n-k}\sum_{j=0}^{k}2^{k-j}r_{j}\\
    &= 2r_{n} + 2\sum_{j = 0}^{n-1}2^{n-j}r_{j} + 2\sum_{k=0}^{n-1}\sum_{j=0}^{k}2^{n-j}r_{j}\\
    &\le 2r_{n} + 2n2^{n}r_{n-1} + 2 n^{2} 2^{n}r_{n-1} \le 2r_{n} + n^{2}2^{n+2}r_{n-1}
  \end{aligned}
\end{displaymath}
This bound, coupled with the assumption on the growth rate for \((r_{n})_{n}\) in
\eqref{eq:rn-growth-rate} and \(|\rho(c_{n},c'_{n})| > 100r_{n}\), guarantees that the regions
\(R_{n}(c_{n}) = \lambda_{n}(c_{n})c_{n}\) remain pairwise disjoint across each level, thus
satisfying Definition~\ref{def:toast}\eqref{item:toast-disjoint}. Consequently, the sequence
\((\mathcal{C}_{n}, \lambda_{n})_{n}\) forms a toast.

To summarize, the construction of a Borel \(\mfR\)-toast for the \(\R^{d+1}\)-flow from the
sequence of cross-sections \((\mathcal{C}_{n})_{n}\) proceeds as follows:
\[ \boxed{\vphantom{fg}\text{cubes}} \to \boxed{\vphantom{fg}\text{fusion}} \to
  \boxed{\vphantom{fg}\text{slice filling}} \to \boxed{\vphantom{fg}\text{fusion}}\, . \]
We claim that the resulting toast \((\mathcal{C}_{n}, \lambda_{n})_{n}\) is indeed an
\(\mfR\)-toast—that is, the final fusion step preserves the slice-wise filling property. This is not
obvious from the geometric perspective.

Consider a region \(\lambda_{n}(c_{n})\), obtained through the construction in~\eqref{eq:fusion}
by fusing regions \(\shull(\lambda'_{i}(c_{i}))\). We can express \(\lambda_{n}(c_{n})\) as a union
of possibly overlapping regions:
\[
  L_{i} = \shull\bigl(\lambda'_{n_{i}}(c^{i}_{n_{i}})\bigr) + \rho(c_{n},c^{i}_{n_{i}}),\quad 1 \le
  i \le m,
\]
where one of the points \(c^{i}_{n_{i}}\) coincides with \(c_{n}\). Note that
unlike~\eqref{eq:fusion}, where added regions stem from the earlier maps \(\lambda_{n_{i}}\), the
representation \(\lambda_{n}(c_{n}) = \bigcup_{i=1}^{m}L_{i}\) directly uses the regions that are
being fused. This follows inductively, since all regions \(\lambda_{k}\) for \(k < n\) are
themselves unions of fused regions.

Without loss of generality, we may assume no \(L_{i}\) is entirely contained in another \(L_{j}\)
(since removing redundant regions preserves the union \(\bigcup_{i} L_{i} = \lambda_{n}(c_{n})\)).
Because slice hulls commute with translations, we have
\[
  \shull\bigl(\lambda'_{n_{i}}(c^{i}_{n_{i}})\bigr) + \rho(c_{n},c^{i}_{n_{i}}) =
  \shull\bigl(\lambda'_{n_{i}}(c^{i}_{n_{i}}) + \rho(c_{n},c^{i}_{n_{i}})\bigr).
\]
Let \(K_{i} = \lambda'_{n_{i}}(c^{i}_{n_{i}}) + \rho(c_{n},c^{i}_{n_{i}})\) and note that
\(K_{i}c_{n} = R'_{n_{i}}(c^{i}_{n_{i}})\).

We claim the sets \(K_{i}\) are pairwise disjoint. If \(K_{i}\) and \(K_{j}\) intersect, then
\(R'_{n_{i}}(c^{i}_{n_{i}})\) and \(R'_{n_{j}}(c^{j}_{n_{j}})\) must overlap. By the coherence
property of the toast structure \((\mathcal{C}_{n}, \lambda'_{n})_{n}\), one region must contain the
other:
\[
  R'_{n_{i}}(c^{i}_{n_{i}}) \subseteq R'_{n_{j}}(c^{j}_{n_{j}}) \quad \text{or} \quad
  R'_{n_{j}}(c^{j}_{n_{j}}) \subseteq R'_{n_{i}}(c^{i}_{n_{i}}).
\]
Assuming the former holds and using freeness of the action, we deduce
\(K_{i} \subseteq K_{j}\). Applying the monotonicity of the slice hull yields:
\[
  L_{i} = \shull(K_{i}) \subseteq \shull(K_{j}) = L_{j},
\]
contradicting our non-containment assumption. Therefore, the sets \(K_{i}\) must be disjoint.

We now invoke Proposition~\ref{prop:slice-hull-additivity} to establish the equality
\[\shull\Bigl(\bigcup_{i}K_{i}\Bigr) = \bigcup_{i}\shull(K_{i}) = \bigcup_{i}L_{i} =
  \lambda_{n}(c_{n}).\]
Applying the \(\shull\) operation to both sides and utilizing its idempotence yields
\(\shull(\lambda_{n}(c_{n})) = \lambda_{n}(c_{n})\). This demonstrates that
\(\lambda_{n}(c_{n}) \in \mfR\) for all \(c_{n} \in \mathcal{C}_{n}\).

The preceding discussion leads to the following result.

\begin{theorem}
\label{thm:existence-of-slice-filled-toasts}
Let \(\mfR \subseteq \vietoris{\R^{d+1}}\) denote the class of slice-wise filled compact sets. Every
free Borel \(\R^{d+1}\)-flow admits a Borel \(\mfR\)-toast.
\end{theorem}

\bibliography{refs.bib}

@article {MR4041104,
    AUTHOR = {Buhovsky, Lev and Gl\"{u}cksam, Adi and Logunov, Alexander and
              Sodin, Mikhail},
     TITLE = {Translation-invariant probability measures on entire
              functions},
   JOURNAL = {J. Anal. Math.},
  FJOURNAL = {Journal d'Analyse Math\'{e}matique},
    VOLUME = {139},
      YEAR = {2019},
     PAGES = {307--339},
      ISSN = {0021-7670,1565-8538},
   MRCLASS = {30D15 (60A10)},
  MRNUMBER = {4041104},
MRREVIEWER = {A.\ Yu.\ Rashkovski\u{\i}},
       DOI = {10.1007/s11854-019-0067-x},
       URL = {https://doi.org/10.1007/s11854-019-0067-x},
}

@book {MR1070713,
    AUTHOR = {Conway, John B.},
     TITLE = {A course in functional analysis},
    SERIES = {Graduate Texts in Mathematics},
    VOLUME = {96},
   EDITION = {Second edition},
 PUBLISHER = {Springer-Verlag, New York},
      YEAR = {1990},
     PAGES = {xvi+399},
      ISBN = {0-387-97245-5},
   MRCLASS = {46-01 (47-01)},
  MRNUMBER = {1070713},
}

@article{Nevanlinna1925,
AUTHOR={Rolf Nevanlinna},
TITLE={{\"U}ber eine {E}rweiterung des {P}oissonschen {I}ntegrals},
VOLUME={24},
JOURNAL={Ann. Soc. Sci. Fenn., Ser A},
YEAR={1925},
PAGES={1--15},
}

@article{Arakelyan,
    TITLE={The {D}irichlet and {N}eumann problems for harmonic functions},
    AUTHOR={Arakelyan, N.},
    JOURNAL={J. Contemp. Math. Anal.},
    VOLUME={43},
    ISSUE={6},
    PAGES={21--38},
    YEAR={2008},
}

@book {GardinerHarmonicApprox,
    AUTHOR = {Gardiner, Stephen J.},
     TITLE = {Harmonic approximation},
    SERIES = {London Mathematical Society Lecture Note Series},
    VOLUME = {221},
 PUBLISHER = {Cambridge University Press, Cambridge},
      YEAR = {1995},
     PAGES = {xiv+132},
      ISBN = {0-521-49799-X},
   MRCLASS = {31-02 (41-02)},
  MRNUMBER = {1342298},
MRREVIEWER = {P. Lappan},
       DOI = {10.1017/CBO9780511526220},
       URL = {https://doi-org.kuleuven.e-bronnen.be/10.1017/CBO9780511526220},
}

@misc{BHM,
      title={Hyperbolic {F}ourier series and the {K}lein--{G}ordon equation}, 
      author={Bakan, Andrew and Hedenmalm, Haakan and Montes-Rodr\'{\i}guez, Alfonso},
      publisher = {arXiv:2401.06871},
      archivePrefix={arXiv},
      primaryClass={math.AP},
      url={https://arxiv.org/abs/2401.06871}, 
}

@misc{GlucksamWeiss,
title={Measurable entire functions {II}},
author={Adi Glücksam and Benjamin Weiss},
publisher={arXiv:2507.13182},
archivePrefix={arXiv},
primaryClass={math.DS},
url={https://arxiv.org/abs/2507.13182},
}

@article {Gardiner81,
    AUTHOR = {Gardiner, S. J.},
     TITLE = {The {D}irichlet and {N}eumann problems for harmonic functions
              in half-spaces},
   JOURNAL = {J. London Math. Soc. (2)},
  FJOURNAL = {Journal of the London Mathematical Society. Second Series},
    VOLUME = {24},
      YEAR = {1981},
    NUMBER = {3},
     PAGES = {502--512},
      ISSN = {0024-6107},
   MRCLASS = {31B20},
  MRNUMBER = {635881},
MRREVIEWER = {\"{U}. Kuran},
       DOI = {10.1112/jlms/s2-24.3.502},
       URL = {https://doi-org.kuleuven.e-bronnen.be/10.1112/jlms/s2-24.3.502},
}

@article {FinkelsteinScheinberg,
    AUTHOR = {Finkelstein, Mark and Scheinberg, Stephen},
     TITLE = {Kernels for solving problems of {D}irichlet type in a
              half-plane},
   JOURNAL = {Advances in Math.},
  FJOURNAL = {Advances in Mathematics},
    VOLUME = {18},
      YEAR = {1975},
    NUMBER = {1},
     PAGES = {108--113},
      ISSN = {0001-8708},
   MRCLASS = {31A25},
  MRNUMBER = {382677},
MRREVIEWER = {David R. Adams},
       DOI = {10.1016/0001-8708(75)90004-3},
       URL = {https://doi-org.kuleuven.e-bronnen.be/10.1016/0001-8708(75)90004-3},
}

@article {MR4025019,
    AUTHOR = {Gauthier, Paul M. and Ransford, Thomas and St-Amant, Simon and
              Turcotte, J\'{e}r\'{e}mie},
     TITLE = {Approximation by random complex polynomials and random
              rational functions},
   JOURNAL = {Ann. Polon. Math.},
  FJOURNAL = {Annales Polonici Mathematici},
    VOLUME = {123},
      YEAR = {2019},
     PAGES = {267--294},
      ISSN = {0066-2216},
   MRCLASS = {32E20 (28A20 30E10 30H50 60G99)},
  MRNUMBER = {4025019},
MRREVIEWER = {Norman Levenberg},
       DOI = {10.4064/ap180912-20-2},
       URL = {https://doi-org.kuleuven.e-bronnen.be/10.4064/ap180912-20-2},
}

@article {MR994977,
    AUTHOR = {Brown, Leon and Schreiber, Bertram M.},
     TITLE = {Approximation and extension of random functions},
   JOURNAL = {Monatsh. Math.},
  FJOURNAL = {Monatshefte f\"{u}r Mathematik},
    VOLUME = {107},
      YEAR = {1989},
     PAGES = {111--123},
      ISSN = {0026-9255},
   MRCLASS = {30E10 (41A10 60G99)},
  MRNUMBER = {994977},
MRREVIEWER = {M. Sambandham},
       DOI = {10.1007/BF01300917},
       URL = {https://doi-org.kuleuven.e-bronnen.be/10.1007/BF01300917},
}

@book {MR1619545,
    AUTHOR = {Srivastava, S. M.},
     TITLE = {A course on {B}orel sets},
    SERIES = {Graduate Texts in Mathematics},
    VOLUME = {180},
 PUBLISHER = {Springer-Verlag, New York},
      YEAR = {1998},
     PAGES = {xvi+261},
      ISBN = {0-387-98412-7},
   MRCLASS = {04A15 (28A05 54-01 54H05)},
  MRNUMBER = {1619545},
MRREVIEWER = {Marek Balcerzak},
       DOI = {10.1007/978-3-642-85473-6},
       URL = {https://doi-org.kuleuven.e-bronnen.be/10.1007/978-3-642-85473-6},
}

@article {MR236341,
    AUTHOR = {Himmelberg, C. J. and Van Vleck, F. S.},
     TITLE = {Some selection theorems for measurable functions},
   JOURNAL = {Canadian J. Math.},
  FJOURNAL = {Canadian Journal of Mathematics. Journal Canadien de
              Math\'{e}matiques},
    VOLUME = {21},
      YEAR = {1969},
     PAGES = {394--399},
      ISSN = {0008-414X},
   MRCLASS = {28.20},
  MRNUMBER = {236341},
MRREVIEWER = {F. J. Almgren, Jr.},
       DOI = {10.4153/CJM-1969-041-7},
       URL = {https://doi-org.kuleuven.e-bronnen.be/10.4153/CJM-1969-041-7},
}

@article {MR4658613,
    AUTHOR = {Sodin, Mikhail and Wennman, Aron and Yakir, Oren},
     TITLE = {The random {W}eierstrass zeta function {I}: {E}xistence,
              uniqueness, fluctuations},
   JOURNAL = {J. Stat. Phys.},
  FJOURNAL = {Journal of Statistical Physics},
    VOLUME = {190},
      YEAR = {2023},
     PAGES = {Paper No. 166},
      ISSN = {0022-4715},
   MRCLASS = {60G55 (28A75 60G10 78A30)},
  MRNUMBER = {4658613},
MRREVIEWER = {Brian Fralix},
       DOI = {10.1007/s10955-023-03169-5},
       URL = {https://doi-org.focus.lib.kth.se/10.1007/s10955-023-03169-5},
}

@misc{tsirelson2007,
      title={Divergence of a stationary random vector field can be always positive (a {W}eiss' phenomenon)}, 
      author={Boris Tsirelson},
      publisher = {arXiv:0801.1050},
      archivePrefix={arXiv},
      primaryClass={math.PR},
      url={https://arxiv.org/abs/0709.1270}, 
}

@article {MR3905607,
    AUTHOR = {Gl\"{u}cksam, Adi},
     TITLE = {Measurably entire functions and their growth},
   JOURNAL = {Israel J. Math.},
  FJOURNAL = {Israel Journal of Mathematics},
    VOLUME = {229},
      YEAR = {2019},
     PAGES = {307--339},
      ISSN = {0021-2172,1565-8511},
   MRCLASS = {37C35 (28D05 30D15 30D35)},
  MRNUMBER = {3905607},
MRREVIEWER = {Jorge\ Buescu},
       DOI = {10.1007/s11856-018-1800-3},
       URL = {https://doi.org/10.1007/s11856-018-1800-3},
}

@book {Kechris,
    AUTHOR = {Kechris, Alexander S.},
     TITLE = {Classical descriptive set theory},
    SERIES = {Graduate Texts in Mathematics},
    VOLUME = {156},
 PUBLISHER = {Springer-Verlag, New York},
      YEAR = {1995},
     PAGES = {xviii+402},
      ISBN = {0-387-94374-9},
   MRCLASS = {03E15 (03-01 03-02 04A15 28A05 54H05 90D44)},
  MRNUMBER = {1321597},
MRREVIEWER = {Jakub Jasi\'{n}ski},
       DOI = {10.1007/978-1-4612-4190-4},
       URL = {https://doi.org/10.1007/978-1-4612-4190-4},
}

@article {Slutsky,
    AUTHOR = {Slutsky, Konstantin},
     TITLE = {On time change equivalence of {B}orel flows},
   JOURNAL = {Fund. Math.},
  FJOURNAL = {Fundamenta Mathematicae},
    VOLUME = {247},
      YEAR = {2019},
    NUMBER = {1},
     PAGES = {1--24},
      ISSN = {0016-2736},
   MRCLASS = {37A20 (54H05)},
  MRNUMBER = {3984276},
MRREVIEWER = {Lewis Bowen},
       DOI = {10.4064/fm677-12-2018},
       URL = {https://doi.org/10.4064/fm677-12-2018},
}

@incollection {MR1422707,
    AUTHOR = {Weiss, Benjamin},
     TITLE = {Measurable entire functions},
      NOTE = {The heritage of P. L. Chebyshev: a Festschrift in honor of the
              70th birthday of T. J. Rivlin},
   JOURNAL = {Ann. Numer. Math.},
  FJOURNAL = {Annals of Numerical Mathematics},
    VOLUME = {4},
      YEAR = {1997},
    NUMBER = {1-4},
     PAGES = {599--605},
      ISSN = {1021-2655},
   MRCLASS = {28D99},
  MRNUMBER = {1422707},
}

@MISC {201919,
    TITLE = {Continuous {W}eierstrass map},
    AUTHOR = {Remling, Christian},
    YEAR = {2015},
    NOTE = {URL:https://mathoverflow.net/q/201919 (version: 2015-04-04)},
    EPRINT = {https://mathoverflow.net/q/201919},
    URL = {https://mathoverflow.net/q/201919}
}

@article {MR4215747,
    AUTHOR = {Slutsky, Konstantin},
     TITLE = {Smooth orbit equivalence of multidimensional {B}orel flows},
   JOURNAL = {Adv. Math.},
  FJOURNAL = {Advances in Mathematics},
    VOLUME = {381},
      YEAR = {2021},
     PAGES = {Paper No. 107626},
      ISSN = {0001-8708,1090-2082},
   MRCLASS = {37A20 (37C15)},
  MRNUMBER = {4215747},
MRREVIEWER = {Lewis\ Bowen},
       DOI = {10.1016/j.aim.2021.107626},
       URL = {https://doi.org/10.1016/j.aim.2021.107626},
}

@book{kallenbergRandomMeasuresTheory2017,
  title = {Random Measures, Theory and Applications},
  author = {Kallenberg, Olav},
  year = {2017},
  series = {Probability {{Theory}} and {{Stochastic Modelling}}},
  volume = {77},
  publisher = {Springer, Cham},
  doi = {10.1007/978-3-319-41598-7},
  isbn = {978-3-319-41596-3 978-3-319-41598-7},
  mrnumber = {3642325},
}

@article{morariu-patrichiWeakHashMetricBoundednly2018,
  title = {{{On the weak-hash metric for boundedly finite
                  integer-valued measures}}},
  author = {{Morariu-Patrichi}, Maxime},
  year = {2018},
  journal = {Bull. Aust. Math. Soc.},
  journal = {Bulletin of the Australian Mathematical Society},
  volume = {98},
  pages = {265--276},
  issn = {0004-9727, 1755-1633},
  doi = {10.1017/S0004972718000485},
  urldate = {2024-03-16},
  langid = {english},
}

@book{daleyIntroductionTheoryPoint2003,
  title = {An Introduction to the Theory of Point Processes. {{Vol}}. {{I}}},
  author = {Daley, D. J. and {Vere-Jones}, D.},
  year = {2003},
  series = {Probability and Its {{Applications}} ({{New York}})},
  edition = {Second edition},
  publisher = {Springer-Verlag, New York},
  isbn = {978-0-387-95541-4},
  mrnumber = {1950431},
}

@book{hormanderAnalysisLinearPartial2003,
  title = {The Analysis of Linear Partial Differential Operators. {{I}}},
  author = {H{\"o}rmander, Lars},
  year = {2003},
  series = {Classics in {{Mathematics}}},
  publisher = {Springer-Verlag, Berlin},
  doi = {10.1007/978-3-642-61497-2},
  isbn = {978-3-540-00662-6},
  mrnumber = {1996773},
}

@book {MR1483074,
    AUTHOR = {Remmert, Reinhold},
     TITLE = {Classical topics in complex function theory},
    SERIES = {Graduate Texts in Mathematics},
    VOLUME = {172},
 PUBLISHER = {Springer-Verlag, New York},
      YEAR = {1998},
     PAGES = {xx+349},
      ISBN = {0-387-98221-3},
   MRCLASS = {30-02},
  MRNUMBER = {1483074},
       DOI = {10.1007/978-1-4757-2956-6},
       URL = {https://doi-org.kuleuven.e-bronnen.be/10.1007/978-1-4757-2956-6},
}

@book{armitageClassicalPotentialTheory2001,
  title = {Classical Potential Theory},
  author = {Armitage, David H. and Gardiner, Stephen J.},
  year = {2001},
  series = {Springer {{Monographs}} in {{Mathematics}}},
  publisher = {Springer-Verlag London, Ltd., London},
  doi = {10.1007/978-1-4471-0233-5},
  isbn = {978-1-85233-618-9},
  mrnumber = {1801253},
}

@article{gaoForcingConstructionsCountable2022,
  title = {Forcing Constructions and Countable {{Borel}} Equivalence Relations},
  author = {Gao, Su and Jackson, Steve and Krohne, Edward and Seward, Brandon},
  year = {2022},
  journal = {J. Symb. Logic},
  fjournal = {The Journal of Symbolic Logic},
  volume = {87},
  pages = {873--893},
  issn = {0022-4812},
  doi = {10.1017/jsl.2022.23},
  urldate = {2023-01-13},
  mrnumber = {4472517},
}

@book{beckerDescriptiveSetTheory1996,
  title = {The Descriptive Set Theory of {{Polish}} Group Actions},
  author = {Becker, Howard and Kechris, Alexander S.},
  year = {1996},
  series = {London {{Mathematical Society Lecture Note Series}}},
  volume = {232},
  publisher = {Cambridge University Press},
  address = {Cambridge},
  urldate = {2011-10-08},
  isbn = {0-521-57605-9},
}

@book{nariciTopologicalVectorSpaces2010,
  title = {Topological {{Vector Spaces}}},
  author = {Narici, Lawrence and Beckenstein, Edward},
  year = {2010},
  edition = {2nd edition},
  publisher = {{Chapman and Hall/CRC}},
  address = {Boca Raton, FL},
  isbn = {978-1-58488-866-6},
  langid = {english}
}

@article{mankiewiczTopologicalLipschitzUniform1974,
  title = {On Topological, {{Lipschitz}}, and Uniform Classification of {{LF-spaces}}},
  author = {Mankiewicz, P.},
  year = {1974},
  journal = {Studia Math.},
  fjournal = {Polska Akademia Nauk. Instytut Matematyczny. Studia Mathematica},
  volume = {52},
  pages = {109--142},
  issn = {0039-3223,1730-6337},
  doi = {10.4064/sm-52-2-109-142},
  mrnumber = {402448}
}

@article{vidossichCharacterizationSeparabilityLF1968,
  title = {Characterization of Separability for ${{LF}}$-Spaces},
  author = {Vidossich, Giovanni},
  year = {1968},
  journal = {Ann. Inst. Fourier},
  fjournal = {Annales de l'institut Fourier},
  volume = {18},
  pages = {87--90},
  publisher = {Association des Annales de l'Institut Fourier},
  issn = {0373-0956},
  urldate = {2024-08-11},
  langid = {english}
}

@book{hormanderIntroductionComplexAnalysis1990,
  title = {An Introduction to Complex Analysis in Several Variables},
  author = {H{\"o}rmander, Lars},
  year = {1990},
  series = {North-{{Holland Mathematical Library}}},
  edition = {Third edition},
  volume = {7},
  publisher = {North-Holland Publishing Co., Amsterdam},
  isbn = {978-0-444-88446-6},
  mrnumber = {1045639}
}

@misc{KechrisMarks,
    TITLE = {Descriptive Graph Combinatorics, preliminary version},
    AUTHOR = {Kechris, Alexander S. and Marks, Andrew},
    YEAR = {2020},
    eprint = {\url{https://www.math.ucla.edu/~marks/papers/combinatorics20book.pdf}},
    URL = {https://www.math.ucla.edu/marks/papers/combinatorics20book.pdf},
    NOTE = {Available at \url{https://www.math.ucla.edu/marks/papers/combinatorics20book.pdf}}
}

@article{kolbasinaPropertyDiscreteSets2008,
  title = {{On a property of discrete sets in \(\mathbb{R}^k\)}},
  author = {Kolbasina, {\relax Ye}. {\relax Yu}.},
  year = {2008},
  journal = {Visn. Khark. Univ. Ser. Mat. Prykl. Mat. Mekh.},
  fjournal = {Visnyk Kharkivs'kogo Universytetu. Seriya Matematyka, Prykladna Matematyka i Mekhanika},
  volume = {826},
  number = {58},
  pages = {52--66},
  issn = {2221-5646},
  langid = {russian}
}

@article {favorovAlmostPeriodicDiscrete2010a,
    AUTHOR = {Favorov, S. and Kolbasina, Ye.},
     TITLE = {Almost periodic discrete sets},
   JOURNAL = {J. Math. Phys. Anal. Geom.},
  FJOURNAL = {Zhurnal Matematichesko\u\i{} Fiziki, Analiza, Geometrii. Journal of Mathematical Physics, Analysis, Geometry},
    VOLUME = {6},
      YEAR = {2010},
     PAGES = {34--47, 135},
      ISSN = {1812-9471,1817-5805},
   MRCLASS = {42A75 (11K70 28A12 43A60)},
  MRNUMBER = {2655763},
MRREVIEWER = {Hui-Sheng\ Ding},
}

@book{conwayFunctionsOneComplex1978,
  title = {Functions of One Complex Variable},
  author = {Conway, John B.},
  year = {1978},
  series = {Graduate {{Texts}} in {{Mathematics}}},
  edition = {Second},
  volume = {11},
  publisher = {Springer-Verlag, New York-Berlin},
  isbn = {978-0-387-90328-6},
  mrnumber = {503901}
}

@book{trevesTopologicalVectorSpaces2006,
  title = {Topological Vector Spaces, Distributions and Kernels},
  author = {Tr{\`e}ves, Fran{\c c}ois},
  year = {2006},
  publisher = {Dover Publications, Inc., Mineola, NY},
  isbn = {978-0-486-45352-1},
  mrnumber = {2296978}
}

@book{schwartzRadonMeasuresArbitrary1973,
  title = {Radon Measures on Arbitrary Topological Spaces and Cylindrical Measures},
  author = {Schwartz, Laurent},
  year = {1973},
  series = {Tata {{Institute}} of {{Fundamental Research Studies}} in {{Mathematics}}},
  volume = {No. 6},
  publisher = {Tata Institute of Fundamental Research, Bombay; by Oxford University Press, London},
  mrnumber = {426084},
}

@article {dudleyContinuityHomomorphisms1961,
    AUTHOR = {Dudley, R. M.},
     TITLE = {Continuity of homomorphisms},
   JOURNAL = {Duke Math. J.},
  FJOURNAL = {Duke Mathematical Journal},
    VOLUME = {28},
      YEAR = {1961},
     PAGES = {587--594},
      ISSN = {0012-7094,1547-7398},
   MRCLASS = {22.10},
  MRNUMBER = {136676},
MRREVIEWER = {H.\ B.\ Griffiths},
       URL = {http://projecteuclid.org/euclid.dmj/1077469840},
}

@article {grosse-erdmannLocallyConvexTopology1995,
    AUTHOR = {Grosse-Erdmann, Karl-Goswin},
     TITLE = {The locally convex topology on the space of meromorphic
              functions},
   JOURNAL = {J. Austral. Math. Soc. Ser. A},
  FJOURNAL = {Australian Mathematical Society. Journal. Series A. Pure
              Mathematics and Statistics},
    VOLUME = {59},
      YEAR = {1995},
     PAGES = {287--303},
      ISSN = {0263-6115},
   MRCLASS = {46E10 (30D30 30H05)},
  MRNUMBER = {1355220},
MRREVIEWER = {Raymond\ Mortini},
}

@article {arensLinearTopologicalDivision1947,
    AUTHOR = {Arens, Richard},
     TITLE = {Linear topological division algebras},
   JOURNAL = {Bull. Amer. Math. Soc.},
  FJOURNAL = {Bulletin of the American Mathematical Society},
    VOLUME = {53},
      YEAR = {1947},
     PAGES = {623--630},
      ISSN = {0002-9904},
   MRCLASS = {09.1X},
  MRNUMBER = {20987},
MRREVIEWER = {O.\ Taussky-Todd},
       DOI = {10.1090/S0002-9904-1947-08857-1},
       URL = {https://doi.org/10.1090/S0002-9904-1947-08857-1},
}

@article{cimaSpacesMeromorphicFunctions1979,
  title = {On Spaces of Meromorphic Functions},
  author = {Cima, J. and Schober, G.},
  year = {1979},
  JOURNAL = {Rocky Mountain J. Math.},
  fjournal = {The Rocky Mountain Journal of Mathematics},
  volume = {9},
  pages = {527--532},
  issn = {0035-7596,1945-3795},
  doi = {10.1216/RMJ-1979-9-3-527},
  mrnumber = {528750},
}

@article{lohmanSeparabilityLinearTopological1974,
  title = {On Separability in Linear Topological Spaces},
  author = {Lohman, Robert H. and Stiles, Wilbur J.},
  year = {1974},
  journal = {Proceedings of the American Mathematical Society},
  volume = {42},
  number = {1},
  pages = {236--237},
  issn = {0002-9939, 1088-6826},
  doi = {10.1090/S0002-9939-1974-0326350-X},
  urldate = {2024-08-11},
  langid = {english},
}

@book {rudinRealComplexAnalysis1987,
    AUTHOR = {Rudin, Walter},
     TITLE = {Real and complex analysis},
   EDITION = {Third edition},
 PUBLISHER = {McGraw-Hill Book Co., New York},
      YEAR = {1987},
     PAGES = {xiv+416},
      ISBN = {0-07-054234-1},
   MRCLASS = {00A05 (26-01 30-01 46-01)},
  MRNUMBER = {924157},
}

@article {jonesApproximationTheoremRunge1975,
    AUTHOR = {Jones, Jr., B. Frank},
     TITLE = {An approximation theorem of {R}unge type for the heat
              equation},
   JOURNAL = {Proc. Amer. Math. Soc.},
  FJOURNAL = {Proceedings of the American Mathematical Society},
    VOLUME = {52},
      YEAR = {1975},
     PAGES = {289--292},
      ISSN = {0002-9939,1088-6826},
   MRCLASS = {35K05},
  MRNUMBER = {387815},
MRREVIEWER = {J.\ R.\ Cannon},
       DOI = {10.2307/2040146},
       URL = {https://doi.org/10.2307/2040146},
}

@article {diazRungeTheoremSolutions1980,
    AUTHOR = {Diaz, R.},
     TITLE = {A {R}unge theorem for solutions of the heat equation},
   JOURNAL = {Proc. Amer. Math. Soc.},
  FJOURNAL = {Proceedings of the American Mathematical Society},
    VOLUME = {80},
      YEAR = {1980},
     PAGES = {643--646},
      ISSN = {0002-9939,1088-6826},
   MRCLASS = {35K05 (31B35 35E20)},
  MRNUMBER = {587944},
MRREVIEWER = {I.\ \L ojczyk-Kr\'olikiewicz},
       DOI = {10.2307/2043440},
       URL = {https://doi.org/10.2307/2043440},
}

@book{haymanSubharmonicFunctionsVol1976,
  title = {Subharmonic Functions. {{Vol}}. {{I}}},
  author = {Hayman, W. K. and Kennedy, P. B.},
  year = {1976},
  series = {London {{Mathematical Society Monographs}}},
  volume = {No. 9},
  publisher = {Academic Press, London-New York},
  mrnumber = {460672}
}

@article{nelsonProofLiouvillesTheorem1961,
  title = {A Proof of {{Liouville}}'s Theorem},
  author = {Nelson, Edward},
  year = {1961},
  journal = {Proceedings of the American Mathematical Society},
  volume = {12},
  pages = {995},
  issn = {0002-9939,1088-6826},
  doi = {10.2307/2034412},
  mrnumber = {259149}
}

@book{narasimhanAnalysisRealComplex1985,
  title = {Analysis on Real and Complex Manifolds},
  author = {Narasimhan, R.},
  year = {1985},
  series = {North-{{Holland Mathematical Library}}},
  volume = {35},
  publisher = {North-Holland Publishing Co., Amsterdam},
  isbn = {978-0-444-87776-5},
  mrnumber = {832683}
}

@article{bagbyUniformHarmonicApproximation1994,
  title = {Uniform Harmonic Approximation on {{Riemannian}} Manifolds},
  author = {Bagby, Thomas and Blanchet, Pierre},
  year = {1994},
  JOURNAL = {J. Anal. Math.},
  FJOURNAL = {Journal d'Analyse Math\'{e}matique},
  volume = {62},
  pages = {47--76},
  issn = {0021-7670,1565-8538},
  doi = {10.1007/BF02835948},
  mrnumber = {1269199}
}

@misc{2405.15111,
Author = {Alexander S. Kechris and Michael Wolman},
Title = {Invariant uniformization},
publisher = {arXiv:2405.15111},
Eprint = {arXiv:2405.15111},
}

@article {kechrisCountableSectionsLocally1992,
    AUTHOR = {Kechris, Alexander S.},
     TITLE = {Countable sections for locally compact group actions},
   JOURNAL = {Ergodic Theory Dynam. Systems},
  FJOURNAL = {Ergodic Theory and Dynamical Systems},
    VOLUME = {12},
      YEAR = {1992},
     PAGES = {283--295},
      ISSN = {0143-3857},
   MRCLASS = {22D40 (03E15 54H11)},
  MRNUMBER = {1176624},
MRREVIEWER = {Ivan L. Reilly},
       DOI = {10.1017/S0143385700006751},
       URL = {https://doi-org.kuleuven.e-bronnen.be/10.1017/S0143385700006751},
}

@article {marksBorelCircleSquaring2017,
    AUTHOR = {Marks, Andrew S. and Unger, Spencer T.},
     TITLE = {Borel circle squaring},
   JOURNAL = {Ann. of Math. (2)},
  FJOURNAL = {Annals of Mathematics. Second Series},
    VOLUME = {186},
      YEAR = {2017},
     PAGES = {581--605},
      ISSN = {0003-486X},
   MRCLASS = {03E15 (05C21 37A20 52B45)},
  MRNUMBER = {3702673},
MRREVIEWER = {Marek Balcerzak},
       DOI = {10.4007/annals.2017.186.2.4},
       URL = {https://doi-org.kuleuven.e-bronnen.be/10.4007/annals.2017.186.2.4},
}

@BOOK{Billingsley1995-wf,
  title     = "Probability and Measure",
  author    = "Billingsley, Patrick",
  publisher = "John Wiley \& Sons",
  series    = "Wiley Series in Probability \& Mathematical Statistics:
               Probability \& Mathematical Statistics",
  edition   =  3,
  month     =  may,
  year      =  1995,
  address   = "Nashville, TN",
}

@BOOK{Hormander2004-vz,
  title     = "The analysis of linear partial differential operators {II}",
  author    = "H{\"o}rmander, Lars",
  publisher = "Springer",
  series    = "Classics in Mathematics",
  edition   =  2005,
  month     =  nov,
  year      =  2004,
  address   = "Berlin",
}

@article {MR2403439,
    AUTHOR = {Favorov, S. Ju.},
     TITLE = {Sunyer-i-{B}alaguer's almost elliptic functions and {Y}osida's
              normal functions},
   JOURNAL = {J. Anal. Math.},
  FJOURNAL = {Journal d'Analyse Math\'{e}matique},
    VOLUME = {104},
      YEAR = {2008},
     PAGES = {307--340},
      ISSN = {0021-7670},
   MRCLASS = {30D30 (30D45)},
  MRNUMBER = {2403439},
MRREVIEWER = {L. R. Sons},
       DOI = {10.1007/s11854-008-0026-4},
       URL = {https://doi-org.kuleuven.e-bronnen.be/10.1007/s11854-008-0026-4},
}

@article{jacksonCountableBorelEquivalence2002,
  title = {Countable {{Borel}} Equivalence Relations},
  author = {Jackson, S. and Kechris, A. S. and Louveau, A.},
  year = {2002},
  journal = {J. Math. Log.},
  fjournal = {Journal of Mathematical Logic},
  volume = {2},
  pages = {1--80},
  issn = {0219-0613},
  doi = {10.1142/S0219061302000138},
  urldate = {2011-10-06}
}

@incollection{boykinBorelBoundednessLattice2007,
  title = {Borel Boundedness and the Lattice Rounding Property},
  booktitle = {Contemporary {{Mathematics}}},
  author = {Boykin, Charles M. and Jackson, Steve},
  editor = {Gao, Su and Jackson, Steve and Zhang, Yi},
  year = {2007},
  volume = {425},
  pages = {113--126},
  publisher = {American Mathematical Society},
  address = {Providence, Rhode Island},
  doi = {10.1090/conm/425/08121},
  urldate = {2020-03-05},
  isbn = {978-0-8218-3819-8 978-0-8218-8104-0},
}

@misc{gaoBorelCombinatoricsAbelian2024,
  title = {Borel {{Combinatorics}} of {{Abelian Group Actions}}},
  author = {Gao, Su and Jackson, Steve and Krohne, Edward and Seward, Brandon},
  number = {arXiv:2401.13866},
  eprint = {2401.13866},
  primaryclass = {math},
  publisher = {arXiv:2401.13866},
  doi = {10.48550/arXiv.2401.13866},
  urldate = {2025-07-14},
  archiveprefix = {arXiv}
}

@misc{gaoContinuousCombinatoricsAbelian2023,
  title = {Continuous {{Combinatorics}} of {{Abelian Group Actions}}},
  author = {Gao, Su and Jackson, Steve and Krohne, Edward and Seward, Brandon},
  month = apr,
  number = {arXiv:1803.03872},
  eprint = {1803.03872},
  primaryclass = {math},
  publisher = {arXiv:1803.03872},
  doi = {10.48550/arXiv.1803.03872},
  urldate = {2025-07-14},
  archiveprefix = {arXiv}
}

@article{ramsayMeasurableGroupActions1985,
  title = {Measurable {{Group Actions Are Essentially Borel Actions}}},
  author = {Ramsay, Arlan},
  year = 1985,
  journal = {Israel Journal of Mathematics},
  volume = {51},
  number = {4},
  pages = {339--346},
  issn = {0021-2172},
  doi = {10.1007/BF02764724},
  mrnumber = {804490}
}

@article {BergerEtAlCMP,
    AUTHOR = {Berger, Pierre and Florio, Anna and Peralta-Salas, Daniel},
     TITLE = {Steady {E}uler flows on {$\mathbb{R}^3$} with wild and universal
              dynamics},
   JOURNAL = {Comm. Math. Phys.},
  FJOURNAL = {Communications in Mathematical Physics},
    VOLUME = {401},
      YEAR = {2023},
    NUMBER = {1},
     PAGES = {937--983},
      ISSN = {0010-3616,1432-0916},
   MRCLASS = {35Q31 (76F02)},
  MRNUMBER = {4604911},
MRREVIEWER = {Lucio\ Galeati},
       DOI = {10.1007/s00220-023-04660-6},
       URL = {https://doi-org.kuleuven.e-bronnen.be/10.1007/s00220-023-04660-6},
}

@book {VassilievIntroTopology,
    AUTHOR = {Vassiliev, V. A.},
     TITLE = {Introduction to topology},
    SERIES = {Student Mathematical Library},
    VOLUME = {14},
      NOTE = {Translated from the 1997 Russian original by A. Sossinski},
 PUBLISHER = {American Mathematical Society, Providence, RI},
      YEAR = {2001},
     PAGES = {xiv+149},
      ISBN = {0-8218-2162-8},
   MRCLASS = {55-02 (57-02)},
  MRNUMBER = {1816237},
       DOI = {10.1090/stml/014},
       URL = {https://doi-org.kuleuven.e-bronnen.be/10.1090/stml/014},
}
\bibliographystyle{plain}

\bigskip
\bigskip

\noindent {Konstantin Slutsky}: Department of Mathematics, Iowa State University, Ames, Iowa, USA
\newline {\tt kslutsky@iastate.edu}

\bigskip

\noindent {Mikhail Sodin}: School of Mathematics, Tel Aviv University, Tel Aviv, Israel \newline
{\tt sodin@tauex.tau.ac.il}

\bigskip

\noindent {Aron Wennman}: Department of Mathematics, KU Leuven, Leuven, Belgium \newline {\tt
  aron.wennman@kuleuven.be}

\end{document}